\documentclass[10pt,reqno]{amsart}
\usepackage[utf8]{inputenc}
\usepackage[T1]{fontenc}
\usepackage{lmodern}
\usepackage{xspace}
\usepackage{amsfonts,amsmath,amssymb,amsthm}
\usepackage[small]{titlesec}
\usepackage{comment}
\usepackage{mathrsfs}

\usepackage{fancyhdr}
\usepackage{xcolor}
\usepackage{upgreek}
\usepackage{frcursive}
\usepackage{mathrsfs}

\newcommand{\mysection}{\setcounter{equation}{0} \section}

\definecolor{Dgreen}{rgb}{0.05, 0.5, 0.06}
\usepackage[a4paper,body={160mm,235mm},centering]{geometry}
\usepackage{hyperref}

\newcommand{\brho}{\boldsymbol{\rho}}

\newcommand{\R}{\mathbb{R}}
\newcommand{\I}{\mathbb{I}}

\newcommand{\ind}{\mathbf{1}}

\newcommand{\E}{\mathbb{E}}

\newcommand{\p}{\partial}

\newcommand \A[1]{{\bf (#1)}}

\newcommand{\dtv}{d_{{\rm{TV}}}}
\newcommand{\mW}{\mathcal W}
\newcommand{\bmu}{\boldsymbol \mu}

\newcommand{\mF}{\mathscr{C}}

\newcommand{\cv}{c_{\mathbf Y}} 
\newcommand{\ch}{c_{\mathbf{HK}}} 
\newcommand{\ccv}{c_{\mathbf{cv}}} 
\def\ind{{~\hbox{\rm l\kern-.4em\hbox{\rm l}}}~} 
\newcommand{\Lw}{\text{L}_{\text{w}}} 

\def\div{{\rm{div}}}

\makeatletter
\newcounter{subhyp}
\let\savedc@hyp\c@hyp
\newcommand{\dotafter}[1]{#1.}
\titleformat{\section}[hang]
{\normalfont\large\bfseries}{\thesection.}{.5em}{\dotafter}[]
\titleformat{\subsection}[runin]
{\normalfont\bfseries}{\thesubsection.}{.4em}{}[.]
\titlespacing*{\subsection}{0pt}{3ex plus 1ex minus .2ex}{1em}
\titleformat{\paragraph}[runin]{\normalfont\bfseries}{\theparagraph.}{.4em}{}[.]
\theoremstyle{plain}
\newtheorem{thm}{Theorem}
\newtheorem{lemme}[thm]{Lemma}
\newtheorem{prop}[thm]{Proposition}

\theoremstyle{definition}

\theoremstyle{remark}
\newtheorem{rem}{Remark}

\renewcommand{\sc}{\mbox{\small{\begin{cursive}s\end{cursive}}}}
\newcommand{\cc}{\mbox{\small{\begin{cursive}c\end{cursive}}}}

\renewcommand{\|}{|}

\begin{document}
\title[]{Multidimensional stable driven McKean-Vlasov SDEs with distributional interaction kernel: a regularization by noise perspective}

\author{P.-E. Chaudru de Raynal}
\address{Laboratoire de Math\'ematiques Jean Leray, University of Nantes, 2, rue de la Houssini\`ere BP 92208
F-44322 Nantes Cedex 3, France.}
\email{pe.deraynal@univ-nantes.fr}
\author{J.-F. Jabir}
\address{Department of Statistics and Data Analysis $\&$ Laboratory of Stochastic Analysis, HSE University, Pokrovsky Blvd, 11, Moscow, Russian Federation.}
\email{jjabir@hse.ru}
\author{S. Menozzi}
\address{LaMME, Universit\'e d'Evry Val d'Essonne, Universit\'e Paris-Saclay, CNRS (UMR 8071), 23 Boulevard de France 91037 Evry, France $\&$ Laboratory of Stochastic Analysis, HSE University, Pokrovsky Blvd, 11, Moscow, Russian Federation.}
\email{stephane.menozzi@univ-evry.fr}
\date{\today}
\maketitle

\begin{abstract}
In this work, we are interested in establishing weak and strong well-posedness for McKean -Vlasov SDEs with additive stable noise and a convolution type non-linear drift with singular interaction kernel in the framework of Lebesgue-Besov spaces. \textcolor{black}{We prove that the well-posedness of the system holds for the thresholds (in terms of regularity indexes) deriving from the scaling of the noise and that the corresponding SDE can be understood in the \textit{classical} sense. Especially, we characterize quantitatively how the non-linearity allows to go beyond the stronger thresholds, coming from Bony's paraproduct rule, usually obtained for \textit{linear} SDEs with singular interaction kernels.} We also specifically characterize in function of the stability index of the driving noise and the parameters of the drift when the dichotomy between weak and strong uniqueness occurs. 

\end{abstract}

\\

\mysection{Introduction and main result}

\subsection{Framework}
Fix a \textcolor{black}{finite} time horizon $T>0$ and a starting time $t\in [0,T) $. Let $s\in (t,T]$ and  consider the, possibly \emph{formal}, McKean-Vlasov Stochastic Differential Equation (SDE)
\begin{equation}\label{main}
X_s^{t,\mu} = \xi + \int_t^s \int  b(v,X_v^{t,\mu}-y) \bmu_v^{t,\mu}(dy) dv + (\mW_s-\mW_t),\quad \xi \sim \mu \in \mathcal P(\R^d),
\end{equation}
with $\mathcal P(\R^d)$ the set of probability measures on $\R^d$, where the \textcolor{black}{non-degenerate} driving noise $(\mW_s)_{s\ge t} $ is a symmetric $\alpha $-stable process with $\alpha\in (1,2]$ (including thus the pure jump and the Brownian cases) and $(\bmu_{s}^{t,\mu})_{s\in [t,T]} $ denotes the 
{\color{black}{(time) marginal laws}} of $ (X_s^{t,\mu})_{s\in [t,T]}$. We aim at proving well-posedness (in a weak or strong sense) for the above equation for distributional interaction kernels $b$. {\color{black}Namely}, we want to consider a Lebesgue-Besov framework assuming that:
\begin{equation}\label{hyp_b}\tag{A}
b \in L^r((t,T],B_{p,q}^\beta(\R^d,\R^d))=:L^r(B_{p,q}^\beta),\qquad \beta \in (-1,0],\, p,q, r\in [1,+\infty].
\end{equation}
{\color{black}Briefly}, when $p = q = +\infty$, and
for any non integer $\beta > 0$, Besov spaces coincide with H\"older spaces {\color{black}with} $B_{\infty,\infty}^\beta(\R^d,\R^d) = C^\beta(\R^d,\R^d)$; when $ \beta \in (-1,0)$, 
{\color{black}they correspond} to a distribution space whose elements can be seen as the generalized derivatives of functions belonging to $B_{\infty,\infty}^{\beta+1}(\R^d,\R^d)=C^{\beta+1}(\R^d,\R^d)$. This also \textit{somehow} indicates that the H\"older modulus blows up at rate $ \beta$. 
{\color{black}More generally, t}he parameters $p$ and $q$ are related to the integrability of such a modulus. We refer to Section 2.6.4 of \cite{Triebel-83a} or Section \ref{SEC_BESOV} for a precise definition of Besov spaces.

{\color{black}\textcolor{black}{In} the setting \eqref{hyp_b}}, we are interested in deriving conditions relating the \emph{stable exponent} $\alpha$, the \emph{integrability indexes} $r,p,q$, the \emph{regularity index} $\beta$ and the dimension $d$ of the system in order to have well-posedness of the former equation in a weak or strong sense. In the following, the quantities $\alpha,r,p,q,\beta$ and $d$ are referred to as the \emph{parameters} (of \eqref{main}).

The question of the well-posedness for irregular or singular {\color{black}\emph{non-degenerate}}\footnote{We can refer to the recent work \cite{zhan:21} for McKean-Vlasov equations of kinetic type (\textcolor{black}{see also \cite{boss:font:jabi:jabi:13} or \cite{butt:fla:otto:zega:19})}} McKean-Vlasov SDEs of the above, or even more general, form is nowadays the subject of a vast and still increasing literature. Historically, McKean-Vlasov SDEs were introduced for the probabilistic interpretation of non-linear parabolic PDEs arising as the mean-field limit of interacting particle systems, \cite{McKean-66}, \cite{McKean-67}. At the time, H.~P. McKean addressed the specific cases of a Boltzmann type equation and the (viscous) Burgers equation, and since then -particularly in the 80s' and 90s'- McKean-Vlasov models have been applied to design and validate particle methods {\color{black}in a variety of situations} in statistical physics and in fluid dynamics.\\

McKean-Vlasov SDEs were notably investigated for the validation of particle approximation and simulation of vortex particle methods for the two and three dimensional incompressible Navier-Stokes equations,  
and more recently the Keller-Segel equations. 
For a more detailed overview on the links between McKean-Vlasov models, numerical particle methods and their application in statistical physics, population dynamics, social science and other engineering fields, we refer the interested reader to the surveys \cite{Bossy-05,JabinWang-17} and the recent monographs \cite{DelMoral-13,ChaDie-21}.

\textcolor{black}{In the last decade}, the theory \textcolor{black}{of} McKean-Vlasov SDEs has received a renew\textcolor{black}{ed} and significant interest with the emergence of mean field games and applications of \textcolor{black}{the} mean-field theory and McKean-Vlasov SDEs in stochastic control problems - see \cite{CarDel-18} and the references therein. This new framework, coupled with new advances on variational calculus on the space of probability measures (e.g. \cite{AmGiSa-08}, \cite{Kolokolstov-10}, \cite{Cardaliaguet-13} and \cite{CarDel-18}), \textcolor{black}{puts the focus on}  well-posedness  questions for McKean-Vlasov SDEs with coefficients \textcolor{black}{having} a general dependency in the measure argument, see e.g. \cite{chau:frik:19}. 

{\color{black}The present work} is the first of a series of two. We will \textcolor{black}{concentrate here} on the regularization by noise aspects associated with McKean-Vlasov SDEs of type \eqref{main} (i.e. for which the measure dependence arises from an interaction kernel).  \textcolor{black}{The} connections with specific models and their extension to the $\alpha $-stable setting will be thoroughly investigated in the companion paper \cite{chau:jabi:meno:22-2}.\\

\noindent\textbf{A regularization by noise perspective \textcolor{black}{for \textit{usual/Non McKean-Vlasov} SDEs}.} For \textit{standard} SDEs, i.e. when the dynamics 
{\color{black}write} (at least formally in the case of a distributional drift):
\begin{equation}\label{main_NON_MCKEAN}
Y_s^{t,\mu} = \xi + \int_t^s   b(v,Y_v^{t,\mu}) dv + (\mW_s-\mW_t),\quad \xi \sim \mu \in \mathcal P(\R^d),
\end{equation}
 and thus {\color{black}when} there is no spatial convolution between the drift and the law, a huge literature investigated as well the smoothing effects of the noise. The underlying idea is that the presence of noise allows to restore uniqueness in some sense \textcolor{black}{when} the corresponding \textcolor{black}{differential equation} without noise would be ill-posed.

Such phenomena have for instance been studied since the seminal works of Zvonkin \cite{Zvonkin-74} - who established, in the scalar \textcolor{black}{Brownian} case, strong well-posedness for \eqref{main_NON_MCKEAN} when the drift is only bounded and {\color{black}(Borel)} measurable -  and then of Stroock and Varadhan \textcolor{black}{\cite{StrookVaradhan-06}} - who extended the result from a weak perspective, in the multidimensional setting, allowing as well the noise to be multiplicative with a diffusion coefficient only measurable in time and space and continuous in \textcolor{black}{space}. The strong uniqueness result of Zvonkin was then extended to the multidimensional case by Veretennikov in \cite{Veretennikov-81}
 and later for $L^r((t,T],L^p(\R^d,\R^d))$ drifts satisfying a Serrin integrability condition \textcolor{black}{by Krylov and R\"ockner} in \cite{kryl:rock:05} (with recent investigation of the critical case {\color{black}in} \cite{kryl:21}, \cite{rock:zhao:21}). We can also refer to the monograph by Flandoli{\color{black},} \cite{Flandoli-10}{\color{black},} for a survey of related topics or to {\color{black}the work of Delarue and Flandoli} \cite{dela:flan:14} in which a vanishing viscosity restores a kind of uniqueness for a Peano type ODE {\color{black} by - r}oughly speaking - {\color{black}selecting} the maximal solutions of the ODE. 
 
 {\color{black}Let us further explain the results obtained in this latter work. Consider the following scalar perturbed Peano dynamics: for $\epsilon>0 $ \textcolor{black}{and (temporarily) $\beta\in (0,1) $},
\begin{equation}\label{PEANO}
dx^\epsilon_{\color{black}s}={\rm sgn}(x^\epsilon_{\color{black}s})|x^\epsilon_{\color{black}s}|^\beta{\color{black}\,ds}+\epsilon d\mW_{\color{black}s},\ x_0^\epsilon=0.
\end{equation}
It is clear that for $\epsilon=0 $ the maximal solutions of \eqref{PEANO} write: $x_{{\color{black}s}}^0=\underline{+}c_\beta {\color{black}s}^{1/(1-\beta)} $. Heuristically, the authors proved that considering $\epsilon>0$ allows to restore uniqueness as far as the noise dominates the maximal solution in short time. In terms of scales, recalling that the typical scale of 
$\mW_{\color{black}s} $ is ${\color{black}s}^{1/\alpha}$, this writes $s^{1/\alpha} > s^{1/(1-\beta)}$ which for $s<1$ equivalently gives the inequality  $1/\alpha <\textcolor{black}{1/ (1-\beta)}$, and therefore the condition 
\begin{equation}\label{T1}\tag{T${}_1$}
\beta>1-\alpha.
\end{equation}
This phenomenon not only gives a heuristic rule for the  thresholds on the regularity indexes that guarantees restoration of uniqueness, but also allows to build counter-examples \textcolor{black}{to weak well-posedness}. \textcolor{black}{To this end we can mention the famous counter-example from Tanaka \textit{et al.}  \cite{tana:tsuc:wata:74} who derived that, for the previous Peano type drift $x\mapsto {\rm sgn}(x)|x|^\beta $, in the scalar case, uniqueness fails starting from 0 for $\beta<1-\alpha$ (see Theorem 3.2 therein).}
We also refer to e.g. \cite{chau:16,Raynal2017RegularizationEO} and \cite{mar:meno:21} for further extensions to (respectively) Brownian ($\alpha=2$) and pure jump ($\alpha \in (1,2)$) degenerate noises. On the contrary well-posedness of \eqref{PEANO} under the condition \eqref{T1} was precisely proved in \textcolor{black}{a particular} Brownian scalar setting \textcolor{black}{by Gradinaru and Offret} in \cite{grad:offr:13}.

From the result of Tanaka \textit{et al.}  \cite{tana:tsuc:wata:74}, it is thus clear that in the \textit{supercritical} regime, i.e. $\alpha< 1 $, only drifts with positive regularity indexes ``$\beta$'' can be addressed. \textcolor{black}{In the present work, we mainly focus on distributional drifts which consequently restricts the admissible stability range to $\alpha \in (1,2]$}. For results in the \textit{supercritical} regime in this \emph{linear} we can refer to \cite{chen:zhan:zhao:21} or to \cite{silv:12}, \cite{chau:meno:prio:19} for a related Schauder theory, which quantify the regularization of the noise. Eventually, \textcolor{black}{we refer to \cite{tana:tsuc:wata:74}, \cite{koma:84} and \cite{chau:meno:prio:critic} for results in the critical case}.\\

When $\alpha \in (1,2]$, \textcolor{black}{and} in the case where $b$ in  \eqref{main_NON_MCKEAN} \textcolor{black}{is} a distribution, i.e. $\beta<0 $ in \eqref{hyp_b},  systems of the form of \eqref{main_NON_MCKEAN} ha\textcolor{black}{ve} been investigated} with a growing interest. From the initial work of \textcolor{black}{Bass and Chen} \cite{bass:chen:01}, for a scalar SDE and a time homogeneous drift, which was investigated through a Dirichlet form approach, there have been various works addressing the well-posedness of \eqref{main_NON_MCKEAN} for time-dependent drifts\footnote{in the time-independent case the problem can still be investigated through Dirichlet forms{\color{black},} \cite{flan:russ:wolf:04}} through PDE {\color{black}analysis} (\cite{flan:isso:russ:17}, \cite{zhan:zhao:17} for local generators, \cite{athr:butk:mytn:18}, \cite{CdRM-20} for the strictly stable case) or rough-path techniques (\cite{dela:diel:16}, \cite{krem:perk:20} for a Markovian setting or \cite{catellier_averaging_2016} for a purely pathwise approach which also allows to address fractional dynamics). At this point, one has to notice that the SDE \eqref{main_NON_MCKEAN} is only \textit{formal} for distributional drifts. The \textcolor{black}{associated} \textit{weak} well-posedness is investigated through the martingale problem formulation, which possibly needs to be tailored, but in any case requires in the Markovian setting to investigate the well-posedness of the underlying associated PDE. Namely, 
\begin{equation}\label{PDE_FIRST_CAR}
\left\{\begin{array}{l}
\partial_t u+b\cdot {D} u+L^\alpha u
=-f\, \text{ on } [0,T)\times \R^d,\\
u(T,\cdot)=0,
\end{array}\right.
\end{equation}
where $L^\alpha$ is the generator of the driving stable process $\mathcal W ${\color{black}, ${D}$ a generalized gradient operator,} and $f$ belongs to a  rich enough class of functions in order to characterize the law of the canonical process associated with the martingale problem solution. Equation \eqref{PDE_FIRST_CAR} is to be solved in a \textit{mild} sense through a Duhamel type formulation. Namely we \textit{formally} rewrite {\color{black}it as}:
\begin{equation}\label{MILD_LIN}
\forall (t,x)\in [0,T]\times \R^d,\quad u(t,x) =  \int_t^T ds P^\alpha_{s-t}[\{f + b\cdot {D}u\}](s,x),
\end{equation}
where $(P_t^\alpha)_{t\ge 0}$ is the semigroup generated by $L^\alpha$.

The point is that this (mild) formulation already emphasizes the difficulties and allows to understand the thresholds on the parameters arising for well-posedness in the \textcolor{black}{aforementioned} articles. 
Indeed, not only is the representation \eqref{MILD_LIN} implicit (in the sense that $u$ appears on both sides of the equality) but for a distributional drift, it is delicate to give a meaning to the term $P^\alpha_{s-t}\textcolor{black}{[ b\cdot Du]}$. This can somehow be done from the Bony paraproduct rule. Namely $b\cdot Du$ can be well-defined, as a distribution, provided the sum of the regularities of each term is positive. Assume now {\color{black}and} for a while that $b $ is time independent and belongs to $ B_{\infty,\infty}^{\beta}(\R^d,\R^d) $, $\beta\in (-1,0) $ (or 
that $b=DB $ with $B\in B_{\infty,\infty}^{\beta+1}(\R^d,\R)$). Suppose as well that the former  product makes sense as a distribution. It can \textcolor{black}{then} only be expected to lie in the same space than $b$, i.e. in $ B_{\infty,\infty}^{\beta}(\R^d,\R^d) $. Thus, for any (smoother) \textcolor{black}{class} of functions \textcolor{black}{to} which $f$ belong\textcolor{black}{s} to, the above mild representation suggests, from a \textcolor{black}{\textit{heuristic}} parabolic bootstrap \textcolor{black}{argument, that the best one can expect is} that $u\in C^{\beta+\alpha}(\R^d,\R) $ with $\beta+\alpha>1 $. 
Consequently, $Du  $ is expected to belong to $C^{\beta+\alpha-1}(\R^d,\textcolor{black}{\R^d}) $ with $ \beta+\alpha-1>0$. From the 
{\color{black} Bony paraproduct rule}, for the product {\color{black} $b\cdot Du$} to be well defined, one should have:
\begin{equation}\label{TRHESHOLD_LIN}\tag{T${}_2$}
\beta+(\beta+\alpha-1)>0\iff \beta>\frac{1-\alpha}{2}.
\end{equation}
This is precisely the threshold that appears in \cite{flan:isso:russ:17}, \cite{athr:butk:mytn:18}, \cite{CdRM-20} in the corresponding setting. \textcolor{black}{This is somehow the price to pay, with respect to the less stringent threshold in \eqref{T1}, to derive weak well-posedness in the singular case}.
Let us specify that in \cite{dela:diel:16}, \cite{krem:perk:20} the authors manage to go below this threshold
adding some structure to the drift, i.e. assuming that this latter can be enhanced into a rough path structure. In that case the threshold can decrease to $(2-2\alpha)/3 $.

Even when having at hand the martingale solution, specifying 
rigorous integral dynamics - for an adequate notion of solution - for the formal SDE \eqref{main_NON_MCKEAN} is not easy when the drift is \textcolor{black}{a} distribution. Under the previously described conditions, it turns out that the integrated drift has to be understood as a Dirichlet process and it seems that the most accurate description of the singular dynamics, which reconstructs the drift as a Young integral with respect to a time-space convolution of the drift and the density of the driving noise, is provided in \cite{dela:diel:16} \textcolor{black}{in the Brownian case} and \cite{CdRM-20} for the stable Young regime.\\
 
\noindent\textbf{Regularization by noise for McKean-Vlasov SDE.} In its simplest setting, notably for Brownian driven diffusions, the (strong) well-posedness of McKean-Vlasov SDEs is fulfilled whenever the coefficients are Lipschitz continuous in $\mathbb R^d\times\mathcal P_\ell(\mathbb R^d)$, $\ell \in [1,+\infty)$ 
{\color{black}and} $\mathcal P_\ell(\mathbb R^d)$ standing for the set of probability measures which \textcolor{black}{admit} a moment of order $\ell$ and 
{\color{black}is} equipped with the topology induced by the Wasserstein metric:
\[
W_\ell(\mu,\nu)=\Big(\inf_{X\sim\mu,Y\sim\nu} \E[|X-Y|^\ell]\Big)^{1/\ell}.
\]
The distance $W_\ell$ appears to be quite natural when dealing with McKean-Vlasov equations with Lipschitz properties, especially because it can be expressed in term of the  corresponding  $L^\ell$ distance. Indeed, on the one hand, this relation allows to implement the Picard-Lindelh\"of fixed point procedure to establish well-posedness results, in a rather similar way than for classical SDEs. On the other hand, it allows to ensure the well-posedness of the associated Mean-Field particle system, viewed as a  high dimensional SDE. We refer to \emph{e.g.} \cite{McKean-67}, \cite{szni:88}, \cite{Meleard-95} and \cite{CarDel-18}, Section 5.7.4, for further details.

It nevertheless appears that, in the \textit{non-degenerate} case, the usual Lipschitz property w.r.t. the Wasserstein metric can be \textcolor{black}{\textit{weakened}} to a Lipschitz property w.r.t. the total variation distance defined, for any $\mu,\nu \in \mathcal P(\R^d)$,  \textcolor{black}{as} $\dtv(\mu,\nu) = \sup_{A \in \mathcal B(\R^d)} |\mu(A)-\nu(A)|$. This goes back to the work of Shiga and Tanaka \cite{ShiTa85} for a particular type of measure dependence and then to Jourdain \cite{Jourdain-97} in a rather general framework - in both works, the diffusion matrix is not allowed to be measure dependent. \textcolor{black}{We used above the word ``\textit{weakened}''} to refer to the fact that, for any probability measures $\mu,\nu$ with full support on some compact subset $\mathcal K$  of $\R^d$ one has, \emph{e.g.} $W_\ell(\mu,\nu) \le {\rm diam}(\mathcal K) \dtv^{1/\ell}(\mu,\nu)$, and a large class of solutions of non-degenerate SDEs lives with high probability in compact subset\textcolor{black}{s} of the considered space. Hence, the \emph{non-degeneracy} of the noise allows to consider this \emph{stronger} topology. This particular smoothing effect of the finite dimensional noise w.r.t. the infinite dimensional measure variable is highlighted in e.g. Section 2 in \cite{chau:frik:19}. Using this quite tricky regularization phenomenon, the authors succeeded therein to derive a rather general well-posedness  theory for Brownian driven McKean-Vlasov SDEs whose drift and diffusion coefficients may be Lipschitz w.r.t. \emph{stronger} topologies than the usual Wasserstein metric. {\color{black}In this regard, we also refer the reader to \cite{CdR-20} where the first author extended Zvonkin's technique to McKean-Vlasov SDEs. To conclude this (partial) survey of the literature, let us finally point out the alternative methodology, highlighted in Mishura and Veretennikov \cite{MisVer-20}, to quantify pathwise $d_{TV}$ distance through a change of probability measures.  While initially \textcolor{black}{designed} to handle weak uniqueness results for Borel measurable kernels with linear growth in the space argument, this methodology has been extended to different situations, \textcolor{black}{from non smooth drifts in both the spatial and measure arguments} (\cite{Lacker-18}) to measure-dependent diffusion coefficients (\cite{RocZha-21}). The use of change of probability measures remain\textcolor{black}{s} nonetheless \textcolor{black}{quite specific} to Brownian driven equations and \textcolor{black}{does not transpose to the general pure jump stable setting}.}
 
 While often limited to the Brownian case, well-posedness results for non-degenerate stable (and pure jumps) driven McKean-Vlasov SDEs were addressed in \cite{frik:kona:meno:20}  (see also \cite{JouMelWoy-08} for more general jumps processes with suitable moments {\color{black}and \cite{Meleard-95} - and subsequent references - for the more specific case of \textcolor{black}{the} Boltzmann equation}). In \textcolor{black}{\cite{frik:kona:meno:20}, the authors} show that one can again consider stronger topologies for the regularity of the coefficients w.r.t. the measure variable and still preserve the well-posedness of the system. Let us mention as well that the approach adopted \textcolor{black}{therein} allows to reach the super-critical case ($\alpha <1$).


We now briefly specify part of the results obtained in \cite{chau:frik:19} and \cite{frik:kona:meno:20}. Therein, by stronger topologies, the authors roughly mean that the coefficients are H\"older (or even bounded measurable) in space and Lipschitz continuous w.r.t. the chosen distance on the space of probability measures. For two given probability measures, this latter is defined as the supremum of their difference integrated against normalized functions whose regularity is linked to the spatial one of the coefficients.
Namely, measurable and bounded (for the drift coefficient in \cite{chau:frik:19}) and H\"older continuous for the others. Again, the regularity assumed on the coefficients imposes, in turn, the class of test functions against which probability measures have to be tested when investigating the well-posedness of the system. 
From the fixed point perspective of those works, this homogeneity (between the metrics) seems rather natural. Indeed, the underlying linear  problems precisely enjoy some stability on the indicated function space (Schauder type estimates).\\

\noindent\textbf{McKean-Vlasov SDE with distributional interaction kernel.}  One of the main objective in the current work is to take full advantage of the non-degeneracy of the noise and of the particular structure of the measure dependence (allowing, in some sense, to regularize the drift through the convolution with the law of the solution - provided the law itself is smooth enough) in order to obtain a well-posedness result for a rather large class of interaction kernel\textcolor{black}{s}, possibly larger than the one considered in the 
{\color{black}{\it classical measure-independent}} case \textcolor{black}{described above}. As it will be heuristically discussed below, it seems to us that such a class is, almost, the largest that could be considered in the current setting.

In order to \textcolor{black}{emphasize}
some of the particularities of the model, let us \textcolor{black}{again} consider for a while that $\beta \in (-1,0]$ and that the interaction kernel is time-homogeneous, \textcolor{black}{i.e.} $b\in B_{\infty,\infty}^{\beta}(\R^d,\R^d)$.\\

\noindent\textit{Regularity with respect to the measure argument and associated metric.} Let us first \textit{formally} define the measure indexed drift 
$$\mathcal B_\nu(x) := \int b(x-y)\nu(dy),\quad \nu\in \mathcal P(\R^d),$$
associated with the non-linear drift of equation \eqref{main}.

Let us point out that this definition is formal in the sense that, since $b$ i{\color{black}s} singular, it is intuitively clear that some additional properties are needed on the probability measure $\nu$ for the quantity to be well defined. A rather natural way to define an appropriate setting is to proceed through duality. 

If we restrict to probability measures with density, i.e. $\nu(dy)=\nu(y) dy$, the integral  $\int {\color{black}dy}\ b(x-y)\nu(y)$
makes sense as soon as $\nu\in B_{1,1}^{\textcolor{black}{-\beta}}(\R^d,\R)$ from the inequality:
\begin{equation}\label{DEF_DRIFT}
|\mathcal B_\nu(x)| = |\int {\color{black}dy}\ b(x-y)\nu(y)
| \le |b|_{B^{\beta}_{\infty,\infty}(\R^d,\R^d)} |\nu|_{B^{-\beta}_{1,1}(\R^d,\R)},
\end{equation} 
see e.g. \cite{Lemarie-02} for a thorough presentation of duality results between Besov spaces. Namely, the \textit{regularity} of the density of the measure is needed to compensate the singularity of the kernel $b$ to define the corresponding drift $\mathcal B_\nu $. As a consequence, this latter is \textcolor{black}{then} defined \textit{pointwise}.

The previous considerations naturally induce a metric on the probability measures whose densities belong to $B^{-\beta}_{1,1}(\R^d,\R)$ for which 
the corresponding measure indexed drift is Lipschitz (in its measure argument). Namely, for all $x\in \R^d $,
\begin{equation}\label{LIP_METRIC}
|\mathcal B_\nu(x) - \mathcal B_{\nu'}(x)| = |\int {\color{black}dy} \ b(x-y)(\nu-\nu')(y)
| \le |b|_{B^{\beta}_{\infty,\infty}(\R^d,\R^d)} |\nu-\nu'|_{B^{-\beta}_{1,1}(\R^d,\R)}.
\end{equation}
In view of the previous discussion on the specific features of the regularization by noise for McKean-Vlasov SDEs, this thus suggests to look for (weak) solutions whose 
{\color{black}{marginal laws}} live for almost every time in the space $B^{-\beta}_{1,1}(\R^d,\R)$.  

\noindent\textit{Attainable thresholds (in term of irregularity of the interaction kernel).} Here, instead of considering the backward Kolmogorov PDE associated with \eqref{main_NON_MCKEAN}, we rather investigate the (forward) non-linear Fokker-Planck equation. We are thus led to consider mild solution\textcolor{black}{s} of this latter of the form
\begin{equation}\label{MILD_NONLIN}
\forall (s,y)\in [t,T]\times \R^d,\quad \rho(s,y) =  p_{s-t}^\alpha \star \mu(y) + \int_t^s dv P^\alpha_{s-v}[\{ 
\div[(\mathcal B_\rho) \cdot \rho]\}](v,y),
\end{equation}
where {\color{black}$\rho(s,\cdot)$ denotes the density function of $\bmu_s^{t,\mu}$, $\star$ the convolution product along the space variable, }$p^\alpha $ stands for the density of the driving process in \eqref{main} and, again, $P^\alpha $ for the corresponding semi-group.
In comparison with \eqref{MILD_LIN}, we do not need to define the product $b \cdot Du$ anymore, but only $\mathcal B_\rho \cdot \rho$ and prove somehow that the divergence of this term will be mapped onto a suitable function space by the stable semi-group $P^\alpha$ - precisely, the space where $\rho$ is supposed to belong to, for almost every time. According to the previous discussion, this space should be $B^{-\beta}_{1,1}(\R^d,\R)$.  Let us now present some formal calculations, proved below, to understand 
 which thresholds \textcolor{black}{are likely to be} attainable. Recall that the space $B^{-\beta}_{1,1}(\R^d,\R)$ can be described as the set of integrable maps having an integrable H\"older modulus. It will be established in Lemma \ref{lem_unifesti_gencase_2} below, through integration by parts together with Young like convolution inequalities in Besov spaces,  that
 \begin{equation}
 \label{CTR_HEURI_2}
 |P_{\textcolor{black}{s-v}}^\alpha \{\div(\mathcal B_\rho(\textcolor{black}{v},\cdot) \cdot \rho(\textcolor{black}{v},\cdot))\}|_{B_{1,1}^{-\beta}}  \le \|\mathcal B_\rho(\textcolor{black}{v},\cdot)\|_{L^\infty} |\rho(\textcolor{black}{v},\cdot)|_{L^1} \|\nabla p_\alpha(s-\textcolor{black}{v},\cdot)\|_{B_{1,1}^{-\beta}}  \le  \frac{C|b|_{B_{\infty,\infty}^\beta}|\rho(v,\cdot)|_{B_{1,1}^{-\beta}}}{(\textcolor{black}{s-v})^{\frac 1\alpha-\frac \beta \alpha}},
 \end{equation}
using \textcolor{black}{\eqref{DEF_DRIFT}} and the fact that we are considering probability measures to derive the last inequality.  Thus, \textcolor{black}{provided we can \textit{somehow} control the $B_{1,1}^{-\beta} $-norm of the solution in time}, we can hope to derive well-posedness \textcolor{black}{of \eqref{MILD_NONLIN}} for 
 \begin{equation}\label{THRESHOLD_MCKEAN}
 \beta>1-\alpha,
\end{equation}
 which \textcolor{black}{gives an integrable time singularity in \eqref{CTR_HEURI_2}} and clearly improves the threshold \eqref{TRHESHOLD_LIN} for the \textcolor{black}{singular} linear case \textcolor{black}{($\beta<0 $). It actually exactly corresponds to the condition \eqref{T1} which is sufficient to guarantee weak well-posedness in the non singular ($\beta \in (0,1)$) pure jump case}. \textcolor{black}{Observe for instance that} in the Brownian setting $\alpha=2 $, \eqref{TRHESHOLD_LIN} gives $\beta>-1/2 $ whereas the above condition reads $\beta>-1 $. This mainly comes from the specific dependence structure of the non-linear drift w.r.t. the law argument. Indeed, in order to define the product $\mathcal B_\rho \cdot \rho$, the interaction kernel in the non-linear drift has already been regularized through the convolution operator and has been mapped, as suggested in \textcolor{black}{\eqref{DEF_DRIFT}}, to $L^\infty(\R^d,\R^d)$ for almost every time. While rather naive, this observation is crucial in our analysis and allows to bypass the thresholds obtained in the linear case. In fact, even in the non-linear case, it is not clear that general dependence w.r.t. the law argument would allow to bypass such thresholds.  Let us also mention that working with probability 
 {\color{black}densities} plays a crucial role as well here. Indeed, the control in $L^1(\R^d,\R)$ norm of $\rho(\textcolor{black}{r},\cdot)$ in \eqref{CTR_HEURI_2} allows to ``remove'' the non-linearity and thus the quadratic dependence in $\rho$ \textcolor{black}{in the} right hand side of \eqref{MILD_NONLIN}. \textcolor{black}{This specific feature will yield well-posedness, for the indicated non-linear threshold, for any positive time $T>0$}. Such phenomenon is referred to as \emph{dequadrification} in the following.\\ 
 
 \begin{rem}\,
 \textit{On the associated metric and space in which the solution belongs to}. For whom are familiar with the smoothing effect of \emph{linear} SDEs (as briefly described above), it would seem reasonable, in view of \eqref{MILD_NONLIN}, to look for the law of the process in the space of probability measures whose regularity is precisely given according to the usual parabolic bootstrap associated with the non-degenerate noise. Namely, as the interaction kernel $b$ belongs to $B_{\infty,\infty}^{\beta}(\R^d,\R^d)$, the corresponding space for probability measures would be $B_{1,1}^{\beta + \alpha}(\R^d,\R)$, similarly to the \emph{linear} case. Note that, in the singular setting, the regularity associated with the parabolic bootstrap, i.e. $\beta+\alpha $, is greater or equal than the minimal regularity, i.e. $-\beta $, needed to define properly the drift (see \eqref{DEF_DRIFT}). Namely,
 $$(\beta + \alpha)-(-\beta) = 2\beta + \alpha\underset{\textcolor{black}{\eqref{THRESHOLD_MCKEAN}}}{ >}2-2\alpha+\alpha = 2-\alpha,$$
meaning that, except in the Brownian case ($\alpha=2$), the regularity of the weak solution should be better than the one obtained in this work. However, we were not able to prove such a result for any positive time $T>0$, but only in small time. We feel that this \textcolor{black}{restriction} comes from the particular smoothing effect of the McKean-Vlasov SDEs previously described and the dequadrification approach, which allows to handle any initial measure but not to benefit from a potential smoothness of this latter to iterate the analysis 
 {\color{black}(this feature will be one of the cores of \cite{chau:jabi:meno:22-2} as briefly mentioned below).}
\end{rem}

The main difficulty of course is to make the above arguments rigorous. This is what we will actually do considering the Fokker-Planck equation associated with the McKean-Vlasov SDE \eqref{main}. We will not proceed through a fixed point procedure but {\color{black} prove that, starting from a mollified version of the drift coefficient,  we can actually control the quantities we highlighted before, uniformly in the mollification parameter  \textcolor{black}{under a suitable condition relating the parameters, Assumption \eqref{cond_gencase} below, which is consistent with the previous discussion and enlarges the framework}. The well-posedness of \eqref{main} is then drawn from a stability argument.}

We conclude this introduction emphasizing that for physics related models, e.g. the Burgers equation \textcolor{black}{considered from the origin of McKean-Vlasov SDEs, one should expect that the case $\beta=-1,\alpha=2 $ could be handled}. This is not the case here because we wanted to focus on a generic approach working viewless of the regularity of the initial law. The \textit{critical cases} (which saturate the previous inequalities for the thresholds) will be \textcolor{black}{addressed} through different tools, somehow more connected with \textit{usual} arguments in non-linear analysis, in the companion paper \cite{chau:jabi:meno:22-2}. The strategy therein will allow to take advantage of a smooth\textcolor{black}{er} initial condition or some specific structural conditions on the drift (e.g. free divergence for fluid related problems) to go beyond the discussed thresholds.\\


\textbf{Organization of the paper.} Our main results are stated in the next section{\color{black}, and their  \textcolor{black}{proofs} are developed in Section \ref{SEC_PROOF_WEAK}}. 
{\color{black}Our} strategy, presented in Section \ref{sec_strategy}, consists first in  establishing the existence of a solution to \eqref{main} in terms of a non-linear martingale problem, through a mollification of the drift coefficient and 
{\color{black}stability} arguments. This approach allows to derive the {\color{black}weak} well-posedness of a solution to \eqref{main} directly from the construction of its 
{\color{black}{marginal laws} seen} as solution to the non-linear Fokker-Planck equation \textcolor{black}{\eqref{MILD_NONLIN}} related to \eqref{main} {\color{black}and a weak uniqueness criterion {\it \`a la} Krylov-R\"ockner}. {\color{black}Strong well-posedness is next derived from a classical pathwise uniqueness result, see Krylov-R\"ockner \cite{kryl:rock:05} for the Brownian case and Xie-Zhang \cite{xie:zhan:20} for the pure jump case.}
 
We present in Section \ref{SEC_BESOV} the useful properties of Besov spaces that will be used in our analysis.

The framework of Besov interaction kernel\textcolor{black}{s} allows to revisit the martingale approach to well-posedness for McKean-Vlasov models in a \textit{quite} systematic way starting from some (seemingly) sharp global density estimate\textcolor{black}{s}, valid for any initial law $\mu$. Section \ref{sec_nl_FK} is dedicated to this step, and Section \ref{sec_WP_SDE} to the derivation of the well-posedness results (in a weak and/or strong sense). Some technical results are presented in the Appendix. 

\subsection{Main results} From now on, in the \textcolor{black}{pure jump} case $\alpha\in(1,2)$, \textcolor{black}{we assume that the driving process $\mW$ in \eqref{main} satisfies the following condition}:

\noindent\textbf{Assumption \A{UE}.} The L\'evy measure $\nu$ of $\mW$ is given by the decomposition $\nu(dz)=w(d\xi)/\rho^{1+\alpha}{\color{black}\I_{\{\rho\textcolor{black}{>} 0\}}d\rho}$ where $w$ is a symmetric measure on the unit sphere $\mathbb S^{d-1}$ which satisfies the uniform non-degeneracy condition:
$$
\kappa^{-1}|\lambda|^\alpha\leq \int_{\mathbb S^{d-1}} |\xi\cdot\,\lambda|^\alpha\, w(d\xi)\leq\kappa|\lambda|^\alpha,\,\text{for all }\lambda\in\R^d\,,
$$
for some $\kappa\ge 1$. In particular, since we here consider an additive noise, its L\'evy measure can actually have a singular spherical part (cylindrical processes are allowed).

\begin{thm}\label{THM_GEN}
Fix $T>0$ and $t\in [0,T)$. Let $b$ in \eqref{main} belong to $L^r((t,T{\color{black}]},B_{p,q}^\beta(\R^d,\R^d)) $ 
and let the parameters satisfy the condition
\begin{equation}\label{cond_gencase}\tag{{\bf C0}}
\quad \beta >  1 -\alpha+ \frac dp +\frac \alpha r.
\end{equation}

Then, for any initial law $\mu \in \mathcal P(\R^d)$, the McKean-Vlasov SDE \eqref{main} admits  a weak solution s.t. its \textcolor{black}{marginal} law\textcolor{black}{s} $(\bmu_r^{t,\mu})_{r\in (0,T]} $ \textcolor{black}{have} a density \textcolor{black}{for almost any time}, i.e. for almost all $r\in (t,T], \  \bmu_r^{t,\mu}(dy)=\brho_{t,\mu}(r,y)dy$, which satisfies:
\begin{equation*}\label{INT_BRUT}
\int_t^T ds | \brho_{t,\mu}(s,\cdot)|_{B^{-\beta}_{p',q'}}^{\bar r}<+\infty,
\end{equation*}
for any $\bar r  \in \big[r',(-\beta/\alpha+d/(\alpha p))^{-1}\big)$ where $p^{-1}+(p')^{-1}=q^{-1}+(q')^{-1}=r^{-1}+(r')^{-1}=1 $. Moreover, the solution is unique among all the weak solutions that satisfy the above properties.\\

The above well-posedness \textcolor{black}{result} moreover holds in a strong sense if \textcolor{black}{the driving process $ $ is rotationally invariant} and the parameters satisfy the following more stringent (when $\alpha \neq 2$) condition:
\begin{equation}\label{COS}\tag{${\bf C0_S}$}
\beta >2-\frac 32\alpha+\frac dp +\frac \alpha r.
\end{equation}  
\end{thm}

\begin{rem}[About the strong uniqueness] \textcolor{black}{In the above theorem we assume for strong uniqueness that $\mathcal W $ is rotationally invariant. This is of course the case for $\alpha=2 $ but it actually means that in the pure jump case we restrict {\color{black}the noise} to 
the isotropic stable process for $\mathcal W $. This is mainly due to the fact {\color{black}that} to establish the result, {\color{black}we apply a uniqueness criterion due to Xie and Zhang, \textcolor{black}{see }\cite{xie:zhan:20}, where \textcolor{black}{a} driving isotropic \textcolor{black}{process} is considered  in order to deal more easily with a multiplicative noise and related heat kernel estimates.}
{\color{black}While clearly restrictive, w}e actually believe that in the current additive framework, the PDE analysis  of the indicated reference could be adapted and that the strong uniqueness result should hold under the condition \eqref{COS} for any symmetric stable driving process satisfying Assumption \textbf{(UE)} (including e.g. symmetric cylindrical processes)}.

Observe {\color{black}finally} that for $\alpha=2$ the conditions \eqref{cond_gencase} and \eqref{COS} for weak and strong uniqueness coincide.
\end{rem}

\subsection{Strategy of proof}\label{sec_strategy}

As explained before, the construction of a weak solution to \eqref{main} proceeds hereafter from a regularization procedure and a stability argument. When doing so, we will use the following approximation result whose proof{\color{black}, which relies on the preliminaries set in Section \ref{SEC_BESOV},} is postponed to the Appendix.

\begin{prop}[Smooth approximation of the drift and associated convergence properties]\label{PROP_APPROX}
Let  $b\in L^r((t,T],B_{p,q}^\beta)$ and $\beta\in (-1,0] $, $1\le p,q \le \infty $.
There exists a sequence of \textcolor{black}{time-space} smooth bounded functions $(b^\varepsilon)_{\varepsilon >0} $ s.t.
$$|b-b^\varepsilon|_{L^{\bar r}((t,T],B_{p,q}^{\tilde \beta})} \underset{\varepsilon \rightarrow 0}{\longrightarrow} 0,\quad \forall \tilde \beta<\beta,$$
with $\bar r=r $ if $r<+\infty $ and for any $\bar r<+\infty $ if $r=+\infty$. {\color{black}Moreover, \textcolor{black}{there exists $ \cc\ge 1$}, $\displaystyle \sup_{\varepsilon>0}|b^\varepsilon|_{L^{\bar r}((t,T],B_{p,q}^{\beta})}\le \textcolor{black}{\cc} |b|_{L^{\bar r}((t,T],B_{p,q}^{\beta})}$.}
\end{prop}
 From the specific \textcolor{black}{convolution}  structure of the non-linear drift, we now introduce, for any measure $\nu$ {\color{black}on $\mathbb R^d$} for which this is meaningful, the notation:
 $$\mathcal B_{\nu}(s,\cdot):=b(s,\cdot)\star \nu(\cdot){\color{black}=\int_{\mathbb R^d} b(s,\cdot-y)\,\nu(dy)}.$$
 We now write similarly, for all $\varepsilon>0 $
 \begin{equation}\label{DRIFT_NON_LIN_MOLL}
 \mathcal B_{\nu}^\varepsilon(s,\cdot):=b^\varepsilon(s,\cdot)\star \nu(\cdot),
 \end{equation}
 which is  well defined for any $\nu \in \mathcal P(\R^d)${\color{black}, regardless of the Besov space where $b$ belongs to,} since $b^\varepsilon $ is smooth and bounded. This allows to introduce in turn the following mollified version of \eqref{main}:
\begin{equation}\label{main_smoothed}
X_s^{\varepsilon,t,\textcolor{black}{\xi}} = \xi + \int_t^s \mathcal B^\varepsilon_{\bmu^{\varepsilon,t,\mu}_v}(\textcolor{black}{v},X_v^{\varepsilon,t,\textcolor{black}{\xi}}) dv + \mW_s-\mW_t,\quad \bmu_s^{\varepsilon,t,\mu}(s,\cdot)=\text{Law}(X_s^{\varepsilon,t,\textcolor{black}{\xi}}),\, s\in (t,T],
\end{equation}
which is strongly (and thus weakly) well-posed for any $\varepsilon>0$, see Frikha, Konakov and Menozzi \cite{frik:kona:meno:20} for $\alpha\in (1,2) $ or \textcolor{black}{Sznitman} \cite{szni:88} in the Brownian case. \textcolor{black}{We emphasize that, due to weak uniqueness, $\bmu_s^{\varepsilon,t,\mu}(s,\cdot) $ only depends on the law $\mu $ of the random initial condition $\xi $}. \\

Let us now prove that \textcolor{black}{the marginal laws} of the unique weak solution of this regularized McKean-Vlasov SDE \textcolor{black}{are} absolutely continuous w.r.t. the Lebesgue measure. Consider the associated decoupled flow, 
{\color{black}that is} the \emph{linear} SDE parametrized by the  measure flow $(\bmu_s^{\varepsilon,t,\mu})_{s\in [t,T]}$ which writes
\begin{equation}\label{decoupled_main}
\tilde X_s^{\varepsilon,t,x,\mu} = x + \int_t^s \mathcal B^\varepsilon_{\bmu^{\varepsilon,t,\mu}_v}(v,\tilde X_v^{\varepsilon,t,x,\mu}) dv + \mW_s-\mW_t,\quad x \in \R^d. 
\end{equation}
Again, for every $\varepsilon >0$, this \emph{linear} SDE admits a unique weak solution. We denote its 
{\color{black}{marginal laws}} by $(\tilde \bmu_s^{\varepsilon,t,x,\mu})_{s\in [t,T]}$. Note importantly that, as
$\bmu^{\varepsilon,t,\mu}$ and $\tilde \bmu^{\varepsilon,t,x,\mu}$ are uniquely determined\textcolor{black}{,} the following key relation holds: for all $A$ in $\mathscr B([t,T]) \otimes \mathscr B(\R^d)$, $\bmu^{\varepsilon,t,\mu}(A) =  \big(\int_t^T dr \int \int \mathbb I_{(r,y)\in A}\tilde \bmu_r^{\varepsilon,t,x,\mu}(dy)\mu(dx)\big)$. 
\textcolor{black}{\textcolor{black}{It is now worth noticing} that in the mollified setting the drift in equation \eqref{decoupled_main} is smooth in the time and space variables and bounded. Indeed, for all $y\in \R^d $
$$|\mathcal B^\varepsilon_{\bmu^{\varepsilon,t,\mu}_r}(r,y) |\le \|b^\varepsilon\|_\infty\int_{\R^d}\bmu^{\varepsilon,t,\mu}_r(dz)=\|b^\varepsilon\|_\infty,$$
and the controls for the derivatives could be derived similarly differentiating $b^\varepsilon $.
This is what allows to derive from Friedman \cite{Friedman-64} for $\alpha=2 $ and Bichteller \textit{et al.} \cite{bich:grav:jaco:87} in the pure jump case that the SDE admits  for any $s\in (t,T]$ a smooth density (we could also refer to \cite{meno:zhan:20} in the rotationally invariant case)}.
Thus, the law $\tilde \bmu_r^{\varepsilon,t,x,\mu}$ is absolutely continuous w.r.t. the Lebesgue measure with density $\tilde \brho_{t,x,\mu}^{\varepsilon}(r,\cdot)$.
 Hence, there exists $\brho_{t,\mu}^{\varepsilon}$ such that:
\begin{equation}\label{relation_density}
\forall A \in \mathscr B([t,T]) \otimes \mathscr B(\R^d),\quad \bmu^{\varepsilon,t,\mu}(A) = \int_A \int_{\R^d} \tilde \brho_{t,x,\mu}^\varepsilon(v,y) \mu(dx) dv dy=:\int_A \brho_{t,\mu}^\varepsilon(v,y) dvdy,
\end{equation}
\emph{i.e.} $\bmu^{\varepsilon,t,\mu}$ is absolutely continuous as well. We emphasize that the relation \eqref{relation_density} holds for any $\alpha \in (1,2]$.\\

We now state a result, which will be at the starting point of our analysis and whose proof is postponed to Appendix \ref{PROOF_LEM_DUHAMEL_MOLL} for the sake of simplicity. 
\begin{lemme}[Duhamel representation for \textcolor{black}{the time marginal laws of the process with mollified interaction kernel}]\label{LEM_DUHAMEL_MOLL}
The following Duhamel representation holds for the solution of equation \eqref{main_smoothed}. For each $\varepsilon>0$, $\brho^\varepsilon_{t,\mu}$ satisfies for  all $s\in (t,T] $ and all $y\in \R^d $:
\begin{eqnarray}\label{main_MOLL}
\brho^\varepsilon_{t,\mu}(s,y) = p^\alpha_{s-t}\star \mu(y) \textcolor{black}{-}\int_t^s dv \Big[\{\mathcal B_{\brho^\varepsilon_{t,\mu}}^\varepsilon(v,\cdot) \brho^\varepsilon_{t,\mu}(v,\cdot)\} \star \nabla p_{s-v}^{\alpha}\Big] (y),
\end{eqnarray}
where $p^\alpha$ stands for the density of the driving process $\mW$, and with a slight abuse of notation w.r.t. \eqref{DRIFT_NON_LIN_MOLL}, $\mathcal B_{\brho^\varepsilon_{t,\mu}}^\varepsilon(v,\cdot)=[b^\varepsilon (v,\cdot)\star  \mathbf \brho^\varepsilon_{t,\mu}(v,\cdot)] $. 
\end{lemme}

We will actually prove (see Lemma \ref{lem_unifesti_gencase_2}) that there exists a constant $C\ge 0$, s.t. for any $\varepsilon>0 $, and $p,q,r,\beta $ satisfying condition \eqref{cond_gencase}  (see Theorem \ref{THM_GEN}):
$$\qquad |\brho^\varepsilon_{t,\mu}|_{L^{r'}\big((t,T],B^{-\beta}_{p',q'}\big)}\le C(T-t)^\theta,\, \theta>0.$$

It is furthermore clear that for any $\varepsilon >0$, $\brho^{\varepsilon}_{t,\mu} $ solves, in the distributional sense, the equation:
\begin{equation*}
\label{PDE_EPS}
\begin{cases}
\partial_s\brho^{\varepsilon}_{t,\mu}(s,y)+ \div ( \mathcal B_{\brho^{\varepsilon}_{t,\mu}}^{\varepsilon}(s,y) \brho^{\varepsilon}_{t,\mu}(s,y))-L^\alpha \brho^{\varepsilon}_{t,\mu}(s,y)=0, \\
\brho^{\varepsilon}_{t,\mu}(t,\cdot)=\mu,
\end{cases}
\end{equation*}
or equivalently, for all $\varphi \in \textcolor{black}{C_0^\infty\big((-T,T)\times \R^{d}\big) }$ (up to a possible symmetrization for  $s<t$ and with $C_0^\infty\big( (-T,T)\times \R^{d}\big)$ standing for {\color{black}the set of }real valued infinitely differentiable functions with compact support in $(-T,T)\times \R^{d}$,
\begin{eqnarray}
\label{PDE_EPS_VAR}
-\int\varphi(t,y)\mu(dy)
&\textcolor{black}{-}&\int_t^T ds \int_{\R^d} dy\,  ( \mathcal B_{\brho^{\varepsilon_k}_{t,\mu}}^{\varepsilon_k}(s,y) \brho^{\varepsilon_k}_{t,\mu}(s,y))\cdot \textcolor{black}{\nabla \varphi(s,y)}\notag\\
&+&\int_t^T ds \int_{\R^d}  dy \brho^{\varepsilon_k}_{t,\mu}(s,y) (-\partial_s+(L^\alpha)^*)\varphi(s,y)=0,
\end{eqnarray}
where $(L^\alpha)^* $ stands for the adjoint of $L^\alpha$
 {\color{black}. The driving process $\mW$ being symmetric,}
then $ (L^\alpha)^*=L^\alpha$.\\
 
Provided that $\brho^{\varepsilon}_{t,\mu}$  admits a limit $\brho_{t,\mu}$ in some appropriate function space ({\color{black} see} Lemma \ref{FIRST_STAB} {\color{black} below}) - which precisely allows to take the limit in the Duhamel formulation \eqref{main_MOLL} - we derive that the limit satisfies (Lemma \ref{lem_ex_duha_gencase})
\begin{eqnarray}\label{main_DUHAMEL}
\brho_{t,\mu}(s,y) = p^\alpha_{s-t}\star \mu (y)  \textcolor{black}{-}\int_t^s dv \Big[\{\brho_{t,\mu}(v,\cdot)\mathcal B_{\brho_{t,\mu}}(v,\cdot)\} \star \nabla p_{s-v}^{\alpha}\Big] (y),
\end{eqnarray}
\emph{i.e.} is a mild solution to the non-linear Fokker-Planck equation related to \eqref{main}. By extension  $\brho_{t,\mu}(s,y)\textcolor{black}{\,dy}$ also provides a distributional solution to the non-linear Fokker-Planck equation related to \eqref{main}:
\begin{eqnarray}
-\int\varphi(t,y)\mu(dy)
&-&\int_t^T ds \int_{\R^d} dy\,  ( \mathcal B_{\brho_{t,\mu}}(s,y) \brho_{t,\mu}(s,y))\cdot \textcolor{black}{\nabla \varphi(s,y)}\notag\\
&+&\int_t^T ds \int_{\R^d}  dy \brho_{t,\mu}(s,y) (-\partial_s+{\color{black}L^\alpha}
)\varphi(s,y)=0, \ \varphi \in \textcolor{black}{C_0^\infty\big((-T,T)\times \R^{d}\big) }.\notag\\
\label{NL_PDE_FK}
\end{eqnarray}

This provides a sufficient (and necessary) basis to construct a solution to \eqref{main} identifying {\color{black}it as} the limit of the martingale problem related with \eqref{main_smoothed}.  More precisely, from the solution to \eqref{main_smoothed}, one can consider for each $\varepsilon >0$ the probability measure $\mathbf P^\varepsilon$ on the space $\Omega_\alpha$ (where $\Omega_\alpha$ stands for the space of c\`adl\`ag functions $\mathbb D([t,T];\R^d)$ if $\alpha \in (1,2)$ and the space of continuous functions $\mathcal C([t,T];\R^d)$ otherwise) such that, for $x(s)$, $t\le s\le T$, the canonical process on $\Omega_\alpha$, and for $\mathbf P^\varepsilon_t(s,dx):=\mathbf P^\varepsilon_t(x(s)\in dx)$ the family of probability measures induced by $x(s)$, we have:  {\color{black} the initial  measure $\mathbf P^\varepsilon_t(t,\cdot)$}
 is equal to $\mu$ 
and for all function $\phi$ twice continuously differentiable on $\R^d$, with bounded derivatives at all order, the process
$$
\phi(x(s))-\phi(x(t))-\int_t^s\left\{\mathcal B_{\mathbf P^\varepsilon_t}(v,x(v))\cdot \nabla\phi(x(v))+L^\alpha(\phi(x(v)))\right\}\,\,dv ,\,t\le s\le T,
$$
is a martingale. Provided, again, that for any $\varepsilon>0$ the  {\color{black}marginal distributions} 
$\mathbf P^\varepsilon_t(s,dx)=\brho_{t,\mu}^\varepsilon(s,x)\,dx$ lie in a appropriate space to ensure that $(\mathbf P_t^\varepsilon)_{\varepsilon >0}$ is compact in $\mathcal P(\Omega_\alpha)$ any corresponding limit along a converging subsequence defines naturally a solution to the (non-linear) martingale problem related to \eqref{main}. From the well-posedness of the limit Fokker-Planck equation one eventually derives uniqueness results for the  
{\color{black}{marginal laws}} giving  in turn the uniqueness of the martingale solution associated with \eqref{main}.

Also, the regularity properties of the  
{\color{black}{time-marginals} } of the limit non-linear Fokker-Planck equation, {\color{black}(}inherited from Lemma \ref{lem_unifesti_gencase_2} {\color{black}below)}, 
\textcolor{black}{allow} to regularize the distributional interaction kernel so that the non-linear drift $\mathcal B_{\brho_{t,\mu}}$ is actually a function living in some Lebesgue space (in time and space). This permits to give a meaning to the dynamics in \eqref{main} by connecting martingale and weak solutions through classical arguments. All these properties are \textcolor{black}{gathered in}  Propositions \ref{prop_ExistenceMain} and \ref{prop_UniquenessMain_gencas}.\\

\textcolor{black}{Furthermore}, when slightly reinforcing \eqref{cond_gencase} to \eqref{COS}, we are able to obtain better estimates (namely in  suitable \textcolor{black}{Lebesgue}-Besov space - with positive regularity index for the latter) on the non-linear drift  $\mathcal B_{\brho_{t,\mu}}$ which allows to derive, in turn, a strong well-posedness result by viewing the McKean-Vlasov SDEs as a \emph{linear} SDE parametrized by the unique weak solution and applying results of \cite{kryl:rock:05} in the Brownian case and{\color{black}, as highlighted previously,} \cite{xie:zhan:20} in the pure jump one for rotationally invariant processes (see Proposition \ref{Prop_STRONG_SOL} below).

\subsection{Notations} 
For any real $\ell \in [1,\infty]$, we denote by $\ell'$ its conjugate \emph{i.e.} $\ell'$ is such that $1/\ell+1/\ell'=1$.

When there are no ambiguities, we \textcolor{black}{will} often \textcolor{black}{omit} to mention the \textcolor{black}{spatial argument} when specifying a function space, i.e. we write $L^\ell(\R^d)=L^\ell $. 

\textcolor{black}{We will use below  the notation $C$ for a generic constant that might change from line to line but only depends on known parameters appearing in \eqref{cond_gencase}}.

\mysection{Besov spaces and associated tools}\label{SEC_BESOV}
\subsection{Reminders on Besov spaces}\label{SEC_REMINDERS}
\textcolor{black}{We recall that  the space $B^\gamma_{\ell,m}\textcolor{black}{(\R^d)}, \ell,m,\gamma\in \R$ is usually defined through the dyadic Littlewood-Paley decomposition {\color{black}(see e.g. \cite{Lemarie-02} or \cite{athr:butk:mytn:20}}). Namely, it consists in the elements of  $\mathcal S'(\R^d)$ (where $\mathcal S(\R^d)$ stands for the Schwartz space)} such that the quantity
\begin{equation*}
|f|_{{ B}^\gamma_{\ell,m},{\mathbf {LP}}}:=|\mathcal F^{-1}(\phi\mathcal F(f))|_{L^\ell}+
\left\{
\begin{aligned}
& \big(\sum_{j\in\mathbb N} 2^{\gamma j m}|\mathcal F^{-1}(\phi_j\mathcal F(f))|^m_{\textcolor{black}{L^\ell}
}\big)^{\frac 1m}\,\text{for}\,1\le m<\infty,\\
&\sup_{j\in\mathbb N}\big(2^{\gamma j}|\mathcal F^{-1}(\phi_j\mathcal F(f))|_{\textcolor{black}{L^\ell}
}\big)\,\text{if}\,m=\infty,
\end{aligned}
\right.
\end{equation*}

is finite where $ \phi\in C_0^\infty\textcolor{black}{(\R^d,\R)}$ s.t. $\phi(0)\neq 0 $,  $\phi_j(y)=2^j\phi(2^j y) $ and $\mathcal F $, $\mathcal F^{-1} $ respectively denote the Fourier and inverse Fourier transform.  
 This quantity also defines a norm for which $\textcolor{black}{B_{\ell,m}^\gamma(\R^d,\R^{\textcolor{black}{k}})}$ is a Banach space - see e.g. Triebel \cite{Triebel-83a}, Theorem 2.3.3., p. 48. Instead of the Lit{\color{black}t}lewood-Paley decomposition inducing  the norm $|\cdot|_{{ B}^\gamma_{\ell,m},{\mathbf{LP}}}$,
 we will from now on use the so-called thermic characterization of Besov spaces which involves \textcolor{black}{an underlying heat kernel}. 
 The thermic characterization is presented in Section 2.6.4 of \cite{Triebel-83a} for the Gaussian and Cauchy heat kernel (corresponding respectively to $ \alpha=2$ and $\alpha=1 $). However, one can also derive {\color{black}(}from Theorem 1 in Section 2.5.1 of the same reference{\color{black})} that $f\in B_{\ell,m}^{\gamma}$ if and only if the quantity below is finite:
 \begin{align}
| f|_{\mathcal H^\gamma_{\ell,m},\textcolor{black}{\tilde \alpha}}:=&|\mathcal F^{-1}(\phi\mathcal F(f))|_{L^\ell}+\left\{
\begin{aligned}
&
\left(\int_0^1\,\frac{dv}{v}v^{(n-\gamma/{\tilde \alpha})m}|\partial^n_v\tilde p^{\tilde \alpha}(v,\cdot)*f|^m_{L^\ell}\right)^{\frac 1m}\,\text{for}\,1\le m<\infty,\\
&\sup_{v\in(0,1]}\left\{v^{(n-\gamma/{\tilde \alpha})}|\partial^n_v\tilde p^{\tilde \alpha}(v,\cdot)*f|_{L^\ell}\right\}\,\text{for}\,m=\infty,
\end{aligned}
\right.\notag\\
=:&|\mathcal F^{-1}(\phi\mathcal F(f))|_{L^\ell}+\mathcal T_{\ell,m}^{\gamma,\textcolor{black}{\tilde \alpha}}(f)
{\color{black}.}
\label{HEAT_CAR}
\end{align}
{\color{black}T}he choice of the parameter ${\tilde \alpha}\in [1,2]$ is for free, $n$ being any non-negative integer (strictly) greater than $\gamma/{\tilde \alpha}$, the function $\phi$ is as above, and $\tilde p^{\tilde \alpha}(v,\cdot)$ denoting the density function at time $v$ of the $d$-dimensional isotropic ${\tilde \alpha}$-stable process. 

\textcolor{black}{In \cite{CdRM-20}, the authors chose ${\tilde \alpha}=\alpha$ for the thermic kernel. It was a natural choice therein since, in order to establish \textit{a priori} estimates on Duhamel type representations which involve the density $p^\alpha$ of the stable driving noise (see again \eqref{main_DUHAMEL} and \eqref{main_MOLL}), the previous thermic characterization allows to benefit from stability and scaling properties  in the associated convolutions when handling directly the corresponding norms}.

%
The quantity in \eqref{HEAT_CAR} also induces a norm{\color{black}, equivalent to $|\cdot |_{{ B}^\gamma_{\ell,m},{\mathbf {LP}}}$,} and for the remaining of the paper we will denote for $f\in B_{\ell,m}^\gamma$:
\begin{equation*}
\label{NORME_BESOV_CONCRETE}
| f|_{B^\gamma_{\ell,m}}:=| f|_{\mathcal H^\gamma_{\ell,m},\textcolor{black}{\tilde \alpha}}.
\end{equation*}
We will refer to the contribution $\mathcal T_{\ell,m}^{\gamma,\textcolor{black}{\tilde \alpha}}(f)=:\mathcal T_{\ell,m}^{\gamma}(f)$ in \eqref{HEAT_CAR} as the thermic part of the norm. 
\textcolor{black}{In the current work, the only place where we will actually explicitly use this induced norm (and the choice $\tilde \alpha=\alpha $) is for the proof of Proposition \ref{PROP_APPROX}
}.

\textcolor{black}{For $k\in \{1,\cdots,d\} $, we say that 
$f=(f_1,\cdots,f_k)\in B_{\ell,m}^\gamma(\R^d,\R^{\textcolor{black}{k}})$ if for all $j\in \{1,\cdots, k\},\ f_j\in B_{\ell,m}^\gamma(\R^d) $}.

\subsection{Properties of Besov spaces}
We recall here some key technical results on Besov spaces \textcolor{black}{and} their embeddings. 
\begin{itemize}
\item[(i)] \emph{Continuous embedding.} With Lebesgue spaces (\cite[Prop. 2.1]{Sawano-18}):  
\begin{equation}
\forall 1 \le \ell \le \infty,\qquad B^0_{\ell,1}\hookrightarrow L^\ell \hookrightarrow B_{\ell,\infty}^0.\label{EMBEDDING}\tag{$\mathbf E_1$}
\end{equation}
Between Besov spaces (\textcolor{black}{see (1.1) in \cite{trie:14} and Proposition 2.2 in \cite{Sawano-18}}):
\begin{equation}\label{BesovEmbedding}\tag{$\mathbf E_2 $}
B^{s_0}_{p_0,q_0}\hookrightarrow B^{s_1}_{p_1,q_1}\,\text{for}\,p_0,p_1,q_0,q_1 \in[1,\infty],\,q_0\le q_1,\, \textcolor{black}{p_0\le p_1} \,\text{such that}\,s_0-d/p_0\geq s_1-d/p_1. 
\end{equation}
We refer to  Section{\color{black}s} 4.1 {\color{black} and} 4.2 in \cite{Sawano-18} for additional embeddings.

\item[(ii)] \emph{Young (or convolution) inequality}. Let  $\gamma \in \R$, $\ell$ and $m$ in $[1,+\infty]$. Then for any $\delta\in\mathbb R$, $\ell_1,\ell_2\in [1,\infty]$ such that $1+\ell^{-1} = \ell_1^{-1} + \ell_2^{-{\color{black}1}}$ and $m_1,m_2\in (0,\infty]$ such that $m_1^{-1} \ge 
\textcolor{black}{(m^{-1}-m_2^{-1})}\vee 0$
\begin{equation}\label{YOUNG}
| f\star g|_{B_{\ell,m}^\gamma} \le \cv | f |_{B^{\gamma - \delta}_{\ell_1,m_1}} | g |_{B^{\delta}_{\ell_2,m_2}}, \tag{$\mathbf Y $}
\end{equation}
for $\cv>0$ a universal constant depending only on $d$. The initial proof of this inequality can be found in \cite[Theorem 3]{Burenkov-90}, see also \cite{KuhSch-21} for recent  convolution inequalities for Besov and Triebel-Lizorkin spaces.
\item[(iii)] \emph{Besov norm of heat kernel (see \cite[Lemma 11]{CdRM-20})}. \textcolor{black}{For $\ell,m\in [1,\infty],\ \gamma\in \R$}, there exists $\ch:=C(\alpha,\textcolor{black}{\ell},m,\gamma,d)>0$ s.t. for all multi-index $\mathbf a \in \mathbb N^{d}$ with $ |\mathbf a|\le 1 $, and $0<v<s<\infty$:
\begin{equation}\label{SING_STABLE_HK}
\big| \p^{\mathbf a} p_{s-v}^{\alpha} \big|_{B^{\gamma}_{\ell,m}}\le \frac{\ch}{(s-v)^{\frac \gamma\alpha+\frac d{\alpha}(1-\frac 1\ell) +\frac {|\mathbf a|}\alpha}}. \tag{$\mathbf H\mathbf K $}
\end{equation}
\item[(iv)] \textcolor{black}{\emph{Duality inequality (see e.g.  \cite[Proposition 3.6]{Lemarie-02}).} 
For $\ell,m\in [1,\infty],\ \gamma\in \R$ and  $(f,g)\in B_{\ell,m}^{\gamma}\times B_{\ell',m'}^{-\gamma} $ which are also functions, it holds:}
\begin{equation}
\label{EQ_DUALITY}\tag{$\mathbf D $}
|\int_{\R^d} f(y) g(y) dy|\le |f|_{B_{\ell,m}^\gamma}|g|_{B_{\ell',m'}^{-\gamma}}. 
\end{equation}
\end{itemize}

\subsection{A convolution Lemma in Besov spaces}
As an immediate consequence of the above technical results, we have the following Lemma.
\begin{lemme}\label{LEMME_FGH} 
Let $f,\, g_1,g_2$  be in $B^{\gamma_f}_{\ell_f,m_f}$, $B^{-\gamma_f}_{\ell_{g_1},m_{g_1}}$, $L^{\ell_{g_2}}$  respectively where 
$\gamma_f$ {\color{black}is} in $[-1,0]$, $\ell_f$, $\ell_{g_1}$, $\ell_{g_2}$ in $[1,\infty]$ and $m_f$, $m_{g_1}$ in $[1,\infty]$. 

For any $ \ell, \ell_h$ and $ m, m_h$ in $[1,\infty]$ satisfying
\begin{equation}\label{COND_FOR_LEMME_FGH}
(\ell_f)^{-1} + (\ell_{g_1})^{-1}+(\ell_{g_2})^{-1} +  (\ell_h)^{-1} = 2+  \ell^{-1},\qquad m_h \le  m,\qquad m_{g_1}^{-1} \ge (1 - m_f^{-1}) \vee 0,
\end{equation}
and any $\gamma \ge 0$ s.t. $h\in B^{\gamma}_{\ell_h, m_h}$, denoting  $\mF := f \star g_1$, there exists $\ccv\ge 1$ such that 
\begin{eqnarray}\label{THE_CTR_LEMME_FGH}
\big|(\mF  g_2) \star h\big|_{B^{\gamma}_{\ell, m}} \le \ccv | f |_{B^{\gamma_f}_{\ell_f,m_f}} | g_1 |_{B^{-\gamma_f}_{\ell_{g_1},m_{g_1}}} |g_2|_{L^{\ell_{g_2}}} \big|h\big|_{B^{ \gamma}_{\ell_h, m_h}}.
\end{eqnarray}
\end{lemme}
Before proving the above result we point out that it will be crucial in the analysis of the non-linear drift term in the Duhamel formulations \eqref{main_MOLL} ({\color{black}with} mollified coefficients) and \eqref{main_DUHAMEL} ({\color{black}the related} limit equation) in order to obtain \textit{a priori} estimates. Roughly speaking{, {\color{black}in the proof of Lemma \ref{lem_unifesti_gencase_2}},} we will apply the above lemma with $\mathscr C=\mathcal B_{\brho^\varepsilon_{t,\mu}}^\varepsilon(v,\cdot) $ 
{\color{black}- corresponding to the case $f=b^\varepsilon(v,\cdot), g_1=\brho^\varepsilon_{t,\mu}(v,\cdot) $ -, $g_2=\brho^\varepsilon_{t,\mu}(v,\cdot) $ and}  $h= \nabla p_{s-v}^{\alpha}$ in the mollified setting and with the same entries without the superscripts in $\varepsilon $  
{\color{black}in} the limit one.

\begin{proof}
Apply first the Young inequality \eqref{YOUNG} to obtain
\begin{eqnarray*}
\big|(\mF  g_2) \star h\big|_{B^{ \gamma}_{ \ell,  m}} \le \cv \big|\mF  g_2\big|_{B^{0}_{\textcolor{black}{\ell_\pi}, \infty}}\big|h\big|_{B^{ \gamma}_{\ell_h, m_h}},
\end{eqnarray*}
for any $ 1+  \ell^{-1}=(\textcolor{black}{\ell_\pi})^{-1} + (\ell_h)^{-1} $ and $ m \ge m_h$. Note now that, applying first the continuous embedding \eqref{EMBEDDING}, then the H\"older inequality and next the continuous embedding \eqref{EMBEDDING} again give\textcolor{black}{s}
\begin{eqnarray*}
\big|\mF  g_2\big|_{B^{0}_{\textcolor{black}{\ell_\pi}, \infty}} &\le& C\big|\mF  g_2\big|_{L^{\textcolor{black}{\ell_\pi}}} \le  C\big|\mF\big|_{L^{\ell_\mF}}  |g_2|_{L^{\ell_{g_2}}} \le  C\big|\mF\big|_{B^{0}_{\ell_{\mF}, 1}}  |g_2|_{L^{\ell_{g_2}}},
\end{eqnarray*}
for any $(\ell_\mF)^{-1} + (\ell_{g_2})^{-1} = (\textcolor{black}{\ell_\pi})^{-1}$. Combining this last integrability condition  with the previous one arising from the convolution inequality \textcolor{black}{yields}
\begin{equation}\label{RELATION_INTER}
\textcolor{black}{ 1+  \ell^{-1}=}(\ell_\mF)^{-1} + (\ell_{g_2})^{-1} +  (\ell_h)^{-1} .
\end{equation}
Moreover, applying one more time the Young inequality \eqref{YOUNG} gives
\begin{eqnarray}
\big|\mF\big|_{B^{0}_{\ell_{\mF}, 1}} = \big| f \star g_1 \big|_{B^{0}_{\ell_{\mF}, 1}} \le \cv \textcolor{black}{| g_1 |_{B^{-\gamma_f}_{\ell_{g_1},m_{g_1}}}}| f |_{B^{\gamma_f}_{\ell_f,m_f}} ,
\end{eqnarray}
with $(\ell_{g_1})^{-1} + (\ell_f)^{-1} = 1+ (\ell_\mF)^{-1}$ and any $m_{g_1}^{-1} \ge (1 - m_f^{-1}) \vee 0$.
 Bringing together the above estimates yields
\begin{eqnarray*}
\big|(\mF  g_2) \star h\big|_{B^{\gamma}_{ \ell, m}} \le  \ccv | f |_{B^{\gamma_f}_{\ell_f,m_f}} | g_1 |_{B^{-\gamma_f}_{\ell_{g_1},m_{g_1}}} |g_2|_{L^{\ell_{g_2}}} \big|h\big|_{B^{ \gamma}_{\ell_h, m_h}},
\end{eqnarray*}
for any
\begin{equation*}
(\ell_f)^{-1} + (\ell_{g_1})^{-1}+(\ell_{g_2})^{-1} +  (\ell_h)^{-1} = 2+  \ell^{-1},\qquad m_h \le  m,\qquad m_{g_1}^{-1} \ge (1 - m_f^{-1}) \vee 0,
\end{equation*}
in view of \eqref{RELATION_INTER}.
\end{proof}

\subsection{On the embedding of $\mathcal P(\R^d) $ into Besov spaces}
We give here the following Lemma which specifies how the set of probability measures on $\R^d $ can be embedded in an intersection of Besov spaces. 
\begin{lemme}\label{lem_proba_in besov} $\mathcal P(\R^d)$ is a subset of $\cap_{\ell \ge 1}B^{-d/\ell'}_{\ell,\infty}$ and $\cap_{\ell \ge 1}B^{-d/\ell'-\epsilon}_{\ell,m}$, {\color{black}for} $\epsilon>0$, $m\in[1,\infty)$ {\color{black}and $\ell\in[1,\infty]$.}  
\end{lemme}
\begin{proof}
Let $\nu $ be an arbitrary probability measure on $\R^d$. 
For any $\ell$ in $[1,\infty)$, \textcolor{black}{$\phi\in C_0^\infty(\R^d,\R),\ \phi(0)\neq 0 $ (as in Section \ref{SEC_REMINDERS})},
\begin{align*}
|\mathcal F^{-1}(\phi \mathcal F(\nu))|_{L^\ell}=&| \mathcal F^{-1}(\phi) {\color{black}\star}\nu|_{L^\ell}=\bigg(\int_{\R^d} {\color{black}dx \ }|\mathcal F^{-1}(\phi){\color{black}\star}\nu(x)|^\ell  
\Big)^{\frac 1\ell}=\bigg(\int_{\R^d} {\color{black}dx \ }\bigg|\int_{\R^d}\mathcal F^{-1}(\phi)(x-y)\nu(dy)\Big|^\ell \bigg)^{\frac 1\ell}\\
\le&  \bigg(\int_{\R^d} {\color{black}dx \ }\int_{\R^d}\Big|\mathcal F^{-1}(\phi) (x-y)\Big|^\ell\nu(dy) 
\bigg)^{\frac 1\ell}=|\mathcal F^{-1}(\phi)|_{L^\ell}<+\infty,
\end{align*}
using  {\color{black}classical convexity inequalities} and the Fubini theorem as well as the fact that, since $\phi\in C_0^\infty$, $\mathcal F^{-1} (\phi)\in \mathcal S $ for the last inequality. If now $\ell=\infty $, one readily gets:
$$|\mathcal F^{-1}(\phi \mathcal F(\nu))|_{L^\infty}= |\mathcal F^{-1}(\phi) {\color{black}\star} \nu|_{L^\infty}\le |\mathcal F^{-1}(\phi)|_{L^\infty}<+\infty. $$
Let now consider $\gamma\le 0 $ and, with the terminology of Section \ref{SEC_REMINDERS}, the thermic part $\mathcal T_{\ell,m}^\gamma(\nu) $ of $|\nu|_{B^{\gamma}_{\ell,m}}$. For this choice, from \eqref{HEAT_CAR}, we can take $n=0$ and write then {\color{black}-} similarly to the previous computations (replacing actually $\mathcal F^{-1}(\phi) $ by $\tilde p^\alpha(v,\cdot) $) and according to \textcolor{black}{\eqref{SING_STABLE_HK} {\color{black}and }\eqref{EMBEDDING}} {\color{black}-} for any $\ell \in [1,+\infty]$,
\[
\mathcal T_{\ell,\infty}^{\gamma}(\nu):=\sup_{v\in (0,1]}v^{-\frac \gamma\alpha}| \tilde p^\alpha(v,\cdot){\color{black}\star}\nu|_{L^\ell}\le \sup_{v\in (0,1]} v^{-\frac \gamma\alpha} | \tilde p^\alpha(v,\cdot)|_{L^\ell}\le C\sup_{v\in (0,1]}v^{\frac 1\alpha[-\gamma-\frac d{ \ell'}]}<+\infty,
\]
whenever $0\textcolor{black}{\le} -\gamma-d/\ell' \iff \gamma\textcolor{black}{\le} -d/\ell' $. On the other hand, for $m\in [1,+\infty) $,
$\mathcal T_{\textcolor{black}{\ell,m}}^{\gamma}(\nu)<+\infty$ whenever $-\gamma-d/\ell'>0$. This gives the claim.
\end{proof}

\subsection{\textcolor{black}{Some useful weighted spaces}}
In the following we will also frequently use weighted Lebesgue spaces in the time variable with arguments valued in Besov spaces. Namely, for $\gamma\in \R_+ $, $ \ell,  m,  r \in [1,\infty]$ we introduce, for $0\le t<  S\le T$,  the Bochner space:
\begin{equation}\label{DEF_L_B_WEIGHTED}
\Lw^r((t,S],B_{\ell,m}^\gamma):=\Bigg\{f:s\in[t,S]\mapsto f(s,\cdot)\in B^{ \gamma}_{ \ell, m}\,\text{measurable and s.t.}\, \int_t^S {\color{black}ds} \ {\color{black}(S-s)^{-\frac r{\alpha}}} | f(s,\cdot)|_{B^{ \gamma}_{ \ell, m}}^{ r}<\infty\Bigg\},
\end{equation}
equipped with its natural norm:
\[
|f|_{\Lw^r((t,S],B_{\ell,m}^\gamma)}:=\left(\int_t^S {\color{black}ds} \ {\color{black}(S-s)^{\textcolor{black}{-\frac r\alpha}}} | f(s,\cdot)|_{B^{ \gamma}_{ \ell, m}}^{ r}\right)^{\textcolor{black}{\frac 1r}}. 
\]
{\color{black}This space} turns out to be a Banach space for the chosen weight {\color{black}$s\in[t,S)\mapsto (S-s)^{-r/\alpha}$ } as well as for a rather large class of weights (see e.g. \cite[Chapter 1]{HyNeVeWe-16}). {\color{black}It can be pointed out that the above weight depends strongly on the right hand limit of the time interval. The choice of making this dependency implicit {\color{black}in the notation of the norm $|\cdot|_{\Lw^r((t,S],B_{\ell,m}^\gamma)}$} \textcolor{black}{is performed in order to} lighten forthcoming computations.}  


\mysection{Proof of Theorem \ref{THM_GEN}}\label{SEC_PROOF_WEAK}

In the whole section, we might assume for some results that {\color{black}the time-horizon} $T>0 $ is \emph{small enough}, \emph{i.e.} smaller than some positive time depending only on the \emph{parameters} $\alpha,\beta,p,q,r,d$. This is somehow required to derive  the a priori estimates {\color{black}(see Lemma \ref{lem_unifesti_gencase_2})} needed to obtain well-posedness of the non-linear Fokker-Planck equation. 
 This leads to an a priori restriction of the time interval $[t,T]$ which is nevertheless just provisional: as \textit{\textcolor{black}{a priori}} estimates will require nothing else on the initial condition but  being a probability measure, and will depend smoothly on the time interval size, they can then be iterated up to an arbitrary time-horizon $T$.


\subsection{The corresponding non-linear Fokker-Planck equation} \label{sec_nl_FK}
We here prove the following Proposition which roughly 
states that for any initial condition being a probability measure, the non-linear Fokker-Planck equation admits a unique solution in a suitable Lebesgue-Besov space, \textcolor{black}{for a certain range of parameters}. More specifically, the range of the parameters allows to quantify: on the one hand, how much the Besov norm of the solution is integrable; on the other hand, how much the local regularity of the Besov norm can be improved in term of the gap $\Gamma$ in  \eqref{cond_gencase}:
\begin{equation}
\Gamma:=  \beta -\Big(  1 -\alpha+ \frac dp +\frac \alpha r\Big)>0. \label{GAP} \tag{\textbf{G}}
\end{equation}

\begin{prop}\label{WP_FK} Assume that the \emph{parameters} are such that \eqref{cond_gencase} holds. Then, for any  $(t,\mu)$  in $[0,T] \times \mathcal P(\R^d)$, the non-linear Fokker-Planck equation \eqref{NL_PDE_FK} admits a solution  
which is unique among all the distributional solutions lying in $ L^{\overline{r}}(B^{-\beta\textcolor{black}{+\vartheta \Gamma}}_{p',q'})$ where
\begin{eqnarray}
\overline{r}\in \Bigg[r',\bigg(\frac 1\alpha\left(-\beta + \textcolor{black}{\vartheta \Gamma}+\frac d{p}\right)\bigg)^{-1}\Bigg), \ \textcolor{black}{\forall \vartheta\in [0,1)}\label{cond_time_gencase}.
\end{eqnarray}

Moreover, \textcolor{black}{for all $s\in [t,T], \brho_{t,\mu}(s,\cdot)$ belongs to $\mathcal P(\R^d)$}. \textcolor{black}{Eventually,  for a.e. $s$ in $(t,T]$, $\brho_{t,\mu}(s,\cdot)$ is absolutely continuous w.r.t. the Lebesgue measure  
and satisfies the Duhamel representation \eqref{main_DUHAMEL}}.
\end{prop}


The preliminary step to prove the previous well-posedness result for the non-linear Fokker-Planck equation consists in establishing suitable \textit{a priori} estimates, in the corresponding function space, for the mollified equation \eqref{main_MOLL}. This is the point of the next Lemma in which we use what we call a \textit{dequadrification} approach (see eq. \eqref{FOR_FURTHER_REF_2} below and explanations there).  We then proceed with a Lemma that gives that for any decreasing sequence $(\varepsilon_k)_k $ \textcolor{black}{going to 0},  $(\brho_{t,\mu}^{\varepsilon_k})_{k} $ is a Cauchy sequence in some appropriate function space (Lemma \ref{FIRST_STAB} below). We then conclude that the limit of the previous convergent sequence  solves the non-linear Fokker Planck equation \eqref{NL_PDE_FK} (Lemma \ref{lem_ex_duha_gencase}) in a distributional sense. We eventually prove, through a rigorous derivation of the Duhamel formulation \eqref{main_DUHAMEL}, that this solution is unique among all the solutions lying in the corresponding function space (Lemma \ref{lem_unique_FK_gencase}). This procedure completes the proof of Proposition \ref{WP_FK}. 

In the following, it is always assumed that the \emph{parameters} are such that \eqref{cond_gencase} holds.

\begin{lemme}\label{lem_unifesti_gencase_2} 
For any $(t,\mu)$ in $[0,T] \times \mathcal P(\R^d)$, for any $\vartheta\in [0,1)$, any $\bar r$ satisfying \textcolor{black}{\eqref{cond_time_gencase}}
there exist $C:=C(d,\alpha,r,\beta,p, \vartheta)>0$ and $\theta:=\theta(d,\alpha,\beta,p,\vartheta)>0$ such that it holds

\begin{eqnarray*}
\forall \varepsilon >0,\quad \int_t^T ds |\brho^\varepsilon_{t,\mu}(s,\cdot)|_{B^{-\beta+ \vartheta \Gamma}_{p',q'}}^{\overline{r}} \le C(T-t)^{\theta},
\end{eqnarray*}
and for any $t<S<T$ with the notations of \eqref{DEF_L_B_WEIGHTED}, for $\check  r'\in \Bigg[r',\bigg(\frac 1\alpha\left(-\beta+\vartheta\Gamma + \frac d{p}+1\right)\bigg)^{-1}\Bigg)$:
\begin{eqnarray}\label{LEMME_2_WITHOUT_SING_2}
\forall \varepsilon >0,\quad |\brho^\varepsilon_{t,\mu}|_{{\rm{L}}_{\rm{w}}^{\check r'}\big((t,S],B_{p',q'}^{-\beta+\vartheta\Gamma}\big)}
 \le C
 (S-t)^{ \frac{\textcolor{black}{\check \delta}}{\textcolor{black}{\check r'}}},
\end{eqnarray}
for some $\textcolor{black}{\check \delta}:={\check \delta(\check r')}>0 $.

\end{lemme}

\begin{proof}[Proof of Lemma \ref{lem_unifesti_gencase_2}]
{\color{black}
For the mollified equation, from \textcolor{black}{the Duhamel representation} \eqref{main_MOLL}, we have for any $\ell,m \in [1,\infty]$,
\begin{eqnarray}
|\brho^\varepsilon_{t,\mu}(s,\cdot)|_{B^{-\beta + \vartheta \Gamma}_{\ell,m}} \le |p^\alpha_{s-t}\star \mu|_{B^{-\beta + \vartheta \Gamma}_{\ell,m}} 
+ \int_t^s dv \bigg|\Big[\{\mathcal B_{\brho^\varepsilon_{t,\mu}}^\varepsilon(v,\cdot) \brho^\varepsilon_{t,\mu}(v,\cdot)\} \star \nabla p_{s-v}^{\alpha}\Big] (\cdot)\bigg|_{B^{-\beta + \vartheta \Gamma}_{\ell,m}}.
\label{AFTER_ONE_TRIANG_INEQ_2ppp}
\end{eqnarray}
At this point, the strategy consists in handling the term
\begin{equation}\label{TOAPPLYLEMMA4}
\bigg|\Big[\{\mathcal B_{\brho^\varepsilon_{t,\mu}}^\varepsilon(v,\cdot) \brho^\varepsilon_{t,\mu}(v,\cdot)\} \star \nabla p_{s-v}^{\alpha}\Big] (\cdot)\bigg|_{B^{-\beta + \vartheta \Gamma}_{\ell,m}},
\end{equation}
by applying Lemma \ref{LEMME_FGH} with $\mathscr C=\mathcal B_{\brho^\varepsilon_{t,\mu}}^\varepsilon(v,\cdot)$ (so that $f=b^\varepsilon(v,\cdot),\, g_1=g_2=\brho^\varepsilon_{t,\mu}(v,\cdot)$) and $h= \nabla p_{s-v}^{\alpha}$. This gives,
\begin{eqnarray*}
\big|(\mathcal B_{\brho^\varepsilon_{t,\mu}}^\varepsilon (v,\cdot) \brho^\varepsilon_{t,\mu}(v,\cdot) ) \star \nabla p^\alpha_{s-v}\big|_{B^{-\beta + \vartheta \Gamma}_{\ell,m}} \le \ccv|  b ^\varepsilon(v,\cdot) |_{B^{\gamma_f}_{\ell_f,m_f}} | \brho^\varepsilon_{t,\mu}(v,\cdot) |_{B^{-\gamma_f}_{\ell_{g_1},m_{g_1}}}  |\brho^\varepsilon_{t,\mu}(v,\cdot)|_{L^{\ell_{g_2}}} \big|\nabla p^\alpha_{s-v}\big|_{B^{-\beta + \vartheta \Gamma}_{\ell_h, m_h}}.
\end{eqnarray*}
provided
\begin{equation}\label{COND_FOR_LEMME_FGH_2}
(\ell_f)^{-1} + (\ell_{g_1})^{-1}+(\ell_{g_2})^{-1} +  (\ell_h)^{-1} = 2+  \ell^{-1},\qquad m_h \le  m,\qquad m_{g_1}^{-1} \ge (1 - m_f^{-1}) \vee 0.
\end{equation}
 A way to get rid of the somehow \textit{quadratic} dependence in $\brho^\varepsilon_{t,\mu} $ in the above term consists in choosing $\ell_{g_2}=1$ to exploit that we are actually dealing with probability densities: $ |\brho^\varepsilon_{t,\mu}(v,\cdot)|_{L^1}=1$. Since we also want some compatibility between the Besov norm with which we will estimate $\brho^\varepsilon_{t,\mu}(s,\cdot) $ in the left hand side of \eqref{AFTER_ONE_TRIANG_INEQ_2ppp} and the one associated with the same quantity $\brho^\varepsilon_{t,\mu}(v,\cdot)$ in the right hand side of the same equation, 
this also imposes $\ell_{g_1} =\ell $, $m_{g_1}=m $. In particular, these choices therefore yield from \eqref{COND_FOR_LEMME_FGH_2} that $\ell_f^{-1}+\ell_h^{-1}=1 $. Since $\ell_f=p $, \textcolor{black}{corresponding to the} initial integrability index of the \textit{singular} kernel $b^\varepsilon$, this yields $\ell_h=p' $ recalling $p^{-1}+(p')^{-1}=1  $. This will contribute to the time-singularity associated with the Besov norm of the gradient of the stable heat kernel, see \eqref{SING_STABLE_HK}. 

The other parameters in \eqref{COND_FOR_LEMME_FGH_2} also naturally follow from the considered assumptions at hand. Namely, since the interaction kernel $f=b^\varepsilon(v,\cdot)$ \textcolor{black}{lies in} the Besov space $B_{p,q}^\beta$ uniformly in $\varepsilon>0$, this suggests to choose $\gamma_f=\beta$. Then, from the condition \eqref{COND_FOR_LEMME_FGH_2}, we get $m_{g_1}=m,\ m^{-1}\ge (1-m_f^{-1})\vee 0 $. Since $m_f=q $ (assumption on $b^\varepsilon $), $m=q'=m_{g_1}$ is a natural choice. Let us now fix $m_h$ so  that $m\ge m_h $. This again yields the natural choice $m_h=1 $, which gives the lowest singularity exponent for the stable heat kernel (see \eqref{SING_STABLE_HK}). When doing so, we thus obtain
\begin{eqnarray}
|\brho^\varepsilon_{t,\mu}(s,\cdot)|_{B^{-\beta + \vartheta \Gamma}_{\ell,q'}} \le  |p^\alpha_{s-t}\star \mu|_{B^{\gamma}_{\ell,q'}} 
+ \ccv \int_t^s dv  |  b ^\varepsilon(v,\cdot) |_{B^{\beta}_{p,q}} | \brho^\varepsilon_{t,\mu}(v,\cdot) |_{B^{-\beta}_{\ell,q'}}  \big|\nabla p^\alpha_{s-v}\big|_{B^{-\beta + \vartheta \Gamma}_{p', 1}}.
\label{FOR_STRONG_UNIQUENESS}
\end{eqnarray}
The only parameter that remains to be fixed is the integrability index $\ell$. A natural choice is given by $\ell =p'$ (this can indeed be seen from the proof of Lemma \ref{LEMME_FGH} where the relations $(\textcolor{black}{\ell_\pi})^{-1} + (\ell_h)^{-1} = 1+  \ell^{-1}$ and $(\ell_\mF)^{-1} + (\ell_{g_2})^{-1} = (\textcolor{black}{\ell_\pi})^{-1}$ naturally yields to \textcolor{black}{$\ell_\pi=1,\ell_{\mF}=\infty $}, which gives in turn that $\ell=\ell_h=\ell_{g_1}=p' $). Eventually, to ensure the compatibility between the Besov norms on $\brho_{t,\mu}^\varepsilon$ appearing in the above right and left hand sides, we use the embedding \eqref{BesovEmbedding} which ensures that $B^{-\beta+\vartheta \Gamma}_{p',q'} \hookrightarrow B^{-\beta}_{p',q'}$, as $\vartheta \Gamma \ge 0$. We end up with the following estimate:
\begin{eqnarray*}
|\brho^\varepsilon_{t,\mu}(s,\cdot)|_{B^{-\beta+\vartheta\Gamma}_{p',q'}} \le  |p^\alpha_{s-t}\star \mu|_{B^{-\beta+\vartheta\Gamma}_{p',q'}} 
+ \ccv C\int_t^s dv  |  b ^\varepsilon(v,\cdot) |_{B^{\beta}_{p,q}} | \brho^\varepsilon_{t,\mu}(v,\cdot) |_{B^{-\beta+\vartheta\Gamma}_{p',q'}}  \big|\nabla p^\alpha_{s-v}\big|_{B^{-\beta+\vartheta\Gamma}_{p', 1}},
\label{FOR_FURTHER_REF_2}
\end{eqnarray*}
 i.e.  we applied  Lemma \ref{LEMME_FGH}, with $\ell=p',m=q',\gamma=-\beta+\vartheta\Gamma, \ \ell_f=p,m_f=q,\gamma_f=\beta,\ \textcolor{black}{\ell_{g_2}=1, \ell_{g_1}=p'},m_{g_2}=q',\ \ell_h=p',m_h=1$, on the second term in the right hand side of  \eqref{AFTER_ONE_TRIANG_INEQ_2ppp} and used eventually \eqref{BesovEmbedding}.
 }
 
Thus, using \eqref{SING_STABLE_HK} and applying the $L^1:L^{r}-L^{r'}$ - H\"older inequality in time, we obtain: 
\begin{eqnarray}\label{BEFORE_CHOICE_2}
|\brho^\varepsilon_{t,\mu}(s,\cdot)|_{B^{-\beta+\vartheta\Gamma}_{p',q'}} \le  |p^\alpha_{s-t}\star \mu|_{B^{-\beta+\vartheta\Gamma}_{p',q'}} + C| b^\varepsilon|_{L^r(B^{\beta}_{p,q})} \bigg(\int_t^s \frac{dv}{(s-v)^{\left(\frac{-\beta+\vartheta\Gamma}{\alpha} + \frac d{\alpha}\frac 1p + \frac 1\alpha\right)r'}}|\brho^\varepsilon_{t,\mu}(v,\cdot)|_{B^{-\beta+\vartheta\Gamma}_{p',q'}}^{r'}\bigg)^{\frac {1}{r'}},
\end{eqnarray}
which actually gives an integrable singularity in time as soon  as 

\begin{equation*}\label{COND_GAMMAELLM_2}
\left(\frac{-\beta+\vartheta\Gamma}{\alpha} + \frac d{\alpha}\frac 1p + \frac 1\alpha\right)r'<1\Longleftrightarrow \beta -\vartheta\Gamma> \frac dp + 1+\frac \alpha r -\alpha\iff (1-\vartheta)\Gamma>0.
\end{equation*}
Note that this condition is indeed fulfilled from \eqref{cond_gencase} and \eqref{GAP} for any $\vartheta\in [0,1) $. Write now from the Young inequality \eqref{YOUNG} and \textcolor{black}{\eqref{SING_STABLE_HK}}:
\begin{align}
|p^\alpha_{s-t}\star \mu|_{B^{-\beta+\vartheta\Gamma}_{p',q'}}\textcolor{black}{\le \textcolor{black}{\cv} |p_{s-t}^\alpha|_{B_{p',q'}^{-\beta+\vartheta\Gamma}} |\mu|_{B_{1,\infty}^0}}\le 
\frac{C}{(s-t)^{\big(\frac{-\beta+\vartheta\Gamma}{\alpha}+\frac d{\alpha  p}\big)}},\label{FIRST_TERM_ROOTS_2}
\end{align}
\textcolor{black}{using as well Lemma \ref{lem_proba_in besov}, which precisely gives that any $\mu \in \mathcal P(\R^d) \in B_{1,\infty}^0 $, for the last inequality}. 
The above controls precisely allow to specify the range for the parameter $\bar r$ \textcolor{black}{for} which we will be able to estimate the $L^{\bar r}(B_{p',q'}^{\beta}) $ norm of the density. Namely, one gets:

\begin{align}\label{SING_CI_2}
|p^\alpha_{s-t}\star \mu|_{B^{-\beta+\vartheta\Gamma}_{p',q'}}^{\bar r}\le 
\frac{C}{(s-t)^{\big(\frac{-\beta+\vartheta\Gamma}{\alpha}+\frac d{\alpha  p}\big)\bar r}},
\end{align}
which gives an integrable singularity in time provided the integrability condition \textcolor{black}{\eqref{cond_time_gencase}} holds for $\bar r $.

Taking now both sides of \eqref{BEFORE_CHOICE_2} to the exponent $\bar r$, we then get from usual convexity inequalities 
and integrating then in $s\in [t,T]$ that, 
\begin{eqnarray}
&&\int_t^T ds |\brho^\varepsilon_{t,\mu}(s,\cdot)|_{B^{-\beta+\vartheta \Gamma}_{p',q'}}^{\bar r
}\label{AFTER_CONV_2} \\
&\le& 2^{\bar r-1}\int_t^T ds\Bigg\{|p^\alpha_{s-t}\star \mu|_{B^{-\beta+\vartheta \Gamma}_{p',q'}}^{\bar r}+C| b^\varepsilon|_{L^r(B^{\beta}_{p,q})}^{\bar r} \bigg(\int_t^s \frac{dv}{(s-v)^{\left(\frac{-\beta+\vartheta \Gamma}{\alpha} + \frac d{p\alpha} + \frac 1\alpha\right)r'}}|\brho^\varepsilon_{t,\mu}(v,\cdot)|_{B^{-\beta+\vartheta \Gamma}_{p',q'}}^{r'}\bigg)^{\frac {\overline{r}}{r'}}
\Bigg\}.\notag
\end{eqnarray}
Since we are considering compact time intervals, we can assume w.l.o.g. that $\bar r>r'$ so that the H\"older inequality  
{\color{black} $L^{1}:L^{\overline{r}/r'}-L^{(1-r'/\overline{r})^{-1}}$  applies and combined with a splitting of \textcolor{black}{the  $(r')^{\rm th}$ power} of the time singularity as $r'(r'/\overline{r})+r'(1-r'/\overline{r})$, \textcolor{black}{it} yields to}
 \begin{eqnarray*}
 &&\int_t^T ds |\brho^\varepsilon_{t,\mu}(s,\cdot)|_{B^{-\beta+\vartheta \Gamma}_{p',q'}}^{\bar r}
\le 2^{\bar r-1}\Bigg\{ \int_t^T ds |p^\alpha_{s-t}\star \mu|_{B^{-\beta+\vartheta \Gamma}_{p',q'}}^{\bar r}  \notag\\
&&+ C| b^\varepsilon|_{L^r(B^{\beta}_{p,q})}^{\bar r} \int_t^T ds \int_t^s \frac{dv}{(s-v)^{\left(\frac{-\beta+\vartheta \Gamma}{\alpha} + \frac d{p\alpha} + \frac 1\alpha\right)r'}} |\brho^\varepsilon_{t,\mu}(v,\cdot)|_{B^{-\beta+\vartheta \Gamma}_{p',q'}}^{\bar r}\bigg(\int_t^s\frac{dv}{(s-v)^{\left(\frac{-\beta+\vartheta \Gamma}{\alpha} + \frac d{p\alpha} + \frac 1\alpha\right)r'}}\bigg)^{\frac{\bar r}{r'}-1}\Bigg\}.\notag
\end{eqnarray*}
Therefore, setting 
\begin{equation}\label{DEF_DELTA}
\delta:=1-\Big(\frac{-\beta+\vartheta \Gamma}{\alpha} + \frac d{p\alpha} + \frac 1\alpha\Big)r'>0,
\end{equation}
from condition{\color{black}s} \eqref{cond_gencase} and \eqref{GAP}, using then the Fubini theorem, one gets:
\begin{align}
\int_t^T ds |\brho^\varepsilon_{t,\mu}(s,\cdot)|_{B^{-\beta+\vartheta \Gamma}_{p',q'}}^{\bar r} 
\le& 2^{\bar r-1}\Bigg\{ \int_t^T ds |p^\alpha_{s-t}\star \mu|_{B^{-\beta+\vartheta \Gamma}_{p',q'}}^{\bar r}  \notag\\
 &+C(T-t)^{\delta (\frac{\bar r}{r'}-1)}| b^\varepsilon|_{L^r(B^{\beta}_{p,q})}^{\overline{r}}\int_t^T dv |\brho^\varepsilon_{t,\mu}(v,\cdot)|_{B^{-\beta+\vartheta \Gamma}_{p',q'}}^{\overline{r}}\int_v^T \frac{ds}{(s-v)^{\left(\frac{-\beta+\vartheta \Gamma}{\alpha} + \frac d{p\alpha} + \frac 1\alpha\right)r'}}\Bigg\}\notag\\
\le &2^{\bar r-1}\Bigg\{ \int_t^T ds |p^\alpha_{s-t}\star \mu|_{B^{-\beta+\vartheta \Gamma}_{p',q'}}^{\bar r}+C(T-t)^{\delta \frac{\bar r}{r'}}| b^\varepsilon|_{L^r(B^{\beta}_{p,q})}^{\overline{r}}\int_t^T dv |\brho^\varepsilon_{t,\mu}(v,\cdot)|_{B^{-\beta+\vartheta \Gamma}_{p',q'}}^{\overline{r}}\Bigg\}.\label{AFTER_FUB}
\end{align}
Recalling now \eqref{SING_CI_2} and \textcolor{black}{\eqref{cond_time_gencase}}{\color{black}, as well as the uniform control of $b^\epsilon$ given in Proposition \ref{PROP_APPROX},} we get that there exists $\theta>0 $ s.t. for $T$ small enough:
$$\bigg(\int_t^T ds |\brho^\varepsilon_{t,\mu}(s,\cdot)|_{B^{-\beta+\vartheta \Gamma}_{p',q'}}^{\bar r}\bigg)^{\frac 1{\bar r}} \le C(T-t)^\theta.$$

Let us turn to the proof of \eqref{LEMME_2_WITHOUT_SING_2}. Restart from \eqref{BEFORE_CHOICE_2} taking  both sides to the  exponent $\textcolor{black}{\check r'}$. Apply a convexity inequality to distribute the exponent on all the terms of the r.h.s., multiply by $(S-s)^{-\frac{\textcolor{black}{\check r'}}\alpha} $ for $t<S\le T$ and integrate {\color{black}the resulting expression }on $(t,S] $, we obtain:
\begin{eqnarray*}
&&\int_t^S \frac{ds}{(S-s)^{\frac{\textcolor{black}{\check r'}}{\alpha} }} |\brho^\varepsilon_{t,\mu}(s,\cdot)|_{B^{-\beta+\vartheta \Gamma}_{p',q'}}^{\textcolor{black}{\check r'}}\\
&\le& C_{\textcolor{black}{\check r'}}\bigg\{\int_t^S \frac{ds}{(S-s)^{\frac{\textcolor{black}{\check r'}}{\alpha} }} |p^\alpha_{s-t}\star \mu|_{B^{-\beta+\vartheta \Gamma}_{p',q'}}^{\check r'}\\
&& + | b^{{\color{black}\epsilon}}|_{L^{\textcolor{black}{\check r}}(B^{\beta}_{p,q})}^{\check r'} \int_t^S  \frac{ds}{(S-s)^{\frac{\check r'}{\alpha}}}   \int_t^s \frac{dv}{(s-v)^{\left(\frac{-\beta+\vartheta \Gamma}{\alpha} + \frac d{p\alpha} + \frac 1\alpha\right)\check r'}}|\brho^\varepsilon_{t,\mu}(v,\cdot)|_{B^{-\beta+\vartheta \Gamma}_{p',q'}}^{\check r'}\bigg\}\\
&\le&C_{\check r'}\bigg\{ \int_t^S \frac{ds}{(S-s)^{\frac{\check r'}{\alpha}}} \frac{1}{(s-t)^{(\frac{-\beta+\vartheta \Gamma}\alpha+\frac d{\alpha p})\check r'}} \\
&&+ | b^{{\color{black}\epsilon}}|_{L^{\textcolor{black}{r}}(B^{\beta}_{p,q})}^{\check r'} \int_t^S dv|\brho^\varepsilon_{t,\mu}(v,\cdot)|_{B^{-\beta+\vartheta \Gamma}_{p',q'}}^{\textcolor{black}{\check r'}}  \int_v^S  \frac{ds}{(S-s)^{\frac{\textcolor{black}{\check r'}}{\alpha}}(s-v)^{\left(\frac{-\beta+\vartheta \Gamma}{\alpha} + \frac d{p\alpha} + \frac 1\alpha\right) \check r'}}\bigg\},
\end{eqnarray*}
using the Fubini theorem and \eqref{FIRST_TERM_ROOTS_2} to derive the last inequality. \textcolor{black}{Since
$$ \left(\frac{-\beta+\vartheta\Gamma}{\alpha} + \frac d{\alpha}\frac 1p + \frac 1\alpha\right)\textcolor{black}{\check r'}<1\Longleftrightarrow \beta -\vartheta\Gamma> \frac dp + 1+\frac \alpha {\check r} -\alpha\iff (1-\vartheta)\Gamma>\alpha(\frac 1{\check r}-\frac 1r),
$$
it is clear that for a fixed $\vartheta\in [0,1) $, $\check r $ can be taken large enough in order to have an integrable singularity in the previous expression. Namely, 
\begin{align}\label{DEF_DELTA_CHECK}
\check r \in \Bigg(\Big(\frac{(1-\vartheta)\Gamma}{\alpha}+\frac 1r\Big)^{-1},r\Bigg]\Longrightarrow \check \delta:=1-\left(\frac{-\beta+\vartheta\Gamma}{\alpha} + \frac d{\alpha}\frac 1p + \frac 1\alpha\right)\textcolor{black}{\check r'}>0. 
\end{align}
}

From \eqref{DEF_DELTA_CHECK} we get that:
\begin{eqnarray*}
\int_t^S \frac{ds}{(S-s)^{\frac{\textcolor{black}{\check r'}}{\alpha} }} |\brho^\varepsilon_{t,\mu}(s,\cdot)|_{B^{-\beta+\vartheta \Gamma}_{p',q'}}^{\textcolor{black}{\check r'}}
&\le &C(S-t)^{\textcolor{black}{\check \delta}}  + C|b^{{\color{black}\varepsilon}}|_{L^r(B^{\beta}_{p,q})}^{\textcolor{black}{\check r'}} \int_t^S\frac{dv}{(S-v)^{\textcolor{black}{\frac{\check r'}{\alpha} -\check \delta}}}|\brho^\varepsilon_{t,\mu}(v,\cdot)|_{B^{-\beta+\vartheta \Gamma}_{p',q'}}^{\textcolor{black}{\check r'}} \\
&\le &C(S-t)^{\textcolor{black}{\check \delta}}  + C|b^{{\color{black}\varepsilon}}|_{L^r(B^{\beta}_{p,q})}^{\textcolor{black}{\check r'}} (S-t)^{\textcolor{black}{\check \delta}}\int_t^S\frac{dv}{(S-v)^{\frac{\textcolor{black}{\check r'}}{\alpha}}}|\brho^\varepsilon_{t,\mu}(v,\cdot)|_{B^{-\beta+\vartheta \Gamma}_{p',q'}}^{\textcolor{black}{\check r'}} \\
&\le & 2C(S-t)^{\textcolor{black}{\check  \delta}},
\end{eqnarray*}
as soon as $S-t$ is small enough to have  $C|b|_{L^r(B^{\beta}_{p,q})}^{\textcolor{black}{\check r'}}(S-t)^{\textcolor{black}{\check \delta}}\le 1/2 $ {\color{black}($\Rightarrow$ $C|b^{\color{black}\varepsilon}|_{L^r(B^{\beta}_{p,q})}^{\textcolor{black}{\check r'}}\textcolor{black}{(S-t)^{\check \delta}}\le 1/2 $)} for the last but one inequality. {\color{black} Identifying the norm $|\cdot|_{\Lw^{\check r'}((t,S],B^{-\beta+\vartheta \Gamma}_{p',q'})}$ on the left-hand side, t}his gives \eqref{LEMME_2_WITHOUT_SING_2} and concludes the proof of the Lemma.
\end{proof}

\begin{lemme}[Stability Lemma]\label{FIRST_STAB} For any initial condition $(t,\mu)$ in $[0,T] \times \mathcal P(\R^d)$, for any decreasing sequence $(\varepsilon_k)_{k\ge 1} $ s.t. $\varepsilon_k \rightarrow  0 $, $\big(\brho^{\varepsilon_k}_{t,\mu}\big)_{k\ge 1} $ is a Cauchy sequence in \textcolor{black}{$L^{\bar r}(B_{p',q'}^{-\beta+ \vartheta \Gamma})\cap L^\infty\big((t,T],L^1\big)$}, for any $\vartheta\in \textcolor{black}{(0,1)} $ with $\bar r  $ as in \textcolor{black}{\eqref{cond_time_gencase}} and $\Gamma$ as in \eqref{GAP}. In particular, there exists $\brho_{t,\mu}$ in  $ L^{\bar r}(B_{p',q'}^{-\beta+ \vartheta \Gamma}) \cap \textcolor{black}{L^\infty\big((t,T],L^1\big)}$ such that
\begin{equation}\label{STRONG_CONV}
 |\brho^{\varepsilon_k}_{t,\mu}-\brho_{t,\mu}|_{L^{\bar r}(B_{p',q'}^{-\beta+\vartheta \Gamma})}+  \textcolor{black}{\sup_{s\in (t,T]}|(\brho^{\varepsilon_k}_{t,\mu}-\brho_{t,\mu})(s,\cdot)|_{L^1} }\underset{k}{\longrightarrow} 0.
\end{equation}
\end{lemme}

\begin{proof}
Fix $k,j\in \mathbb N $ meant to be large. Assume w.l.o.g. that $k\ge j $.  From \eqref{main_MOLL}, for any $m\in \{k,j\}$, $ \brho^{\varepsilon_m}_t $ solves:
\begin{align*}
\brho^{\varepsilon_m}_{t,\mu}(s,y)  =  p^\alpha_{s-t}\star \mu(y)
-  \int_t^s dv \Big[\{\mathcal B_{\brho^{\varepsilon_m}_{t,\mu}}^{\varepsilon_m}(v,\cdot) \brho^{\varepsilon_m}_{t,\mu}(v,\cdot)\} \star \nabla p_{s-v}^{\alpha}\Big] (y),
\end{align*}
where we recall that, from Lemma \ref{lem_unifesti_gencase_2}, $\brho_{t,\mu}^{\varepsilon_m}\in L^{\bar r}(B_{p',q'}^{-\beta+\vartheta\Gamma}) $ with $\bar r $ as in \textcolor{black}{\eqref{cond_time_gencase}}. We have
\begin{eqnarray*}
 \brho^{\varepsilon_k}_{t,\mu}(s,y) -\brho^{\varepsilon_j}_{t,\mu}(s,y) &=&- \int_t^s dv \Big[\{{\mathcal B}_{\brho^{\varepsilon_k}_{t,\mu}}^{\varepsilon_k}(v,\cdot) \brho^{\varepsilon_k}_{t,\mu}(v,\cdot)-{\mathcal B}_{\brho^{\varepsilon_j}_{t,\mu}
}^{\varepsilon_j}(v,\cdot) \brho^{\varepsilon_j}_{t,\mu}(v,\cdot)\} \star \nabla p_{s-v}^{\alpha}\Big] (y).
\end{eqnarray*}
To show that $|\brho^{\varepsilon_k}_{t,\mu} -\brho^{\varepsilon_j}_{t,\mu}|_{L^{\bar r}(B_{p',q'}^{-\beta+\vartheta \Gamma})}$ is small, we will proceed 
somehow \textit{similarly} to the proof of Lemma \ref{lem_unifesti_gencase_2}. The convolution inequalities used therein (see e.g. Lemma \ref{LEMME_FGH} 
{\color{black}yielding to} \eqref{FOR_FURTHER_REF_2}) then naturally lead to estimate the $L^1$ norm of the difference $(\brho^{\varepsilon_k}_{t,\mu} -\brho^{\varepsilon_j}_{t,\mu})(v,\cdot)$ {\color{black}for} $v\in [t,s], \ s\in [t,T]$. We will use the same approach to upper-bound this latter quantity in terms of the appropriate spatial Besov norm. To this end, we need to state an estimate for the difference with generic parameters for the Besov spaces, in order to be able to get controls on both the $B_{p',q'}^{-\beta+\vartheta\Gamma} $ and the $L^1$ spatial norms (through the embeddings \eqref{EMBEDDING} for the latter) of the difference. Write for any $\gamma \ge 0$, $\ell,m \ge1 $:

\begin{eqnarray}\label{DEF_DELTA_RHO_RHO_KP}
| \brho^{\varepsilon_k}_{t,\mu}(s,\cdot) -\brho^{\varepsilon_j}_{t,\mu}(s,\cdot)|_{B^{\gamma}_{\ell,m}}
&\le &\int_t^s dv \Big|\{{\mathcal B}_{\brho^{\varepsilon_k}_{t,\mu}}^{\varepsilon_k}(v,\cdot) \brho^{\varepsilon_k}_{t,\mu}(v,\cdot)-{\mathcal B}_{\brho^{\varepsilon_k}_{t,\mu}}^{\varepsilon_k}(v,\cdot) \brho^{\varepsilon_j}_{t,\mu}(v,\cdot)\} \star \nabla p_{s-v}^{\alpha}\Big|_{B^{\gamma}_{\ell,m}} \notag\\
&&+\int_t^s dv \Big|\{{\mathcal B}_{\brho^{\varepsilon_k}_{t,\mu}}^{\varepsilon_k}(v,\cdot) \brho^{\varepsilon_j}_{t,\mu}(v,\cdot)-{\mathcal B}_{\brho^{\varepsilon_j}_{t,\mu}}^{\varepsilon_k}(v,\cdot) \brho^{\varepsilon_j}_{t,\mu}(v,\cdot)\} \star \nabla p_{s-v}^{\alpha}\Big|_{B^{\gamma}_{\ell,m}} \notag\\
&&+\int_t^s dv \Big|\{{\mathcal B}_{\brho^{\varepsilon_j}_{t,\mu}}^{\varepsilon_k}(v,\cdot) \brho^{\varepsilon_j}_{t,\mu}(v,\cdot)-{\mathcal B}_{\brho^{\varepsilon_j}_{t,\mu}}^{\varepsilon_j}(v,\cdot) \brho^{\varepsilon_j}_{t,\mu}(v,\cdot)\} \star \nabla p_{s-v}^{\alpha}\Big|_{B^{\gamma}_{\ell,m}} \notag\\
&=:& \int_t^s dv \Big[\Delta_{\gamma,\ell,m}^1(t,v) + \Delta_{\gamma,\ell,m}^2(t,v) + \Delta_{\gamma,\ell,m}^3(t,v)\Big].
\end{eqnarray}

Similarly to the proof of Lemma \ref{lem_unifesti_gencase_2}, applying Lemma \ref{LEMME_FGH} successively (and respectively) to the maps $f = b^{\varepsilon_k}(v,\cdot)$ (resp. $f = b^{\varepsilon_k}(v,\cdot)$, $f = (b^{\varepsilon_k}-b^{\varepsilon_j})(v,\cdot))$, $g_1 = \brho^{\varepsilon_k}_{t,\mu}(v,\cdot)$ (resp. $g_1 = (\brho^{\varepsilon_k}_{t,\mu}-\brho^{\varepsilon_j}_{t,\mu})(v,\cdot)$, $g_1 = \brho^{\varepsilon_j}_{t,\mu}(v,\cdot)$) and $g_2 =(\brho^{\varepsilon_k}_{t,\mu}-\brho^{\varepsilon_j}_{t,\mu})(v,\cdot)$ (resp. $\textcolor{black}{g_2 = \brho^{\varepsilon_j}_{t,\mu}(v,\cdot)}$) and $h=\nabla p_{s-v}^\alpha $ with $\gamma=-\beta+\vartheta\Gamma$, $ \ell=p'$, $ m=q'$; $\gamma_f=\beta$ \textcolor{black}{(and $\gamma_f=\beta-\vartheta \Gamma $ for $\Delta_{\gamma,\ell,m}^3 $)}, $\ell_f=p$ and $m_f=q$; $ \ell_{g_1} = p'$,  $m_{g_1}=q'$, $\ell_{g_2} = 1$ and $\ell_h = \ell$; $m_h=1$ (see also \eqref{FOR_FURTHER_REF_2}), we derive that:

\begin{eqnarray*}
\Delta_{-\beta+\vartheta\Gamma,p',q'}^1(t,v) &\le&  C|b^{\varepsilon_k}(v,\cdot) |_{B^{\beta}_{p,q}}
|\brho^{\varepsilon_k}_{t,\mu}(v,\cdot)|_{B^{-\beta}_{p',q'}} |(\brho^{\varepsilon_k}_{t,\mu}-\brho^{\varepsilon_j}_{t,\mu})(v,\cdot)|_{L^1}\big| \nabla p_{s-v}^{\alpha} \big|_{B^{-\beta+\vartheta\Gamma}_{p',1}},\\
\Delta_{-\beta+\vartheta\Gamma,p',q'}^2(t,v) &\le&  C|b^{\varepsilon_k}(v,\cdot) |_{B^{\beta}_{p,q}} |(\brho^{\varepsilon_k}_{t,\mu}-\brho^{\varepsilon_j}_{t,\mu})(v,\cdot)|_{B^{-\beta}_{p',q'}}  |\brho^{\varepsilon_k}_{t,\mu}(v,\cdot)|_{L^1} \big|\nabla p_{s-v}^{\alpha} \big|_{B^{-\beta+\vartheta\Gamma}_{p',1}},\\
\Delta_{-\beta+\vartheta\Gamma,p',q'}^3(t,v) &\le&  C|( b^{\varepsilon_k}- b^{\varepsilon_j})(v,\cdot) |_{B^{\textcolor{black}{\beta-\vartheta\Gamma}}_{p,q}} |\brho^{\varepsilon_j}_{t,\mu}(v,\cdot)|_{B^{\textcolor{black}{-\beta+\vartheta\Gamma}}_{p',q'}}  |\brho^{\varepsilon_k}_{t,\mu}(v,\cdot)|_{L^1} \big|\nabla p_{s-v}^{\alpha} \big|_{B^{-\beta+\vartheta\Gamma}_{p',1}}.
\end{eqnarray*}
Thus, \eqref{SING_STABLE_HK} together with \eqref{BesovEmbedding} {\color{black}(which notably gives: $B^{-\beta+\vartheta\Gamma}_{p',q'}\hookrightarrow B^{-\beta}_{p',q'} \ $)} and \eqref{DEF_DELTA_RHO_RHO_KP} for the above {\color{black}choice of the} parameters {\color{black}$\gamma,\ell,m$,} give

\begin{eqnarray}\label{eq:inter_stabi11}
&& |\brho^{\varepsilon_k}_{t,\mu}(s,\cdot) -\brho^{\varepsilon_j}_{t,\mu}(s,\cdot)|_{B^{-\beta+\vartheta\Gamma}_{p',q'}} \\ 
&\le& C\int_t^s \frac{dv}{(s-v)^{\frac{-\beta+\vartheta\Gamma+1}{\alpha} + \frac d{\alpha p}}} \Big\{ |b^{\varepsilon_k}(v,\cdot)|_{B^{\beta}_{p,q}} \textcolor{black}{|\brho^{\varepsilon_k}_{t,\mu}(v,\cdot)|_{B^{-\beta+\vartheta\Gamma}_{p',q'}}}|(\brho^{\varepsilon_k}_{t,\mu}-\brho^{\varepsilon_j}_{t,\mu})(v,\cdot)|_{L^1}\notag \\
&&\quad + |b^{\varepsilon_k}(v,\cdot)|_{B^{\beta}_{p,q}} |(\brho^{\varepsilon_k}_{t,\mu}-\brho^{\varepsilon_j}_{t,\mu})(v,\cdot)|_{B^{-\beta+\vartheta\Gamma}_{p',q'}} + |(b^{\varepsilon_k}-b^{\varepsilon_j})(v,\cdot)|_{B^{\textcolor{black}{\beta-\vartheta \Gamma}}_{p,q}}|\brho^{\varepsilon_j}_{t,\mu}(v,\cdot)|_{B^{-\beta+\vartheta\Gamma}_{p',q'}}\Big\}.\notag
\end{eqnarray}
To estimate the  $L^1$ norm of the above difference,  
{\color{black} come back to \eqref{DEF_DELTA_RHO_RHO_KP}, take now $\gamma = 0$, $\ell=m=1$}
and apply Lemma \ref{LEMME_FGH} with  
{\color{black}the \textcolor{black}{following} choice \textcolor{black}{of} parameters :} $\gamma_f=\beta$, $\ell_f=p$ and $m_f=q$; $ \ell_{g_1} = p'$,  $m_{g_1}=q', \ell_{g_2} = 1,$ and $\ell_h = \ell\textcolor{black}{=1}$, $m_h=1$.  Use \textcolor{black}{then} the embedding  \eqref{EMBEDDING} {\color{black}and \eqref{SING_STABLE_HK}} to obtain from the $L^1 : L^{r}-L^{r'}$ \textcolor{black}{and $L^1 : L^{\check r}-L^{\check r'}$}-H\"older \textcolor{black}{inequalities} that, for all $\textcolor{black}{s\in (t,T]}$,

\begin{eqnarray*}
 |\brho^{\varepsilon_k}_{t,\mu}(s,\cdot) -\brho^{\varepsilon_j}_{t,\mu}(s,\cdot)|_{L^1}  &\underset{\eqref{EMBEDDING}}{\le} &C\  \textcolor{black}{|\brho^{\varepsilon_k}_{t,\mu}(s,\cdot) -\brho^{\varepsilon_j}_{t,\mu}(s,\cdot)|_{B_{1,1}^0}}\\
&\le&
 C\int_t^s dv \textcolor{black}{|\nabla p_{s-v}^\alpha|_{B_{1,1}^0}} \Big\{ |b^{\varepsilon_k}(v,\cdot)|_{B^{\beta}_{p,q}} \textcolor{black}{|\brho^{\varepsilon_k}_{t,\mu}(v,\cdot)|_{B^{-\beta+\vartheta\Gamma}_{p',q'}}}|(\brho^{\varepsilon_k}_{t,\mu}-\brho^{\varepsilon_j}_{t,\mu})(v,\cdot)|_{L^1}\notag \\
&&\quad + |b^{\varepsilon_k}(v,\cdot)|_{B^{\beta}_{p,q}} |(\brho^{\varepsilon_k}_{t,\mu}-\brho^{\varepsilon_j}_{t,\mu})(v,\cdot)|_{B^{-\beta+\vartheta\Gamma}_{p',q'}} \\
&&+ |(b^{\varepsilon_k}-b^{\varepsilon_j})(v,\cdot)|_{B^{\textcolor{black}{\beta-\vartheta \Gamma}}_{p,q}}|\brho^{\varepsilon_j}_{t,\mu}(v,\cdot)|_{B^{-\beta+\vartheta\Gamma}_{p',q'}}\Big\}\\
 &\textcolor{black}{\underset{\eqref{SING_STABLE_HK}}{\le}}& C \Bigg( \int_t^s\frac{dv}{(s-v)^{\frac{r'}{\alpha}}} \textcolor{black}{|\brho^{\varepsilon_k}_{t,\mu}(v,\cdot)|_{B^{-\beta+\vartheta\Gamma}_{p',q'}}^{r'}} \Bigg)^{\frac 1{r'}}  |b^{\varepsilon_k}|_{L^r(B^{\beta}_{p,q})} \sup_{v \in (t,s]}\big\{|(\brho^{\varepsilon_k}_{t,\mu}-\brho^{\varepsilon_j}_{t,\mu})(v,\cdot)|_{L^1}\big\}\\
&& + C  |b^{\varepsilon_k}|_{L^r(B^{\beta}_{p,q})} \Bigg( \int_t^s\frac{dv}{(s-v)^{\frac{r'}{\alpha}}} |(\brho^{\varepsilon_k}_{t,\mu}-\brho^{\varepsilon_j}_{t,\mu})(v,\cdot)|_{B^{-\beta+\vartheta\Gamma}_{p',q'}}^{r'} \Bigg)^{\frac 1{r'}}  \\
&&+ C |(b^{\varepsilon_k}-b^{\varepsilon_j})|_{L^{\textcolor{black}{\check r}}(B^{\textcolor{black}{\beta-\vartheta\Gamma}}_{p,q})}  \Bigg( \int_t^s\frac{dv}{(s-v)^{\frac{\textcolor{black}{\check r'}}{\alpha}}} |\brho^{\varepsilon_j}_{t,\mu}(v,\cdot)|_{B^{-\beta+\vartheta\Gamma}_{p',q'}}^{r'} \Bigg)^{\frac 1{\textcolor{black}{\check r'}}},
\end{eqnarray*}
with $\check r=r $ if $r<+\infty$ and any finite $\check r$ large enough if $r=+\infty $, i.e. the conjugate exponent $\check r' $ belongs to the interval indicated after \eqref{LEMME_2_WITHOUT_SING_2}. We insist that this additional step is needed when $r=+\infty $ to use Proposition \ref{PROP_APPROX} which gives the convergence of the \textcolor{black}{mollified interaction} kernel.

Taking the supremum in $s$ on $(t,S]$ for some $S\le T$ on both sides {\color{black} and recalling the notation \eqref{DEF_L_B_WEIGHTED}, 
 we obtain 
 }
\begin{eqnarray*}
&&\sup_{s\in (t,S]}\big\{|\brho^{\varepsilon_k}_{t,\mu}(s,\cdot) -\brho^{\varepsilon_j}_{t,\mu}(s,\cdot)|_{L^1}\big\}\\
&\le&  C\sup_{s\in (t,S]} \textcolor{black}{|\brho^{\varepsilon_k}_{t,\mu}|_{\textcolor{black}{\Lw^{r'}}\big((t,s],B_{p',q'}^{-\beta+\vartheta\Gamma}\big)}}  
|b^{\varepsilon_k}|_{L^r(B^{\beta}_{p,q})} \sup_{s \in (t,S]}\big\{|(\brho^{\varepsilon_k}_{t,\mu}-\brho^{\varepsilon_j}_{t,\mu})(s,\cdot)|_{L^1}\big\}
\\
&&+C  |b^{\varepsilon_k}|_{L^r(B^{\beta}_{p,q})}\sup_{s\in (t,S]}  |(\brho^{\varepsilon_k}_{t,\mu}-\brho^{\varepsilon_j}_{t,\mu})|_{{\color{black}\Lw^{r'}\big((t,s],B_{p',q'}^{-\beta+\vartheta\Gamma}\big)}}  \\
&&+ C |(b^{\varepsilon_k}-b^{\varepsilon_j})|_{L^{{\check r}}(B^{\textcolor{black}{\beta-\vartheta\Gamma}}_{p,q})}  \sup_{s\in (t,S]} |\brho^{\varepsilon_j}|_{{\color{black}\Lw^{\textcolor{black}{\check r'}}\big((t,s],B_{p',q'}^{-\beta+\vartheta\Gamma}\big)}}.
\end{eqnarray*}
Thus, from \eqref{LEMME_2_WITHOUT_SING_2} for $T$ such that $|b|_{L^r(B_{p,q}^\beta)}(\textcolor{black}{T-t})^{\delta/{r'}}<\textcolor{black}{\cc^{-1}} $, $ \delta=\check \delta(r')$ and Proposition \ref{PROP_APPROX}, we obtain
\begin{eqnarray}\label{esti_inter_imp_conv}
\sup_{s\in (t,S]}\big\{|\brho^{\varepsilon_k}_{t,\mu}(s,\cdot) -\brho^{\varepsilon_j}_{t,\mu}(s,\cdot)|_{L^1}\big\}\notag &\le& 
C  |b^{\varepsilon_k}|_{L^r(B^{\beta}_{p,q})} \sup_{s\in (t,S ]}|(\brho^{\varepsilon_k}_{t,\mu}-\brho^{\varepsilon_j}_{t,\mu})|_{{\color{black}\Lw^{r'}\big((t,s],B_{p',q'}^{-\beta+\vartheta\Gamma}\big)}}   \\
&&+ C (S-t)^{\textcolor{black}{\frac{\check  \delta}{\check r'}}} |(b^{\varepsilon_k}-b^{\varepsilon_j})|_{L^{{\check r}}(B^{\textcolor{black}{\beta-\vartheta \Gamma}}_{p,q})} .
\end{eqnarray}
{\color{black}\textcolor{black}{The above control provides} the estimate for the running norm $s\mapsto \sup_{v\in (t,s]}\big\{|\brho^{\varepsilon_k}_{t,\mu}(v,\cdot) -\brho^{\varepsilon_j}_{t,\mu}(v,\cdot)|_{L^1}$,} \textcolor{black}{which} plugged into \eqref{eq:inter_stabi11}  \textcolor{black}{yields}
\begin{eqnarray}
 &&|\brho^{\varepsilon_k}_{t,\mu}(s,\cdot) -\brho^{\varepsilon_j}_{t,\mu}(s,\cdot)|_{B^{-\beta+\vartheta\Gamma}_{p',q'}} \notag\\
&\le& C\bigg\{|b^{\varepsilon_k}|_{L^r(B^{\beta}_{p,q})}
 \sup_{v\in (t,s]} |(\brho^{\varepsilon_k}_{t,\mu}-\brho^{\varepsilon_j}_{t,\mu})|_{{\color{black}\Lw^{r'}\big((t,v],B_{p',q'}^{-\beta+\vartheta\Gamma}\big)}} + (s-t)^{\textcolor{black}{\frac{ \check \delta}{\check r'}}}|(b^{\varepsilon_k}-b^{\varepsilon_j})|_{L^{\textcolor{black}{\check r}}(B^{\textcolor{black}{\beta-\vartheta\Gamma}}_{p,q})}  \bigg\}\notag\\
&&\quad \times\int_t^s \frac{dv}{(s-v)^{\frac{-\beta+\vartheta\Gamma+1}{\alpha} + \frac d{\alpha p}}}\bigg\{  |b^{\varepsilon_k}(v,\cdot)|_{B^{\beta}_{p,q}}|\brho^{\varepsilon_k}_{t,\mu}(v,\cdot)|_{B^{-\beta+\vartheta\Gamma}_{p',q'}} \notag\bigg\} \notag\\
&& +C\int_t^s \frac{dv}{(s-v)^{\frac{-\beta+\vartheta\Gamma+1}{\alpha} + \frac d{\alpha p}}} \bigg\{|b^{\varepsilon_k}(v,\cdot)|_{B^{\beta}_{p,q}} |(\brho^{\varepsilon_k}_{t,\mu}-\brho^{\varepsilon_j}_{t,\mu})(v,\cdot)|_{B^{-\beta+\vartheta\Gamma}_{p',q'}} \notag\\
&&\hspace*{1cm}+ |(b^{\varepsilon_k}-b^{\varepsilon_j})(v,\cdot)|_{B^{\textcolor{black}{\beta-\vartheta\Gamma}}_{p,q}}|\brho^{\varepsilon_j}_{t,\mu}(v,\cdot)|_{B^{-\beta+\vartheta\Gamma}_{p',q'}}\bigg\}.
\label{PRIMER_FOR_INT_STAB}
\end{eqnarray}
{\color{black}At this stage, we can proceed similarly as in the proof of Lemma \ref{lem_unifesti_gencase_2}, leaving the component factored by $|(b^{\varepsilon_k}-b^{\varepsilon_j})|_{L^{\textcolor{black}{\check r}}(B^{\textcolor{black}{\beta-\vartheta\Gamma}}_{p,q})}$ as a source term.}
 \textcolor{black}{Use} first
the $L^1:L^{r}-L^{r'}$-H\"older inequality for the first \textcolor{black}{two terms} in the above \textcolor{black}{integrals} and the $L^1:L^{\check r}-L^{\check r'}$ for the last \textcolor{black}{integral} term. \textcolor{black}{Multiply} then both sides by $(S-s)^{-1/\alpha}$ for some $s< S\le T$. \textcolor{black}{Take} next the $(r')^{{\rm th}}$ power,  \textcolor{black}{use consequently} a convexity inequality in the resulting expression and finally integrate the whole on $(t,S]$.  \textcolor{black}{We} eventually obtain
\begin{eqnarray*}
 && \int_t^S
 \frac{ds}{(S-s)^{\frac{r'}\alpha}}|\brho^{\varepsilon_k}_{t,\mu} (s,\cdot) -\brho^{\varepsilon_j}_{t,\mu}(s,\cdot)|_{B^{-\beta+\vartheta\Gamma}_{p',q'}}^{r'} \\
 &\le& C
 \bigg\{|b^{\varepsilon_k}|_{L^r(B^{\beta}_{p,q})}^{r'} \sup_{v\in (t,S]}|(\brho^{\varepsilon_k}_{t,\mu}-\brho^{\varepsilon_j}_{t,\mu})|_{{\color{black}\Lw^{r'}\big((t,v],B_{p',q'}^{-\beta+\vartheta\Gamma}\big)}}^{r'} + (S-t)^{\textcolor{black}{\frac{\check \delta r'}{\check r'}}}  |(b^{\varepsilon_k}-b^{\varepsilon_j})|_{L^{\textcolor{black}{\check r}}(B^{\textcolor{black}{\beta-\vartheta\Gamma}}_{p,q})}^{r'}  \bigg\}\\
&&\qquad \times\int^S_t \frac{ds}{(S-s)^{\frac{r'}\alpha}}\int_t^s \frac{dv}{(s-v)^{r'\big(\frac{-\beta+\vartheta\Gamma+1}{\alpha} + \frac d{\alpha p}\big)}}\bigg\{  |\brho^{\varepsilon_k}_{t,\mu}(v,\cdot)|_{B^{-\beta+\vartheta\Gamma}_{p',q'}}^{r'}\textcolor{black}{|b^{\varepsilon_k}|_{L^r(B_{p,q}^\beta)}^{r'}} \notag\bigg\} \\
&& +C\int^S_t \frac{ds}{(S-s)^{\frac{r'}\alpha}}\textcolor{black}{\Bigg[}\int_t^s \frac{dv}{(s-v)^{r'\big(\frac{-\beta+\vartheta\Gamma+1}{\alpha} + \frac d{\alpha p}\big)}}
|b^{\varepsilon_k}|_{L^r(B^{\beta}_{p,q})}^{r'} |(\brho^{\varepsilon_k}_{t,\mu}-\brho^{\varepsilon_j}_{t,\mu})(v,\cdot)|_{B^{-\beta+\vartheta\Gamma}_{p',q'}}^{r'}\\
 &&+\textcolor{black}{|(b^{\varepsilon_k}-b^{\varepsilon_j})|_{L^{\textcolor{black}{\check r}}(B^{\textcolor{black}{\beta-\vartheta\Gamma}}_{p,q})}^{r'} \bigg\{\int_t^s \frac{dv}{(s-v)^{\check r'\big(\frac{-\beta+\vartheta\Gamma+1}{\alpha} + \frac d{\alpha p}\big)}} |\brho^{\varepsilon_j}_{t,\mu}(v,\cdot)|_{B^{-\beta+\vartheta\Gamma}_{p',q'}}^{\check r'}\bigg\}^{\frac{r'}{\check r'}}}\textcolor{black}{\Bigg]}%
 .\notag
\end{eqnarray*}
From the Fubini theorem {\color{black}and} using 
{\color{black}again} the $L^1: L^{\check r}-L^{\check r'}$ inequality for the last contribution if $r'=1<\check r' $ ( 
{\color{black}$\Leftrightarrow$} $r=\infty,\ \check r<+\infty $), we now derive:
\begin{eqnarray*}
 && \int_t^S
 \frac{ds}{(S-s)^{\frac{r'}\alpha}}|\brho^{\varepsilon_k}_{t,\mu} (s,\cdot) -\brho^{\varepsilon_j}_{t,\mu}(s,\cdot)|_{B^{-\beta+\vartheta\Gamma}_{p',q'}}^{r'}\\
 &\le& C\bigg\{|b^{\varepsilon_k}|_{L^r(B^{\beta}_{p,q})}^{r'}  \sup_{v\in (t,S]}|(\brho^{\varepsilon_k}_{t,\mu}-\brho^{\varepsilon_j}_{t,\mu})|_{ 
 {\color{black}\Lw^{r'}\big((t,v],B_{p',q'}^{-\beta+\vartheta\Gamma}\big)}}^{r'}+ (S-t)^{\textcolor{black}{ \frac{\check \delta r'}{\check r'}}} |(b^{\varepsilon_k}-b^{\varepsilon_j})|_{L^{\textcolor{black}{\check r}}(B^{\textcolor{black}{\beta-\vartheta\Gamma}}_{p,q})}^{r'} \bigg\}\\
&&\qquad \times\int_t^Sdv \int_v^S \frac{ds}{(S-s)^{\frac{r'}\alpha} (s-v)^{r'\big(\frac{-\beta+\vartheta\Gamma+1}{\alpha} + \frac d{\alpha p}\big)}}\bigg\{  |\brho^{\varepsilon_k}_{t,\mu}(v,\cdot)|_{B^{-\beta+\vartheta\Gamma}_{p',q'}}^{r'} \textcolor{black}{|b^{\varepsilon_k}|_{L^r(B_{p,q}^\beta)}^{r'}}\notag\bigg\} \\
&& +C\textcolor{black}{\Bigg[}\int_t^S dv \int_v^S \frac{ds}{(S-s)^{\frac{r'}\alpha}(s-v)^{r'\big(\frac{-\beta+\vartheta\Gamma+1}{\alpha} + \frac d{\alpha p}\big)}}|b^{\varepsilon_k}|_{L^r(B^{\beta}_{p,q})}^{r'} |(\brho^{\varepsilon_k}_{t,\mu}-\brho^{\varepsilon_j}_{t,\mu})(v,\cdot)|_{B^{-\beta+\vartheta\Gamma}_{p',q'}}^{r'}\\
&& + |(b^{\varepsilon_k}-b^{\varepsilon_j})|_{L^{\textcolor{black}{\check r}}(B^{\textcolor{black}{\beta-\vartheta\Gamma}}_{p,q})}^{r'}\textcolor{black}{(S-t)^{\frac {\I_{r=\infty}}{\check r}}}\int_t^Sdv \int_v^S \frac{ds}{(S-s)^{\frac{\textcolor{black}{\check r'}}\alpha}(s-v)^{\textcolor{black}{\check r'}\big(\frac{-\beta+\vartheta\Gamma+1}{\alpha} + \frac d{\alpha p}\big)}} |\brho^{\varepsilon_j}_{t,\mu}(v,\cdot)|_{B^{-\beta+\vartheta\Gamma}_{p',q'}}^{\textcolor{black}{\check r'}}\textcolor{black}{\Bigg]}\notag\\
 &\le& C\bigg\{|b^{\varepsilon_k}|_{L^r(B^{\beta}_{p,q})}^{r'}  \sup_{v\in (t,S]}|(\brho^{\varepsilon_k}_{t,\mu}-\brho^{\varepsilon_j}_{t,\mu})|_{{\color{black}\Lw^{r'}\big((t,v],B_{p',q'}^{-\beta+\vartheta\Gamma}\big)}}^{r'}+ (S-t)^{ \textcolor{black}{\frac{\check \delta r'}{\check r'}}} |(b^{\varepsilon_k}-b^{\varepsilon_j})|_{L^{\textcolor{black}{\check r}}(B^{\textcolor{black}{\beta-\vartheta\Gamma}}_{p,q})}^{r'} \bigg\}\\
&&\qquad \times\int_t^S\frac{dv}{(S-v)^{\frac{r'}\alpha+r'\big(\frac{-\beta+\vartheta\Gamma+1}{\alpha} + \frac d{\alpha p}\big)-1}}\bigg\{  |\brho^{\varepsilon_k}_{t,\mu}(v,\cdot)|_{B^{-\beta+\vartheta\Gamma}_{p',q'}}^{r'} \textcolor{black}{|b^{\varepsilon_k}|_{L^r(B_{p,q}^\beta)}^{r'}}\notag\bigg\} \\
&& +C\textcolor{black}{\Bigg[}\int_t^S \frac{dv}{ (S-v)^{\frac{r'}\alpha+r'\big(\frac{-\beta+\vartheta\Gamma+1}{\alpha} + \frac d{\alpha p}\big)-1}} 
 |b^{\varepsilon_k}|_{L^r(B^{\beta}_{p,q})}^{r'} |(\brho^{\varepsilon_k}_{t,\mu}-\brho^{\varepsilon_j}_{t,\mu})(v,\cdot)|_{B^{-\beta+\vartheta\Gamma}_{p',q'}}^{r'} \\
&&+ |(b^{\varepsilon_k}-b^{\varepsilon_j})|_{L^{\textcolor{black}{\check r}}(B^{\textcolor{black}{\beta-\vartheta\Gamma}}_{p,q})}^{r'}\textcolor{black}{(S-t)^{\frac {\I_{r=\infty}}{\check r}}}\int_t^S \frac{dv}{ (S-v)^{\frac{\textcolor{black}{\check r'}}\alpha+\textcolor{black}{\check r'}\big(\frac{-\beta+\vartheta\Gamma+1}{\alpha} + \frac d{\alpha p}\big)-1}}|\brho^{\varepsilon_j}_{t,\mu}(v,\cdot)|_{B^{-\beta+\vartheta\Gamma}_{p',q'}}^{\textcolor{black}{\check r'}}\textcolor{black}{\Bigg]}\notag\\
 &\le& C \bigg\{|b^{\varepsilon_k}|_{L^r(B^{\beta}_{p,q})}^{r'} \sup_{v\in (t,T]}|(\brho^{\varepsilon_k}_{t,\mu}-\brho^{\varepsilon_j}_{t,\mu})|_{{\color{black}\Lw^{r'}\big((t,v],B_{p',q'}^{-\beta+\vartheta\Gamma}\big)}}^{r'} + (S-t)^{ \textcolor{black}{\frac{\check \delta r'}{\check r'}}} |(b^{\varepsilon_k}-b^{\varepsilon_j})|_{L^{\textcolor{black}{\check r}}(B^{\textcolor{black}{\beta-\vartheta\Gamma}}_{p,q})}^{r'} \bigg\} (S-t)^{ \textcolor{black}{\delta}} \textcolor{black}{|b^{\varepsilon_k}|_{L^r(B_{p,q}^\beta)}^{r'}}\\
&&+ \bigg\{|b^{\varepsilon_k}|_{L^r(B^{\beta}_{p,q})}^{r'} \sup_{v\in (t,T]}|(\brho^{\varepsilon_k}_{t,\mu}-\brho^{\varepsilon_j}_{t,\mu})|_{{\color{black}\Lw^{r'}\big((t,v],B_{p',q'}^{-\beta+\vartheta\Gamma}\big)}}^{r'} \textcolor{black}{(S-t)^{ \delta}}+ |(b^{\varepsilon_k}-b^{\varepsilon_j})|_{L^{\textcolor{black}{\check r}}(B^{\textcolor{black}{\beta-\vartheta\Gamma}}_{p,q})}^{r'}\textcolor{black}{(S-t)^{ \textcolor{black}{\check \delta}+\frac{\I_{r=\infty}}{\check r}}}\bigg\}\notag,
\end{eqnarray*}
where, for the last inequality, we use {\color{black}the uniform controls on $b^{\varepsilon}$ and $\brho^{\varepsilon}_{t,\mu}$ given by Proposition \ref{PROP_APPROX} and} \eqref{LEMME_2_WITHOUT_SING_2} {\color{black}in Lemma \ref{lem_unifesti_gencase_2}}, and the fact that, from  \eqref{DEF_DELTA} and \eqref{DEF_DELTA_CHECK}: 
$$-\delta=r'\bigg(\frac{-\beta+\vartheta\Gamma+1}{\alpha} + \frac d{\alpha p}\bigg)-1\textcolor{black}{<} 0,\ \textcolor{black}{-\check \delta=\check r'\bigg(\frac{-\beta+\vartheta\Gamma+1}{\alpha} + \frac d{\alpha p}\bigg)-1\textcolor{black}{<}0}$$ 
{\color{black}(the contribution \textcolor{black}{in} \textcolor{black}{$ 1/{\check r}$} in the last power only arises when $r=\infty$).}
Taking now the supremum over  $S\in (t,T]${\color{black}, we recover on the left-hand side above \textcolor{black}{the quantity} $\sup_{v\in (t,T]}|(\brho^{\varepsilon_k}_{t,\mu}-\brho^{\varepsilon_j}_{t,\mu})|_{{\color{black}\Lw^{r'}\big((t,v],B_{p',q'}^{-\beta+\vartheta\Gamma}\big)}}^{r'}$} and for $T>t\ge 0$ small enough, we deduce that t\textcolor{black}{here exists $C\ge 1$ s.t. }
\begin{eqnarray}\label{CAUCHY_BOUND_LP_B_W_SUP}
\sup_{v\in (t,T]}|(\brho^{\varepsilon_k}_{t,\mu}-\brho^{\varepsilon_j}_{t,\mu})|_{{\color{black}\Lw^{r'}\big((t,v],B_{p',q'}^{-\beta+\vartheta\Gamma}\big)}}^{r'} \le \textcolor{black}{C} |(b^{\varepsilon_k}-b^{\varepsilon_j})|_{L^{\textcolor{black}{\check r}}(B^{\textcolor{black}{\beta-\vartheta \Gamma}}_{p,q})}^{r'}.
\end{eqnarray}
As $ |\brho^{\varepsilon_k}_{t,\mu} -\brho^{\varepsilon_j}_{t,\mu}|_{L^{\textcolor{black}{r'}}\big((t,T],B_{p',q'}^{-\beta+\vartheta\Gamma}\big)}^{r'} \le  \sup_{v\in (t,T]}|(\brho^{\varepsilon_k}_{t,\mu}-\brho^{\varepsilon_j}_{t,\mu})|_{{\color{black}\Lw^{r'}\big((t,v],B_{p',q'}^{-\beta+\vartheta\Gamma}\big)}}^{r'}$ and $L^{r'}\big((t,T],B^{-\beta+\vartheta\Gamma}_{p',q'}\big)$ is complete we have, from Proposition \ref{PROP_APPROX},  that the convergence of $(\brho_{t,\mu}^{\varepsilon_k})_{k \ge 1}$ holds in $L^{r'}\big((t,T],B^{-\beta+\vartheta\Gamma}_{p',q'}\big)$ in a strong sense. The strong convergence in $L^{\infty}(L^1)$ follows from  \eqref{esti_inter_imp_conv}, again Proposition \ref{PROP_APPROX} and \eqref{CAUCHY_BOUND_LP_B_W_SUP}. The proof of \eqref{STRONG_CONV} is thus complete for $\bar r=r'$.\\

Let us now  discuss how we can improve the integrability exponent {\color{black}$\bar r$} for the convergence to prove \eqref{STRONG_CONV} in full generality. Restarting from \eqref{PRIMER_FOR_INT_STAB}, write first for $\bar r $ satisfying \textcolor{black}{\eqref{cond_time_gencase}},
\begin{eqnarray*}
 &&\int_t^T |\brho^{\varepsilon_k}_{t,\mu}(s,\cdot) -\brho^{\varepsilon_j}_{t,\mu}(s,\cdot)|_{B^{-\beta+\vartheta\Gamma}_{p',q'}}^{\bar r}ds \\
 &\le& C\bigg\{|b^{\varepsilon_k}|_{L^r(B^{\beta}_{p,q})}^{\bar r}\sup_{v\in (t,T]} |(\brho^{\varepsilon_k}_{t,\mu}-\brho^{\varepsilon_j}_{t,\mu})|_{{\color{black}\Lw^{r'}\big((t,v],B_{p',q'}^{-\beta+\vartheta\Gamma}\big)}}^{\bar r} + (T-t)^{\textcolor{black}{\frac{ \check\delta \bar r}{\check r'}}}|(b^{\varepsilon_k}-b^{\varepsilon_j})|_{L^{\textcolor{black}{\check r}}(B^{\beta}_{p,q})}^{\bar r}  \bigg\}\\
&&\quad \times|b^{\varepsilon_k}|_{L^r(B^{\beta}_{p,q})}^{\bar r}\int_t^T ds  \Big(\int_t^s \frac{dv}{(s-v)^{r'(\frac{-\beta+\vartheta\Gamma+1}{\alpha} + \frac d{\alpha p})}}| \brho^{\varepsilon_k}_{t,\mu}(v,\cdot)|_{B^{-\beta+\vartheta\Gamma}_{p',q'}}^{r'} \notag\Big)^{\frac{\bar r}{r'}} \\
&& +C|b^{\varepsilon_k}|_{L^r(B^{\beta}_{p,q})}^{\bar r }\int_t^T ds \Big(\int_t^s \frac{dv}{(s-v)^{r'(\frac{-\beta+\vartheta\Gamma+1}{\alpha} + \frac d{\alpha p})}}  |(\brho^{\varepsilon_k}_{t,\mu}-\brho^{\varepsilon_j}_{t,\mu})(v,\cdot)|_{B^{-\beta+\vartheta\Gamma}_{p',q'}}^{r'}\Big)^{\frac{\bar r}{r'}} \\
&&+ C|(b^{\varepsilon_k}-b^{\varepsilon_j})|_{L^{\textcolor{black}{\check r}}(B^{\textcolor{black}{\beta-\vartheta \Gamma}}_{p,q})}^{\bar r}\int_t^T ds \Big(\int_t^s \frac{dv}{(s-v)^{\textcolor{black}{\check r'}(\frac{-\beta+\vartheta\Gamma+1}{\alpha} + \frac d{\alpha p})}}  |\brho^{\varepsilon_j}_{t,\mu}(v,\cdot)|_{B^{-\beta+\vartheta\Gamma}_{p',q'}}^{\textcolor{black}{\check r'}}\Big)^{\frac{\bar r}{\textcolor{black}{\check r'}}}.\notag
\end{eqnarray*}
We now conclude as in the proof of Lemma \ref{lem_unifesti_gencase_2}, through again an additional H\"older inequality and the Fubini theorem (see 
{\color{black}the procedure between \eqref{AFTER_CONV_2} and \eqref{AFTER_FUB}}). Namely,
\begin{eqnarray*}
 &&\int_t^T |\brho^{\varepsilon_k}_{t,\mu}(s,\cdot) -\brho^{\varepsilon_j}_{t,\mu}(s,\cdot)|_{B^{-\beta+\vartheta\Gamma}_{p',q'}}^{\bar r}ds \\
 &\le& C\bigg\{|b^{\varepsilon_k}|_{L^r(B^{\beta}_{p,q})}^{\bar r}\sup_{v\in (t,T]} |(\brho^{\varepsilon_k}_{t,\mu}-\brho^{\varepsilon_j}_{t,\mu})|_{{\color{black}\Lw^{r'}\big((t,v],B_{p',q'}^{-\beta+\vartheta\Gamma}\big)}}^{\bar r} + (T-t)^{\textcolor{black}{\frac{ \check \delta \bar r}{\check r'}}}|(b^{\varepsilon_k}-b^{\varepsilon_j})|_{L^{\textcolor{black}{\check r}}(B^{\beta}_{p,q})}^{\bar r}  \bigg\}\\
&&\quad \times (T-t)^{\delta\frac{\bar r}{r'}}|b^{\varepsilon_k}|_{L^r(B^{\beta}_{p,q})}^{\bar r}\int_t^T dv |\brho^{\varepsilon_k}_{t,\mu}(v,\cdot)|_{B^{-\beta+\vartheta\Gamma}_{p',q'}}^{\bar r} \notag \\
&& +C\Bigg\{(T-t)^{\delta\frac{\bar r}{r'}}|b^{\varepsilon_k}|_{L^r(B^{\beta}_{p,q})}^{\bar r }\int_t^T dv  |(\brho^{\varepsilon_k}_{t,\mu}-\brho^{\varepsilon_j}_{t,\mu})(v,\cdot)|_{B^{-\beta+\vartheta\Gamma}_{p',q'}}^{\bar r} \\
&&+ \textcolor{black}{(T-t)^{\check \delta\frac{\bar r}{\check r'}}}|(b^{\varepsilon_k}-b^{\varepsilon_j})|_{L^{\textcolor{black}{\check r}}(B^{\textcolor{black}{\beta-\vartheta \Gamma}}_{p,q})}^{\bar r}\int_t^T \textcolor{black}{dv} |\brho^{\varepsilon_j}_{t,\mu}(v,\cdot)|_{B^{-\beta+\vartheta\Gamma}_{p',q'}}^{\bar r}\Bigg\},\notag
\end{eqnarray*}
\textcolor{black}{assuming w.l.o.g. that since $\bar r>r' $, and if $r=\infty $, i.e. $r'<\check r' $ one can take $\bar r>\check r'>r' $ (recall that if $r<+\infty$, $\check r'=r'$)}. 
Hence, for $T$  small enough and using again Lemma \ref{lem_unifesti_gencase_2} \textcolor{black}{(equation \eqref{LEMME_2_WITHOUT_SING_2})} we obtain:
\begin{eqnarray*}
 &&\int_t^T |\brho^{\varepsilon_k}_{t,\mu}(s,\cdot) -\brho^{\varepsilon_j}_{t,\mu}(s,\cdot)|_{B^{-\beta+\vartheta\Gamma}_{p',q'}}^{\bar r}ds \\
 &\le& C\bigg\{|b^{\varepsilon_k}|_{L^r(B^{\beta}_{p,q})}^{\bar r}\sup_{v\in (t,T]} |(\brho^{\varepsilon_k}_{t,\mu}-\brho^{\varepsilon_j}_{t,\mu})|_{{\color{black}\Lw^{r'}\big((t,v],B_{p',q'}^{-\beta+\vartheta\Gamma}\big)}}^{\bar r} + (T-t)^{\textcolor{black}{\frac{ \check \delta \bar r}{\check r'}}}|(b^{\varepsilon_k}-b^{\varepsilon_j})|_{L^{\textcolor{black}{\check r}}(B^{\beta}_{p,q})}^{\bar r}  \bigg\}\\
&&\quad \times \textcolor{black}{(T-t)^{\theta}}(\textcolor{black}{(T-t)^{\delta\frac{\bar r}{r'}}}|b^{\varepsilon_k}|_{L^r(B^{\beta}_{p,q})}^{\bar r}+1)\\
&\le & C(T-t)^{\theta}|(b^{\varepsilon_k}-b^{\varepsilon_j})|_{L^{\textcolor{black}{\check r}}(B^{\textcolor{black}{\beta-\vartheta\Gamma}}_{p,q})}^{\bar r} ,
\end{eqnarray*}
using \eqref{CAUCHY_BOUND_LP_B_W_SUP} for the last inequality. Proposition \ref{PROP_APPROX} then gives the stated convergence.
\end{proof}

\begin{lemme}\label{lem_ex_duha_gencase}
Let $(t,\mu)$ in $[0,T] \times \mathcal P(\R^d)$ and let $(\varepsilon_k)_{k \ge 1}$ be a decreasing sequence \textcolor{black}{going to 0} and $\brho_{t,\mu}$ in \textcolor{black}{$L^{\bar r}(B_{p',q'}^{-\beta+\vartheta \Gamma})$\textcolor{black}{,} with $\bar r$ as in \textcolor{black}{\eqref{cond_time_gencase}}, $\vartheta\in [0,1) $} and \textcolor{black}{s.t. for all $s\in [t,T], \brho_{t,\mu}(s,\cdot)\in {\mathcal P}(\R^d) $}, be the limit point of $(\brho_{t,\mu}^{\varepsilon_k})_{k\ge 1}$ \textcolor{black}{with $\brho_{t,\mu}^{\varepsilon_k} $ solving \eqref{main_MOLL}}. Then $\brho_{t,\mu}$ satisfies the non-linear Fokker-Planck equation \eqref{NL_PDE_FK} (in a distributional sense) and admits the Duhamel type representation \eqref{main_DUHAMEL}.
\end{lemme}
\begin{proof} We assume here w.l.o.g. and for simplicity that $r<+\infty$. The case $r=+\infty$ could be handled 
{\color{black}reproducing the arguments below combined \textcolor{black}{with the corresponding ones} in the proof of the previous lemma.}
Write from \eqref{PDE_EPS_VAR}:
\begin{eqnarray*}
\label{PDE_EPS_VAR_FOR_LIM}
0=-\int\varphi(t,\textcolor{black}{y})\mu(\textcolor{black}{dy}) &\textcolor{black}{-}& \int_t^T ds \int_{\R^d} dy\,  ( \mathcal B_{\brho_{t,\mu}}(s,y) \brho_{t,\mu}(s,y))\cdot \textcolor{black}{\nabla \varphi(s,y)}\notag\\
&+&\int_t^T ds \int_{\R^d}  dy\, \brho_{t,\mu}(s,y) (-\partial_s+\textcolor{black}{L^\alpha})\varphi(s,y)+\Delta_{\brho_{t,\mu},\brho^{\varepsilon_k}_{t,\mu}}(\varphi),
\end{eqnarray*}
where:
\begin{eqnarray*}
&&\Delta_{\brho_{t,\mu},\brho^{\varepsilon_k}_{t,\mu}}(\varphi)\notag\\
&:=&\int_t^T ds \int_{\R^d}  dy\, (\brho^{\varepsilon_k}_{t,\mu}(s,y)-\brho_{t,\mu}(s,y)) (-\partial_s+\textcolor{black}{L^\alpha})\varphi(s,y) \label{DELTA_K_DISTRIB}\\
&&\textcolor{black}{-}\int_t^T ds \int_{\R^d} dy \, ( \mathcal B_{\brho^{\varepsilon_k}_{t,\mu}}^{\varepsilon_k}(s,y) \brho^{\varepsilon_k}_{t,\mu}(s,y)-\mathcal B_{\brho_{t,\mu}}(s,y) \brho_{t,\mu}(s,y))\cdot \textcolor{black}{\nabla \varphi(s,y)} =:\Delta_{\brho_{t,\mu},\brho^{\varepsilon_k}_{t,\mu}}^1(\varphi)+\Delta_{\brho_{t,\mu},\brho^{\varepsilon_k}_{t,\mu}}^2(\varphi).\notag
\end{eqnarray*}
Since $\varphi \in \textcolor{black}{C_0^\infty}\big((-T,T)\times \R^d\big) $ it is clear that $(\partial_s-\textcolor{black}{L^\alpha})\varphi\in L^{ r}(B_{p,q}^{\beta}) $ so that we readily get from \eqref{STRONG_CONV} that $|\Delta_{\brho_{t,\mu},\brho^{\varepsilon_k}_{t,\mu}}^1(\varphi)|\underset{k}{\longrightarrow} 0$.

For the second term $\Delta_{\brho_{t,\mu},\brho^{\varepsilon_k}_{t,\mu}}^2(\varphi)$, write:
\begin{eqnarray*}
|\Delta_{\brho_{t,\mu},\brho^{\varepsilon_k}_{t,\mu}}^2(\varphi)| &\le& \int_t^T ds \int_{\R^d} dy | ( \mathcal B_{\brho^{\varepsilon_k}_{t,\mu}}^{\varepsilon_k}(s,y) \brho^{\varepsilon_k}_{t,\mu}(s,y)-\mathcal B_{\brho_{t,\mu}}(s,y) \brho_{t,\mu}(s,y))| |\nabla \varphi(s,\cdot)|_{L^\infty} \\
&\le& |\nabla \varphi|_{L^\infty(L^\infty)} \int_t^T ds\Bigg\{ \int_{\R^d} dy |\mathcal B_{\brho^{\varepsilon_k}_{t,\mu}}^{\varepsilon_k}(s,\cdot)- \mathcal B_{\brho^{\varepsilon_k}_{t,\mu}}(s,\cdot)|_{L^\infty} \brho^{\varepsilon_k}_{t,\mu}(s,y)\\
&&+\int_{\R^d} dy |\mathcal B_{\brho^{\varepsilon_k}_{t,\mu}}(s,\cdot)- \mathcal B_{{\brho}_{t,\mu}}(s,\cdot)|_{L^\infty} \brho^{\varepsilon_k}_{t,\mu}(s,y)\\
&&+\int_{\R^d} dy | \mathcal B_{{\brho}_{t,\mu}}(s,\cdot)|_{L^\infty} |\brho^{\varepsilon_k}_{t,\mu}(s,y)-\brho_{t,\mu}(s,y)|\Bigg\}.
  \end{eqnarray*}
Recalling that $\brho^{\varepsilon_k}_{t,\mu}(s,\cdot) $ is a probability density and using as well the  duality control  \eqref{EQ_DUALITY} between Besov spaces and the H\"older inequality in time, we derive that
  for any $\textcolor{black}{\vartheta\in (0,1) }$:
\begin{eqnarray*}
| \Delta_{\brho_{t,\mu},\brho^{\varepsilon_k}_{t,\mu}}^{2}(\varphi)| &\le& {\color{black}C}\Big\{ |b-b^{\varepsilon_{\color{black}k}}|_{L^r(B^{\textcolor{black}{\beta-\vartheta\Gamma}}_{p,q})}  |\brho_{t,\mu}^{\varepsilon_k}|_{L^{r'}(B^{\textcolor{black}{-\beta+\vartheta\Gamma}}_{p',q'})}  +  |b|_{L^r(B^{\beta}_{p,q})}  |\brho_{t,\mu}^{\varepsilon_k}-\brho_{t,\mu}|_{L^{r'}(B^{-\beta}_{p',q'})}\label{DELTA_K_DISTRIB_2}\\
&& \quad  + |b|_{L^r(B^{\beta}_{p,q})}  |\brho_{t,\mu}|_{L^{r'}(B^{-\beta}_{p',q'})}  |\brho_{t,\mu}^{\varepsilon_k}-\brho_{t,\mu}|_{L^{\infty}(L^1)} \Big\} |\nabla \varphi|_{L^{\infty}(L^\infty)}.\notag
\end{eqnarray*}
Using Lemma \ref{FIRST_STAB} and Proposition \ref{PROP_APPROX} we conclude  that $|\Delta_{\brho_{t,\mu},\brho^{\varepsilon_k}_{t,\mu}}^2(\varphi)|\underset{k}{\longrightarrow} 0$ so that $\brho_{t,\mu}$ satisfies \eqref{NL_PDE_FK} in a distributional sense.\\

To establish the Duhamel representation \eqref{main_DUHAMEL} observe first that for any $s\in (t,T] $, $\varphi\in \textcolor{black}{C_0^\infty}\big((-T,T)\times \R^d\big) $,
\begin{align}
\label{PDE_EPS_VAR_BIS_S}
\int_{\R^d} dy \brho^{\varepsilon_k}_{t,\mu}(s,y)\varphi(s,y)
-\int\varphi(t,\textcolor{black}{y})\mu(\textcolor{black}{dy})
 {\color{black}-\int_t^s dv \int_{\R^d} dy\, \{\mathcal B_{\brho^{\varepsilon_k}_{t,\mu}}^{\varepsilon_k}(v,y) \brho^{\varepsilon_k}_{t,\mu}(v,y)\}\cdot \ \nabla\varphi(v,y)}
\notag\\
+\int_t^s dv \int_{\R^d}  dy \brho^{\varepsilon_k}_{t,\mu}(v,y) (\textcolor{black}{-\partial_s+L^\alpha})\varphi(v,y)=0.
\end{align}
Introducing
\begin{eqnarray*}
\Delta_{\brho_{t,\mu},\brho^{\varepsilon_k}_{t,\mu}}^s(\varphi)&:=&\int_{\R^d} dy (\brho^{\varepsilon_k}_{t,\mu}-\brho_{t,\mu})(s,y)\varphi(s,y)
\notag\\
&&{\color{black}-\int_t^s dv \int_{\R^d} dy\, \{ \mathcal B_{\brho^{\varepsilon_k}_{t,\mu}}^{\varepsilon_k}(v,y) \brho^{\varepsilon_k}_{t,\mu}(v,y)-\mathcal B_{\brho_{t,\mu}}(v,y) \brho_{t,\mu}(v,y)\}\cdot \nabla\varphi(v,y)}\notag\\
&&+\int_t^s dv \int_{\R^d}  dy (\brho^{\varepsilon_k}_{t,\mu}(v,y)-\brho_{t,\mu}(v,y)) (-\partial_s+\textcolor{black}{L^\alpha})\varphi(v,y),
\end{eqnarray*}
it can hence be deduced from the previous arguments that, along a suitable subsequence $(\varepsilon_{k_m})_{m\ge 1} $ (to handle the first additional term through the converse of the Lebesgue theorem from the strong convergence \eqref{STRONG_CONV}), that $\Delta_{\brho_{t,\mu},\brho^{\varepsilon_{k_m}}_{t,\mu}}^s(\varphi) \underset{m}{\longrightarrow} 0 $ \textcolor{black}{for almost all $s\in [t,T] $}. Hence, from \eqref{PDE_EPS_VAR_BIS_S}, \textcolor{black}{for almost all $s\in [t,T] $,}
\begin{eqnarray}
\label{PDE_VAR_BIS_S}
\int_{\R^d} dy \brho_{t,\mu}(s,\textcolor{black}{y})\varphi(s,\textcolor{black}{y})
-\int\varphi(t,\textcolor{black}{y})\mu(\textcolor{black}{dy})
-\int_t^s dv \int_{\R^d} d\textcolor{black}{y} \  \mathcal B_{\brho_{t,\mu}}(v,\textcolor{black}{y}) \brho_{t,\mu}(v,\textcolor{black}{y})){\color{black}\ \cdot}\ \nabla \varphi(v,\textcolor{black}{y})\notag\\
+\int_t^s dv \int_{\R^d}  d\textcolor{black}{y} \ \brho_{t,\mu}(v,\textcolor{black}{y}) (-\partial_s+\textcolor{black}{L^\alpha})\varphi(v,\textcolor{black}{y})=0.
\end{eqnarray}
\textcolor{black}{To establish \eqref{main_DUHAMEL} we can start from the Duhamel representation which holds for the density $\brho^{\textcolor{black}{\varepsilon}}_{t,\mu}$ associated with the mollified coefficients (see the proof of Lemma \ref{LEM_DUHAMEL_MOLL} in Appendix \ref{PROOF_LEM_DUHAMEL_MOLL} below). Namely, the above equation \eqref{main_MOLL} which we now recall for clarity:
\begin{align*}
\brho^{\textcolor{black}{\varepsilon}}_{t,\mu}(s,x)= p^\alpha_{s-t}\star \,\mu(x)-\int_t^s dv\, \Big(\nabla p^\alpha_{s-v}\star \{ \mathcal B_{\brho^{\varepsilon}_{t,\mu}}^{\varepsilon}(v,x) \brho^{\varepsilon}_{t,\mu}(v,x)\}\Big).
\end{align*}
The arguments used in Lemma \ref{FIRST_STAB} to establish the $L^1$ convergence then give \eqref{main_DUHAMEL} for the limit.}

\end{proof}

We complete this section with a uniqueness result for \eqref{main_DUHAMEL} in \textcolor{black}{$L^{\overline{r}}(B^{-\beta+\vartheta\Gamma}_{p',q'})$, with $\bar r $ as in \eqref{cond_time_gencase}, $\vartheta\in [0,1) $}, which, with the preceding lemma, yields to the well-posedness result stated in Proposition \ref{WP_FK}.

\begin{lemme}\label{lem_unique_FK_gencase}
For any $(t,\mu)$ in $[0,T] \times \mathcal P(\R^d)$, the non-linear Fokker-Planck equation \eqref{NL_PDE_FK} admits a unique solution $(\brho_{t,\mu}(s,\cdot))_{s\in [t,T]} $ in $\textcolor{black}{L^{\overline{r}}(B^{-\beta+\vartheta \Gamma}_{p',q'})}$  with $\bar r $ as in \textcolor{black}{\eqref{cond_time_gencase}}, \textcolor{black}{$\vartheta\in [0,1) $} and $\Gamma$ given by \eqref{GAP}, s.t. \textcolor{black}{for all $ s\in [t,T], \brho_{t,\mu}(s,\cdot)\in \mathcal P(\R^d)$}.
\end{lemme}
\begin{proof}
Assume that $\brho^{(1)}_{t,\mu}$ and $\brho^{(2)}_{t,\mu}$ are two possible  solutions of \eqref{NL_PDE_FK}.  
{\color{black}From the last part of the proof of Lemma \ref{lem_ex_duha_gencase}, one can easily check that \textcolor{black}{density functions  solving \eqref{NL_PDE_FK}}, and which lie in $L^{r'}(B^{-\beta}_{p',q'})$, also satisfy the Duhamel formulation \eqref{main_DUHAMEL}}. 
Then, for a.e. $t\le s\le T$, $y$ in $\R^d$,
\begin{eqnarray*}
 \brho^{(1)}_{t,\mu}(s,y) -\brho^{(2)}_{t,\mu}(s,y)
 = -\int_{t}^s dv \Big[\{{\mathcal B}_{\brho^{(1)}_{t,\mu}}(v,\cdot) \brho^{(1)}_{t,\mu}(v,\cdot)-
 {\mathcal B}_{\brho^{(2)}_{t,\mu}}(v,\cdot) \brho^{(2)}_{t,\mu}(v,\cdot)\} \star \nabla p_{s-v}^{\alpha}\Big] (y).
\end{eqnarray*}

Reproducing the stability arguments of the proof of Lemma \ref{FIRST_STAB} (see in particular those leading to \eqref{eq:inter_stabi11}), one can check that
\begin{align*}
| \brho^{(1)}_{t,\mu}(s,\cdot) -\brho^{(2)}_{t,\mu}(s,\cdot) |_{B^{-\beta}_{p',q'}}
&\le C\int_{t}^{s}\frac {dv}{(s-v)^{\frac {-\beta+1}{\alpha}+\frac d{\alpha p}}}|b(v,\cdot)|_{B^\beta_{p,q}}\Big\{|\brho^{(1)}_{t,\mu}(v,\cdot)|_{B^{-\beta}_{p',q'}}|\brho^{(1)}_{t,\mu}(v,\cdot) -\brho^{(2)}_{t,\mu}(v,\cdot)|_{L^1}\\
&+|\brho^{(1)}_{t,\mu}(v,\cdot) -\brho^{(2)}_{t,\mu}(v,\cdot)|_{B^{-\beta}_{p',q'}}\Big\}.
\end{align*}
On the one hand, since $\brho_{t,\mu}^{(1)} $ satisfies \eqref{main_DUHAMEL}, the analysis of Lemma \ref{lem_unifesti_gencase_2} can be reproduced and the controls stated therein remain valid with $\brho_{t,\mu}^\varepsilon $ replaced by $\brho_{t,\mu}^{(1)} $. In particular:
$$ \Bigg( \int_{t}^s\frac{dv}{(s-v)^{\frac{r'}{\alpha}}} |\brho^{(1)}_{t,\mu}(v,\cdot)|_{B^{-\beta}_{p',q'}}^{r'} \Bigg)\le C(s-t)^\theta.$$
On the other hand, the analysis of Lemma \ref{FIRST_STAB} can be reproduced to control the term $|\brho^{(1)}_{t,\mu}(v,\cdot) -\brho^{(2)}_{t,\mu}(v,\cdot)|_{L^1} $. In particular, similarly to \eqref{esti_inter_imp_conv}, one gets,

\begin{align*}
\sup_{v\in (t,s]}\big\{|\brho^{(1)}_{t,\mu}(v,\cdot) -\brho^{(2)}_{t,\mu}(v,\cdot)|_{L^1}\big\} \le C  |b|_{L^r(B^{\beta}_{p,q})} \sup_{v\in (t,s ]}|(\brho^{(1)}_{t,\mu}-\brho^{(2)}_{t,\mu})|_{\textcolor{black}{\Lw^{r'}}\big((t,v],B_{p',q'}^{-\beta}\big)},
\end{align*}
and similarly to \eqref{CAUCHY_BOUND_LP_B_W_SUP} one can show that it holds:
\begin{eqnarray*} 
\label{BD_DIFF_LR4_POIDS}
\sup_{v\in (t,T]}|(\brho^{(2)}_{t,\mu}-\brho^{(1)}_{t,\mu})|_{\textcolor{black}{\Lw^{r'}}\big((t,v],B_{p',q'}^{-\beta}\big)}^{r'} =0.
\end{eqnarray*}
Hence, $\sup_{v\in (t,s]}\big\{|\brho^{(1)}_{t,\mu}(v,\cdot) -\brho^{(2)}_{t,\mu}(v,\cdot)|_{L^1}\big\}=0$ and 
\begin{align*}
| \brho^{(1)}_{t,\mu}(s,\cdot) -\brho^{(2)}_{t,\mu}(s,\cdot) |_{B^{-\beta}_{p',q'}}
&\le C\int_{t}^{s}\frac {dv}{(s-v)^{\frac {-\beta+1}{\alpha}+\frac d{\alpha p}}}|b(v,\cdot)|_{B^\beta_{p,q}}\Big\{|\brho^{(1)}_{t,\mu}(v,\cdot) -\brho^{(2)}_{t,\mu}(v,\cdot)|_{B^{-\beta}_{p',q'}}\Big\}\\
&\le C|b|_{L^r(B_{p,q}^\beta)}\bigg(\int_{t}^s\frac {dv}{(s-v)^{(\frac {-\beta+1}{\alpha}+\frac {d}{\alpha p})r'}}
|\brho^{(1)}_{t,\mu}(v,\cdot) -\brho^{(2)}_{t,\mu}(v,\cdot)|^{r'}_{B^{-\beta}_{p',q'}}\bigg)^{\frac 1{r'}}.
\end{align*}
Taking the $\bar r $ exponent, integrating in time and using the H\"older inequality and the Fubini theorem yield the statement. These arguments are similar to those at the end of the proof of Lemma \ref{FIRST_STAB}.

%
%
\end{proof}

\subsection{From the Fokker-Planck equation to the martingale problem}
\label{sec_WP_SDE}

To prove the  \textcolor{black}{weak well-posedness part of} Theorem \ref{THM_GEN}, it remains to relate the non-linear Fokker-Planck equation to the non-linear martingale problem.

We first establish the existence of a weak solution to \eqref{main} from the tightness \textcolor{black}{of $(\mathbf P^{\varepsilon_k})_{k\ge 1}$, where $\mathbf P^{\varepsilon_k} $ stands for the solution to the non-linear martingale problem related to \textcolor{black}{the mollified} \textcolor{black}{McKean-Vlasov SDE} \eqref{main_smoothed} with mollification parameter $\varepsilon_k \underset{k}{\rightarrow} 0$}. \textcolor{black}{We then prove that any related limit point actually solves the martingale problem associated with \eqref{main}}. Those arguments are detailed in Proposition \ref{prop_ExistenceMain} below. 
The uniqueness of the solution will then be derived through an enhanced (to the stable pure jump case) Krylov-R\"ockner like criterion, \textcolor{black}{see Proposition \ref{prop_UniquenessMain_gencas}}.
Both \textcolor{black}{the existence and uniqueness} results rely on a uniform estimate for the mollified non-linear drift $\mathcal B^\varepsilon_{\brho_{t,\mu}^\varepsilon}$  involved in \eqref{main_smoothed} which, collecting the estimates established in Section \ref{sec_nl_FK} reads as follows:

\begin{lemme}\label{LEMME_FOR_TIGHTNESS_AND_IDENTIFICATION}
Assume that \eqref{cond_gencase} holds. For any $(t,\mu)$ in $[0,T]\times \mathcal P(\R^d)$, any $\varepsilon >0$, the mollified non-linear drift $\mathcal B^\varepsilon_{\brho_{t,\mu}^\varepsilon}$ in \eqref{main_smoothed} is in $L^{r_0}(B^0_{\infty,1})$ for all $r_0$ in $\Big(0,\big(\frac 1{r}+\frac 1{\overline{r}}\big)^{-1}\Big]$, for $\overline{r}$ satisfying \eqref{cond_time_gencase} and the following estimate holds:
\begin{eqnarray}\label{unifestim_drift_gencase}
\forall t\le S\le T,\,|\mathcal B^\varepsilon_{\brho_{t,\mu}^\varepsilon}|_{L^{r_0}((t,S],B^0_{\infty,1})}\le C(S-t)^{\frac 1{r_0}-\Big(\frac 1{r}+\frac 1{\overline{r}}\Big)}|b|_{L^r(B^{\beta}_{p,q})}|\brho_{t,\mu}^\varepsilon|_{L^{\overline{r}}((t,S],B^{-\beta}_{p',q'})}.
\end{eqnarray}
\end{lemme}
\begin{proof}
From the Young inequality \eqref{YOUNG}, one gets for all $s\in (t,T] $:
$$|\mathcal B^\varepsilon_{\brho_{t,\mu}^\varepsilon}(s,\cdot)|_{B^0_{\infty,1}}\le \textcolor{black}{c_{\mathbf{Y}}}|b^\varepsilon(s,\cdot)|_{B_{p,q}^\beta} 
|\brho_{t,\mu}^\varepsilon(s,\cdot)|_{B_{p',q'}^{-\beta}}.
$$
Take now $r_0$ as indicated, then $r>r_0 $ and use the H\"older inequality to derive:
$$|\mathcal B^\varepsilon_{\brho_{t,\mu}^\varepsilon}|_{L^{r_0}((t,S],B^0_{\infty,1})}\le \textcolor{black}{c_{\mathbf{Y}}}|b^\varepsilon|_{L^r(B^{\beta}_{p,q})}\Big(\int_t^T ds |\brho_{t,\mu}^\varepsilon(s,\cdot)|_{B_{p',q'}^{-\beta}}^{r_0\frac{r}{r-r_0}}\Big)^{\frac 1{r_0}-\frac {1}r}.$$
Since $r_0\in\big(0,\big(\frac 1{r}+\frac 1{\overline{r}}\big)^{-1}\big]$, then $\bar r\ge \frac{rr_0}{r-r_0} $. If the equality holds the claim directly follows from Lemma \ref{lem_unifesti_gencase_2}. If $\bar r>\frac{rr_0}{r-r_0} $ an additional H\"older inequality \textcolor{black}{in time} is needed. Indeed,
$$|\mathcal B^\varepsilon_{\brho_{t,\mu}^\varepsilon}|_{L^{r_0}((t,S],B^0_{\infty,1})}\le \textcolor{black}{c_{\mathbf{Y}}}|b^\varepsilon|_{L^r(B^{\beta}_{p,q})}\Big(\int_t^T ds |\brho_{t,\mu}^\varepsilon(s,\cdot)|_{B_{p',q'}^{-\beta}}^{\bar r}\Big)^{\frac {1}{\bar r}} (T-t)^{\frac 1{r_0}-(\frac 1r+\frac1{\bar r})}.$$
\end{proof}

\paragraph{Existence results}

We recall that a measure $\mathbf P$ (a probability measure on the canonical space $\Omega_\alpha$) solves the non-linear martingale problem related to \eqref{main} on $[t,T] $ if:
\begin{itemize}
\item[(i)] $\mathbf P\circ x(t)^{-1}=\mu$;
\item[(ii)] for a.a. $s\in(t,T]$, $\mathbf P\circ x(s)^{-1}$ is absolutely continuous w.r.t. Lebesgue measure and its density belongs to $L^{r'}((t,T],B_{p',q'}^{-\beta})$.
\item[(iii)] for all $f$ in $\mathcal C^1([t,T],\mathcal C_{\textcolor{black}{0}}^2(\mathbb R^d))$, the process
\begin{equation}\label{MP_NL}
\Bigg\{f(s,x(s))-f(t,x(t))-\int_t^s \Big(\textcolor{black}{\partial_v f(v,x(v))}+\mathcal B_{\mathbf P\circ x(v)^{-1}}(v,x(v)) \cdot {\color{black}\nabla}
 f(v,x(v))+L^\alpha(f)(v,x(v))\Big)\,dv\Bigg\}_{t\le s\le T},\tag{${\rm MP}_{{\rm NL}}$}
\end{equation}
is a $\mathbf P$ martingale. 
\end{itemize}

\begin{rem}
We point out that in our singular drift setting it is rather natural to require some smoothness properties on the marginal laws of the canonical process of $\mathbf P$. Those are precisely the ones which allow to define properly almost everywhere the non-linear drift in \eqref{MP_NL}. Pay attention anyhow that even in this setting $\mathcal B_{\mathbf P\circ x(v)^{-1}}(v,x(v))$ is still potentially singular in time and we cannot directly appeal to well known results to ensure well-posedness.

Also, we will \textcolor{black}{from now on} identify the marginal law and its density, i.e. if $\mathbf P( x(v)\in dx):=\mathbf P_{t,\mu}(v,dx)=\brho_{t,\mu}(v,x)dx $, then we denote $\mathcal B_{\mathbf P\circ x(v)^{-1}}(\textcolor{black}{v},\cdot)=\mathcal B_{\brho_{t,\mu}}(v,\cdot) $.
\end{rem}

\begin{prop}\label{prop_ExistenceMain} Under the assumption \eqref{cond_gencase}, the solution to the non-linear martingale problem related to \eqref{main_smoothed} converges, in $\mathcal P(\Omega_\alpha)$ equipped with its weak topology, to a solution to the non-linear martingale problem related to \eqref{main}. 
\end{prop}

\begin{proof}[Proof of Proposition \ref{prop_ExistenceMain}] Let $(\mathbf P^{\varepsilon_n})_{n\ge 0}$ be the sequence of solutions to the non-linear martingale problem related to \eqref{main_smoothed} where $\varepsilon_n \underset{n}{\rightarrow} 0 $. 

\noindent\textit{Tightness.} 
\textcolor{black}{We recall that if $\alpha=2 $, the Kolmogorov criterion is a necessary and sufficient condition to derive tightness of the sequence $ (\mathbf P^{\varepsilon_n})_{n\ge 0}$ in $\mathcal C([t,T],\R^d) $, see e.g. \cite[Theorem 7.3]{Billingsley-99}. It reads that for any $\eta>0 $ there exist $a$ and $n_0$ s.t. for all $n\ge n_0 $,
\begin{equation}\label{I_K}
\mathbf P^{\varepsilon_n}[|X^{\varepsilon_n,t,\mu}_t|=|\xi|\ge a]\le \eta,
\end{equation}
and for all $\eta>0,\epsilon>0$ there exist $\delta\in (0,1) $ and $n_0$ s.t. for all $n\ge n_0 $,
\begin{equation}\label{P_K}
 \mathbf P^{\varepsilon_n}[\sup_{|s-v|\le \delta}|X^{\varepsilon_n,t,\mu}_s-X^{\varepsilon_n,t,\mu}_v|\ge \epsilon] \le \eta .
 \end{equation}
On the other hand, when $\alpha<2 $, the Aldous tightness criterion provides the tightness of $ (\mathbf P^{\varepsilon_n})_{n\ge 0}$ in $\mathbb D([t,T],\R^d) $, see e.g. \cite[Theorem 16.10]{Billingsley-99}. The criterion reads as follows:
\begin{equation}\label{I_A}
\lim_{a\rightarrow\infty}\mathbf P^{\varepsilon_n}[\sup_{s\in [t,T]} |X_s^{\varepsilon_n,t,\mu}|\ge a]=0,
\end{equation}
and for all $\eta>0,\epsilon>0$ there exist $\delta_0,n_0$ s.t. for all $0\le \delta\le \delta_0, n\ge n_0 $ and any stopping time $t\le \tau\le T $:
\begin{align}\label{P_A}
\mathbf P^{\varepsilon_n}[|X^{\varepsilon_n,t,\mu}_{\tau+\delta}-X^{\varepsilon_n,t,\mu}_\tau|\ge \epsilon] \le \eta .
\end{align}
Observe now that
the integrability conditions \eqref{I_K} or \eqref{I_A} are indeed fulfilled in the corresponding case. It is direct for \eqref{I_K}. Also, from the dynamics in \eqref{main_smoothed}
\begin{align*}
|X_s^{\varepsilon_n,t,\xi}| &\le |\xi| + \int_t^s |\mathcal B^{\varepsilon_n}_{\bmu^{\varepsilon_n,t,\mu}_v}(v,X_v^{\varepsilon_n,t,\xi})| dv + |\mW_s-\mW_t|\\
&\le |\xi| + \int_t^s \|\mathcal B^{\varepsilon_n}_{\bmu^{\varepsilon_n,t,\mu}_v}(v,\cdot)\|_{L^\infty} dv + |\mW_s-\mW_t|
\underset{\eqref{EMBEDDING}}{\le} |\xi| + \int_t^s \|\mathcal B^{\varepsilon_n}_{\bmu^{\varepsilon_n,t,\mu}_v}(v,\cdot)\|_{B_{\infty,1}^0} dv + |\mW_s-\mW_t|\\
&\le |\xi| +|\mathcal B^{\varepsilon_n}_{\brho_{t,\mu}^{\varepsilon_n}}|_{L^{r_0}((t,s],B^0_{\infty,1})}(s-t)^{\frac 1{r_0'}}+ |\mW_s-\mW_t|\\
&\le |\xi| +C|b|_{L^r(B^{\beta}_{p,q})}|\brho_{t,\mu}^{\varepsilon_n}|_{L^{\overline{r}}((t,T],B^{-\beta}_{p',q'})}(T-t)^{1-\Big(\frac 1{r}+\frac 1{\overline{r}}\Big)}+|\mW_s-\mW_t|,
\end{align*}
using \eqref{unifestim_drift_gencase} from Lemma \ref{LEMME_FOR_TIGHTNESS_AND_IDENTIFICATION} for the last inequality. From the uniform control, w.r.t. $\varepsilon_n $, of Lemma \ref{lem_unifesti_gencase_2} we eventually obtain:
\begin{align*}
|X_s^{\varepsilon_n,t,\xi}| \le |\xi| +C|b|_{L^r(B^{\beta}_{p,q})}(T-t)^{1-\Big(\frac 1{r}+\frac 1{\overline{r}}\Big)}+|\mW_s-\mW_t|,
\end{align*}
from which \eqref{I_A} now easily follows. Observe that similar arguments yield:
\begin{align*}
|X_{\tau+\delta}^{\varepsilon_n}-X_{\tau}^{\varepsilon_n}|\le C|b|_{L^r(B^{\beta}_{p,q})}\delta^{1-\Big(\frac 1{r}+\frac 1{\overline{r}}\Big)}+|\mW_{\tau+\delta}-\mW_{\tau}|,
\end{align*}
and the statements \eqref{P_K} and \eqref{P_A} then easily follow from the properties of the stable driving process}. 


\noindent\textit{Limit points.} Let $(\mathbf P^{\varepsilon_k})_k$ be a converging subsequence and let $\mathbf P$ be its limit. In addition to the weak convergence of $(\mathbf P^{\varepsilon_k})_k$ toward\textcolor{black}{s} $\mathbf P$, according to Lemma \ref{FIRST_STAB}, the 
{\color{black}marginal distributions} $\big(\mathbf P^{\varepsilon_k}_{t,\mu}(s,dx)=\brho^{\varepsilon_k}_{t,\mu}(s,x)\,dx\big)_k$ \textcolor{black}{converge} strongly toward\textcolor{black}{s} $\mathbf P_{t,\mu}(s,dx)=\brho_{t,\mu}(s,x)\,dx$ in $L^{r'}(B^{-\beta}_{p',q'})$. This strong convergence implies the convergence of $(\mathcal B^{\varepsilon_k}_{\brho^{\varepsilon_k}_{t,\mu}})_k$ toward\textcolor{black}{s} $\mathcal B_{\brho_{t,\mu}}$ in $L^{r_0}(L^\infty)$. Indeed, for $r_0$ as in \eqref{unifestim_drift_gencase}, and proceeding as in the proof of Lemma \ref{LEMME_FOR_TIGHTNESS_AND_IDENTIFICATION}:
\begin{eqnarray*}
&&|\mathcal B^{\varepsilon_k}_{\brho^{\varepsilon_k}_{t,\mu}}-\mathcal B_{\brho_{t,\mu}}|_{L^{r_0}(L^\infty)}\textcolor{black}{\underset{\eqref{EMBEDDING}}{\le}}\textcolor{black}{C} |\mathcal B^{\varepsilon_k}_{\brho^{\varepsilon_k}_{t,\mu}}-\mathcal B_{\brho_{t,\mu}}|_{L^{r_0}(B_{\infty,1}^0)}\\
&\le& C\Bigg(|\brho^{\varepsilon_k}_{t,\mu}|_{L^{\textcolor{black}{\bar{r}}}(B^{-\beta\textcolor{black}{+\vartheta\Gamma}}_{p',q'})}|b^{\varepsilon_k}-b|_{L^{\textcolor{black}{\check r}}(B^{\beta\textcolor{black}{-\vartheta\Gamma}}_{p,q})}+
|b|_{L^r(B^{\beta}_{p,q})}|\brho^{\varepsilon_k}_{t,\mu}-\brho_{t,\mu}|_{L^{\textcolor{black}{\bar{r}}}(B^{-\beta}_{p',q'})}\Bigg)\underset{k}\to 0,
\end{eqnarray*}
\textcolor{black}{where, as in the proof of Lemma \ref{FIRST_STAB}, $\check r=r $ if $r<+\infty$ and any finite $\check r$ large enough if $r=+\infty $, i.e. the conjugate exponent $\check r' $ belongs to the interval indicated after \eqref{LEMME_2_WITHOUT_SING_2}. The previous convergence follows from Proposition \ref{PROP_APPROX} and Lemmas \ref{lem_unifesti_gencase_2} and \ref{FIRST_STAB}}.

This also  yields, as a natural extension, the variational limit.  For all $\Phi\in \textcolor{black}{C_0^\infty}((t,T)\times \R^d,\R^d)$,
\[
\lim_{k\rightarrow \infty}\int_t^T \int \mathcal B^{\varepsilon_k}_{\brho^{\varepsilon_k}_{t,\mu}}(s,x)\cdot \Phi(s,x)\,dx\,ds=
\int_t^T \int \mathcal B_{\brho^{}_{t,\textcolor{black}{\mu}}}(s,x)\cdot \Phi(s,x)\,dx\,ds\textcolor{black}{.}
\]
This last convergence is sufficient to ensure that (\textcolor{black}{see e.g. Lemma 5.1 in \cite{EthierKurtz-85}}): for any $0\le t_1\le \cdots \le t_i\le \cdots \le t_n\le t\le s\le T$, $\Psi_1,\cdots,\Psi_n$ continuous bounded, and for any $\phi$ of class $ C^2_{\textcolor{black}{0}} \textcolor{black}{(\R^d,\R)}$, 
\[
\mathbb E_{\mathbf P^{\varepsilon_k}}\!\!\left[\Pi_{i=1}^n\Psi_i(x(t_i))\!\int_t^s\!\mathcal B^{\varepsilon_k}_{\brho^{\varepsilon_k}_{t,\mu}}(v,x(v))\cdot \nabla \phi(x(v))\,dv\right]\to_k
\!\!\mathbb E_{\mathbf P^{}}\left[\Pi_{i=1}^n\Psi_i(x(t_i))\!\int_t^s\!\mathcal B_{\brho^{}_{t,\textcolor{black}{\mu}}}(v,x(v))\cdot \nabla \phi(x(v))\,dv\right].
\]
This \textcolor{black}{precisely gives} that $\mathbf P$ is solution of the non-linear martingale problem related to \eqref{main}.
\end{proof}

\paragraph{Weak uniqueness results}
\begin{prop}[Uniqueness result]\label{prop_UniquenessMain_gencas} Under the assumption \eqref{cond_gencase}, for any $(t,\mu)$ in $[0,T] \times \mathcal P(\R^d)$ the SDE \eqref{main} admits at most one weak solution 
 s.t. its marginal laws $\big(\bmu^{t,\mu}_s(\cdot)\big)_{s\in [t,T]} $ have a density for a.e. $s$ in $(t,T]$, i.e. $\bmu^{t,\mu}_s(dx) = \brho_{t,\mu}(s,x)dx$
and
$\brho_{t,\mu}\in L^{\bar r}(B_{p',q'}^{-\beta})$ for any $\bar {r}$ satisfying \eqref{cond_time_gencase}. 
\end{prop}
\begin{proof}
According to Proposition \ref{WP_FK},
 any pair of solutions to the non-linear martingale problem related to \eqref{main} whose 
 {\color{black}{(time) marginal distributions}} lie in $L^{\bar{r}}(B^{-\beta}_{p',q'})$ have the same 
 {\color{black}marginal distributions.} The non-linear drift $\mathcal B_{\brho_{t,\mu}}(s,\cdot)$ is therefore, for a.e. $s$ in $[t,T]$, given by $b(s,\cdot)\star \brho_{t,\mu}(s,\cdot)$, $\brho_{t,\mu}$ being the unique solution in $L^{\textcolor{black}{\bar{r}}}(B^{-\beta}_{p',q'})$ of the non-linear Fokker-Planck equation \eqref{NL_PDE_FK}.
 
 As $\brho_{t,\mu}$ is in $L^{\textcolor{black}{\bar {r}}}(B^{-\beta}_{p',q'})$, one can further check that, similarly to \eqref{unifestim_drift_gencase}:
\begin{eqnarray}
|\mathcal B_{\brho_{t,\mu}}|_{L^{{\color{black}r_0}}((t,T),B_{\infty,1}^{0})}\le \cv |b|_{L^r(B^{\beta}_{p,q})}|\brho_{t,\mu}|_{L^{r_0}(B^{-\beta}_{p',q'})} \le \cv |b|_{L^r(B^{\beta}_{p,q})}|\brho_{t,\mu}|_{L^{{\color{black}\bar r}}(B^{-\beta}_{p',q'})}\label{drift_estim_gencase_a},
\end{eqnarray}
for any ${\color{black}r_0}$ and ${\color{black}\bar r}$ such that $(\textcolor{black}{ r_0})^{-1}=({\color{black}\bar r})^{-1}+(r)^{-1}$ and $r'\le {\color{black}\bar r}<(-\beta/\alpha+d/\alpha p)^{-1}$. Thus, the drift $\mathcal B_{\brho_{t,\mu}}$ in \eqref{main} is well defined and all weak solutions to \eqref{main} whose density belongs to  $L^{\textcolor{black}{\bar{r}}}(B^{-\beta}_{p',q'})$ share the same drift $\mathcal B_{\brho_{t,\mu}}$. Uniqueness of the non-linear martingale problem thus follows from uniqueness of the solution $\mathbf P$ (a probability measure on the canonical space $\Omega_\alpha$) to the (linear) martingale problem related to $(\mu, L^\alpha,F)$, \textcolor{black}{where $F=\mathcal B_{\brho_{t,\mu}} $}, defined as:
\begin{itemize}
\item[(i)] $\mathbf P\circ x(t)^{-1}=\mu$;
\item[(ii)] for a.a. $s\in(t,T]$, $\mathbf P\circ x(s)^{-1}$ is absolutely continuous w.r.t. Lebesgue measure;
\item[(iii)] for all $f$ in $\mathcal C^1([t,T],\mathcal C_{\textcolor{black}{0}}^2(\mathbb R^d))$, the process
\begin{equation}\label{MP}
\Bigg\{f(s,x(s))-f(t,x(t))-\int_t^s \Big(\textcolor{black}{\partial_v f(v,x(v))}+F(v,x(v)) \cdot {\color{black}\nabla}
 f(v,x(v))+L^\alpha(f)(v,x(v))\Big)\,dv\Bigg\}_{t\le s\le T},\tag{MP}
\end{equation}
is a $\mathbf P$ martingale. 
\end{itemize}
The following Lemma gives sufficient condition for the solution to the (linear) martingale problem \eqref{MP} to be unique.
\begin{lemme}\label{KRStable} Assume that $F$ is in $L^{\sc}(\textcolor{black}{L^\infty})$ with $1\le \sc\le \infty$ and 
\begin{equation}\label{CR_KR_INFTY}
\textcolor{black}{\frac \alpha{\sc}<\alpha-1}.
\end{equation}
Then the (\textcolor{black}{linear}) martingale problem related to the triplet $(\mu, L^\alpha,F)$ admits at most one solution.
\end{lemme}
To conclude, it thus only remains to check wether the condition stated in the above Lemma is satisfied by the non-linear drift $\textcolor{black}{F=}\mathcal B_{\brho_{t,\mu}}$. This latter is immediately implied by \eqref{drift_estim_gencase_a} as we have that $(s,x)\mapsto \mathcal B_{\brho_{t,\mu}}(s,x)$ is in $L^{{\color{black}r_0}}(B_{\infty,1}^{0})$ and, according to \eqref{EMBEDDING}, in $L^{{\color{black}r_0}}(L^\infty)$. From condition \eqref{cond_gencase} \textcolor{black}{and the definition of $\Gamma $ in \eqref{GAP}, taking $ \bar r\in \Bigg[\big( \frac 1\alpha(-\beta+\frac dp+\vartheta \Gamma)\big)^{-1},\big( \frac 1\alpha(-\beta+\frac dp)\big)^{-1} \Bigg),\ \vartheta \in (0,1)$}, it hence holds that  $\alpha/{\color{black}r_0} < \alpha-1$ which gives the previous criterion with $ \sc = {\color{black}r_0}$.
 \end{proof}
\begin{rem} \,
\begin{itemize}
\item (On the formulation of the  linear martingale problem \eqref{MP}). While the conditions $(i)$ and $(iii)$ are rather classical (see e.g. \cite{StrookVaradhan-06}, \cite{EthierKurtz-85}), the condition $(ii)$ is enforced in order to exhibit a short demonstration for the uniqueness criterion stated in Lemma \ref{KRStable} below. This condition enables  notably to take advantage of the results obtained in \cite{CdRM-20}.
\item (On the Krylov-R\"ockner like criterion in Lemma \ref{KRStable})\footnote{Observe that when $\alpha=2$, the condition \eqref{CR_KR_INFTY}, i.e. $\frac{\alpha}{\sc}<\alpha-1 $, indeed reads as a special case of the Krylov-R\"ockner criterion $\frac{2}{q}+\frac{d}{p}<\alpha-1=2-1=1 $ for $p=\infty,\sc=q$, from which strong uniqueness is also derived in the Brownian setting in \cite{kryl:rock:05}. Whence the name. The condition \eqref{CR_KR_INFTY}
here leads to weak uniqueness.
}. The criterion \eqref{CR_KR_INFTY} is a particular case of a condition appearing in \cite{CdRM-20}. Namely, 
it was established in the quoted work a uniqueness result for a modified form of the (linear) martingale problem associated with a drift $\mathcal B_{\brho_{t,\mu}}$ in $L^{\sc}(B^{\textcolor{black}{\beta_{{\rm lin}}}}_{\ell,m})$ with $m\in[1,\infty]$ and 
\[
\textcolor{black}{\beta_{\rm lin}}\in \textcolor{black}{\Big(}\frac{1}{2}\big(1-\alpha+\frac{d}{\ell}+\frac{\alpha}{\sc}\big),0\Big).
\] 
\textcolor{black}{The restriction $ \beta_{\rm lin}<0$ in this paper was essentially motivated to emphasize the distributional setting. However, the arguments and techniques initially introduced in \cite{CdRM-20} can be adapted to handle the case $ \beta_{\rm lin}=0$. This precisely allows to derive a weak uniqueness result for a "classical" martingale problem under the aforementioned Krylov-R\"ockner type criterion \eqref{CR_KR_INFTY}. This is the purpose of Lemma \ref{KRStable}}.
\end{itemize}
\end{rem}

\begin{proof}[Proof of Lemma \ref{KRStable}]
{\color{black}According to \eqref{EMBEDDING}, \textcolor{black}{the statement of the lemma will follow if we prove it for} a drift $F$ in $L^{\sc}((t,T],B^{0}_{\infty,1})$. Taking $\epsilon=1-\gamma>0$, meant to be very small, with 
\[
\gamma\in \bigg[\frac{1}{2}\big(3-\alpha 
+\frac{\alpha}{\sc}\big),1\bigg)\Leftrightarrow 1-\gamma\in \bigg(0,\frac{1}{2}\big(\alpha-1 
-\frac{\alpha}{\sc}\big)\bigg],
\]
we further recover the initial setting of \cite{CdRM-20}, \textcolor{black}{i.e.  from \eqref{BesovEmbedding}} $F\in L^{\sc}((t,T],B^{-\epsilon}_{\infty,1})$. \textcolor{black}{(Importantly, it
 is precisely the condition \eqref{CR_KR_INFTY}  which allows 
 {\color{black}this lifting}.}) {\color{black}Next,} from Theorem 2 in \cite{CdRM-20}, it follows that for any $t<T_0<\infty$ and any smooth bounded function $h$ on $[t,T_0]\times \mathbb R^d$, the  backward PDE
\begin{equation*}
\left\{
\begin{aligned}
&\partial_s u+F\cdot {D}u+L^\alpha u=-h\,\text{on}\,\textcolor{black}{(t,T_0)}\times\mathbb R^d,\\
&u(T_0,\cdot)=0,
\end{aligned}
\right.
\end{equation*}
admits a unique bounded (mild) solution in $C^{0,1}([0,T_0]\times\R^d,\R)$ for which ${D}u\in C_b^0([0,T_0],B_{\infty,\infty}^{\theta-1-\varepsilon'}) $
 for $\theta:=\gamma-1+\alpha-{\alpha}/{\sc}$ and for any $\varepsilon'>0 $ \textcolor{black}{, again meant to be small such that $\theta-1-\varepsilon'>0$ (this can be achieved from \eqref{CR_KR_INFTY} which gives $\theta -1=(\gamma-1)+\alpha- \alpha/{\sc}-1>1-\gamma=\varepsilon$ for $\varepsilon$ small enough and taking $\varepsilon'=\varepsilon$)}. In particular $u$ {\color{black}satisfies}: 
 for all $\varphi$ in $\textcolor{black}{C_0^\infty}((t,T_0]\times\mathbb R^d)$,
\begin{equation}\label{EDP_WEAK}
\int^{T_0}_{t} ds \int dx\textcolor{black}{\Big[}\,u(s,x)(\partial_s -L^\alpha)(\varphi)(s,x)-   \Big(F(s,x)\cdot Du(s,x)\Big) \textcolor{black}{\varphi}(s,x)- h(s,x) \textcolor{black}{\varphi}(s,x)\textcolor{black}{\Big]}=0.
\end{equation}
\textcolor{black}{Note that, according to the control previously obtained in \cite{CdRM-20}, and as $F$ is here in $L^{\sc}((t,T], \textcolor{black}{B_{\infty,1}^0}
)$, the first order term, $\Big(F\cdot {D}u\Big) \textcolor{black}{\varphi}$, is well-defined.}
Introduce now $\{\phi_n \}_n$, a sequence of compactly supported time-space mollifiers such that, for all $n$, $\textcolor{black}{\phi_n(\cdot)=n^{(d+1)}\phi(n\cdot)}$ where $\phi$ is {\color{black}a} symmetric, non-negative \textcolor{black}{function}, equal to $1$ on the unit ball of $\mathbb R^{d+1}$ and \textcolor{black}{s.t.} $\int ds\int dz\,\phi(s,z)=1$. Setting $u_n=\phi_n \otimes u $ - where, borrowing the notation of \cite{Friedman-64}, $\otimes$ stands for the time-space convolution product -  we now claim that \textcolor{black}{for $\mathbf P_1$ and $\mathbf P_2$ two solutions to the martingale problem related to $(\mu, L^\alpha,F)$ it holds that $\mathbf P^i$-a.s.} (for $i=1,2$), $\textcolor{black}{t<s}\le T_0$,
\[
\int_s^{T_0} (\partial_v+F\cdot {D} +L^\alpha)u_n(v,x(v)) \ dv=-\int_s^{T_0} h_n(v,x(v))dv+\int_s^{T_0} R_n(v,x(v))dv,
 \]
where $h_n=\phi_n \otimes h$ and
\[
R_n(v,x):=F(v,x)\cdot Du_n(v,x)-\phi_n\otimes \Big(F\cdot {D}u\Big)(v,x). 
\]
 This expression arises from 
writing the action of the differential operator on $u_n$ which gives, for all $(v,x)$ in $[\textcolor{black}{s},T_0]\times\mathbb R^d$, \textcolor{black}{for $n$ large enough}\footnote{\textcolor{black}{so that ${\rm supp}( \phi_n) \subset B(0,(s-t)/2)\times B_{\R^d}(0,1)\subset \R\times \R^d$}},
  \begin{eqnarray*}
&&(\partial_v+F(v,x)\cdot {D} +L^\alpha)(u\otimes \phi_n(v,x))\\
&=& (\partial_v+F(v,x)\cdot {D} +L^\alpha) \Big(\int_{}dw\int dz\, u(w,z) \phi_n(v-w,x-z)\Big)\\
&=&\int_{}dw\int dz \ u(\textcolor{black}{w},z) (\partial_v +L^\alpha)(\phi_n)(v-w,x-z)+F(v,x)\cdot {D}\Big(\int_{}dw\int dz \ u(w,z) \phi_n(v-w,x-z)\Big)\\
&=&\int_{}dw\int dz \ u(w,z) (-\partial_{\textcolor{black}{w}} +L^\alpha)(\phi_n)(v-w,x-z)-\int_{}dw\int dz\, u(\textcolor{black}{w},z) F(v,x)\cdot ({D}_{\textcolor{black}{z}}\phi_n)(v-w,x-z)\\
&=&-\int_{}dw\int dz \left\{F(w,z)\cdot {D}u(\textcolor{black}{w},z)+h(\textcolor{black}{w},z)\right\}\phi_n(v-w,x-z)\\
&&-\int_{}dw\int dz\, u(\textcolor{black}{w},z) F(v,x)\cdot ({D}_{\textcolor{black}{z}}\phi_n)(v-\textcolor{black}{w},x-z)\\
&=&\int_{}dw\int dz  [F(v,x)-F(w,z)]\cdot {D}u(w,z)  \,\phi_n(v-w,x-z)-h_n(v,x),
\end{eqnarray*}
\textcolor{black}{where we used \eqref{EDP_WEAK} with $\varphi(w,z)=\phi_n(v-w,x-z) $, which indeed belongs to ${C_0^\infty}((t,T_0]\times\mathbb R^d) $ for $n$ large enough (and fixed $s>t $), for the last but one equality}.
Also, the last equality follows from an integration by parts.
 The symmetry of $\phi_n$  eventually brings
\[
\int_{}dw\int dz  [F(v,x)-F(w,z)]\cdot {D}u(w,z)  \phi_n(v-w,x-z)=F(v,x)\cdot Du_n(v,x)-\phi_n\otimes \Big(F\cdot {D}u\Big)(v,x).
\]
Denoting by $\mathbf P_i^{s,x_0}$ the conditional probability of $\mathbf P^{s,x_0}_i$ given $\{x(s)=x_0\}$, and taking $f(s,x)=u_n(s,x)$ in \eqref{MP}, it follows that, for $i=1,2$,
\begin{equation}\label{TO_WEAK_LIMIT}
\mathbb E_{\mathbf P_i^{s,x_0}}\left[\textcolor{black}{u_n(T_0,x(T_0))}+\int_s^{T_0}h_n(v,x(v))\,dv\right]\textcolor{black}{-	u_n(s,x_0)}=\mathbb E_{\mathbf P_i^{s,x_0}}\left[\int_s^{T_0}R_n(v,x(v))\,dv\right].
\end{equation}
 Letting $n$ \textcolor{black}{tend} to infinity, by continuity and boundedness of $u$,  $u_n(T_0,x(T_0))$ converges a.s. to $0$ \textcolor{black}{as well as $\E[u_n(T_0,x(T_0))] $  (0 is the terminal condition of the backward equation). Similarly, $u_n(s,x_0)$ converges to $u(s,x_0) $. Also, again} by continuity and boundedness of $h$,
\[
\lim_{n}\mathbb E_{\mathbf P_i^{s,x_0}}\left[u_n(T_0,x(T_0))+\int_s^{T_0}h_n(v,x(v))\,dv\right]-u_n(s,x_0)=\mathbb E_{\mathbf P_i^{s,x_0}}\left[\int_s^{T_0}h(v,x(v))\,dv\right]-u(s,x_0).	
\]
We next claim that 
\begin{equation}\label{CTR_REMAINDER_MP}
\lim_{n}\mathbb E_{\mathbf P_i^{s,x_0}}\left[\int_s^{T_0} R_n(v,x(v))dv\right]= 0.
\end{equation}
To check this claim, following $(ii)$ denote by $\varrho^i(s,x_0;v,\cdot)$ the density function of \textcolor{black}{the canonical process of}
$\mathbf P_i^{s,x_0}$ at time $v$, and observe now that: 
\begin{eqnarray}
&&\Big{|}\E_{\mathbf P_i^{s,x_0}}\left[\int_s^{T_0} R_n(v,x(v)) \ dv \right]\Big{|}\notag\\
&= &  \Big{|} \int_s^{T_0} dv \int dy \ F(v,y)\cdot {D}u_n(v,y)\varrho^i(s,x_0;v,y)
-\int_s^{T_0} d\textcolor{black}{v} \int d\textcolor{black}{y} \ \phi_n\otimes \big(F\cdot {D}u\big)(v,y)\varrho^i(s,x_0;\cdot,\cdot)(v,y)\Big{|}\notag\\
&= &  \Big| \int_s^{T_0} dv \int dy \ F(v,y)\cdot {D}u_n(v,y)\varrho^i(s,x_0;v,y)-\int_{s}^{T_0} dv \int dy \ F(v,y)\cdot {D}u(v,y)\phi_n\otimes\varrho^i(s,x_0;\cdot,\cdot)(v,y)\Big|\notag\\
&= &    
\Big|\int_s^{T_0} dv \int dy \ F(v,y)\cdot \big[{D}u_n(v,y)-{D}u(v,y)\big]\varrho^i(s,x_0;v,y)\notag\\
&&+\int_s^{T_0} dv \int dy \ F(v,y)\cdot {D}u(v,y)\big[\varrho^i(s,x_0;v,y)-\phi_n\otimes\varrho^i(s,x_0;\cdot,\cdot)(v,y)\big]\Big|,\label{THE_SPLIT_DU_REMAINDER}
\end{eqnarray}
\textcolor{black}{expanding the convolution and using the symmetry of $\phi_n $ as well as the Fubini theorem for the second equality}.
At this stage, \textcolor{black}{\eqref{CTR_REMAINDER_MP}} is a consequence of the convergence of the mollified sequences $Du_n$ and $\phi_n\otimes\varrho^i(s,x_0;\cdot,\cdot)$ to $Du $ and $\varrho^i(s,x_0;\cdot,\cdot)$ respectively.

\textcolor{black}{For the first term in the r.h.s. of \eqref{THE_SPLIT_DU_REMAINDER} we will use the uniform H\"older continuity of the gradient in time and space (see Theorem 2 in \cite{CdRM-20}) which gives that for all $(\textcolor{black}{v},y)\in [s,T_0]\times \R^d $,
\begin{align*}
|{D}u_n(v,y)-{D}u(v,y)|&\le \int_{\R^{d+1}} dw dz |Du(v-w,y-z)-Du(v,y)|\phi_n(w,z) \\
&\le C[Du(\cdot,\cdot)]_{(\theta-1-\varepsilon')/\alpha ,\theta-1-\varepsilon'} \textcolor{black}{n^{-(\theta-1-\varepsilon')/\alpha}} \underset{n}{\longrightarrow } 0.
\end{align*}}

\textcolor{black}{Applying then the H\"older inequality, we derive
\begin{eqnarray}
&& 	\Big|\int_s^{T_0} dv \int dy \ F(v,y)\cdot \big[{D}u_n(v,y)-{D}u(v,y)\big]\varrho^i(s,x_0;v,y)\Big|\notag\\ 
&\le & |F|_{L^{\sc}(L^\infty)}\left(\int_s^{T_0}\,dv \left(\int \,dy \ \big|{D}u_n(v,y)-{D}u(v,y)\big|\varrho^i(s,x_0;v,y)\right)^{\sc'}\right)^{1/\sc'}\notag\\
& \le & |F|_{L^{\sc}(L^\infty)}
\left(\int_s^{T_0}\,dv \ \big|{D}u_n(v,\cdot)-{D}u(v,\cdot)\big|_{L^\infty}^{\sc'}\right)^{1/\sc'}\underset{n\rightarrow\infty}{\longrightarrow }0.
\label{FIRST_TERM_DU_SPLIT_DU_REMAINDER}
\end{eqnarray}
}

\textcolor{black}{To handle the second term in the r.h.s. of \eqref{THE_SPLIT_DU_REMAINDER} we will use} the continuity property of the shift operator \textcolor{black}{for} Lebesgue norms: for $f\in L^{\check \sc}(\mathbb R,L^\ell)$ with $1\le \ell<\infty$, $1\le \check \sc< \infty$ - non-necessarily equal - the mapping $(w,z)\mapsto |f(\cdot +w,\cdot+z)-f(\cdot,\cdot)|_{L^{\check \sc}(\mathbb R,L^\ell)}$ is continuous.  This property naturally yields 
$|f\otimes\phi_n-f|_{L^{\check \sc}(\mathbb R,L^\ell)}\rightarrow 0$ as $n$ tends to infinity.

\textcolor{black}{Now}, extending by zero $\varrho^i(s,x_0;\cdot,\cdot)$ outside $\textcolor{black}{[s,T_0]}\times\mathbb R^{d}$, $\varrho^i(s,x_0;\cdot,\cdot)$ can be viewed as an element of $L^{\sc'}(L^1)$ and so
 \begin{eqnarray*}
&& \Big|	\int_s^{T_0} dv \int dy \ F(v,y)\cdot {D}u(v,y)\big[\varrho^i(s,x_0;v,y)-\phi_n\otimes\varrho^i(s,x_0;\cdot,\cdot)(v,y)\big]\Big|\\
&\le & |F|_{L^{\sc}(L^\infty)}|{D}u|_{L^\infty(L^\infty)}  \left(\int_s^{T_0}\,dv \left(\int \,dy\Big|\varrho^i(s,x_0;v,\cdot)-\phi_n\otimes\varrho^i(s,x_0;\cdot,\cdot)(v,\cdot)\Big|_{L^1}\right)^{\sc'}\right)^{1/\sc'}\underset{n\rightarrow\infty}{\longrightarrow }0,
\end{eqnarray*}
\textcolor{black}{which, together with \eqref{FIRST_TERM_DU_SPLIT_DU_REMAINDER}, gives \eqref{THE_SPLIT_DU_REMAINDER}}.

Consecutively, for $i=1$ or $i=2$, the quantity
\[
\mathbb E_{\mathbf P_i^{s,x_0}}\left[\int_s^{T_0}h(v,x(v))\,dv\right]
\]
is uniquely given by $u(s,x_0)$ (passing to the limit in \eqref{TO_WEAK_LIMIT}). As $h$ can be taken in a determining class for probability measures, this grants immediately the equality of the 
{\color{black}{marginal distributions}} of $\mathbf P_1$ and $\mathbf P_2$, and by extension  the equality $\mathbf P_1=\mathbf P_2$ on $\Omega_\alpha$. \textcolor{black}{Indeed, as $h$ is a smooth bounded function, the fact that $\mathbb E_{\mathbf P_1^{s,x_0}}\left[\int_s^{T_0}h(v,x(v))\,dv\right]=\mathbb E_{\mathbf P_2^{s,x_0}}\left[\int_s^{T_0}h(v,x(v))\,dv\right] $ easily yields that $\mathbb E_{\mathbf P_1^{s,x_0}}\left[\psi(x(v))\right]=\mathbb E_{\mathbf P_2^{s,x_0}}\left[\psi(x(v))\right]$ for $v\in [s,T_0],\ \psi \in \Psi$ where $\Psi $ is a determining class of functions. This in turn allows to apply Theorem 4.2 in Chapter 4 of \cite{EthierKurtz-85}.}


 }

\end{proof}

\subsection{On strong solutions} The following Proposition \textcolor{black}{gives} conditions on the \textit{parameters} that ensure existence and uniqueness of a strong solution for the McKean-Vlasov SDE \eqref{main}.
\begin{prop}\label{Prop_STRONG_SOL}
Under the assumption \eqref{COS}, for any $(t,\mu)$ in $[0,T] \times \mathcal P(\R^d)$ there exists a unique strong solution to \eqref{main} such that its law $\bmu^{t,\mu}$ lies in $L^{\bar r}(B_{p',q'}^{-\beta})$ for any $\overline{r}$ satisfying \eqref{cond_time_gencase} and such that for a.e. $s$ in $(t,T]$, $\bmu^{t,\mu}_s(dx) = \brho_{t,\mu}(s,x)dx$.
\end{prop}

\begin{proof}
As weak uniqueness holds, strong well-posedness follows from the strong well-posedness of the \emph{linear} version of the McKean-Vlasov SDE \eqref{main} where the law is frozen. This thus ends up to prove that the drift 
\begin{align}\label{DEF_NL_DR}
\mathfrak b: [0,T] \times \R^d \ni (s,x) \mapsto \mathfrak b(s,x) =  \int_{\R^d} b(s,x-y) \bmu_s^{t,\mu}(dy)=\int_{\R^d}{\color{black}dy \ }b(\textcolor{black}{s},x-y)\brho_{t,\mu}(s,y)
,
\end{align}
satisfies a Krylov and R\"ockner type condition, see \cite{kryl:rock:05} if $\alpha=2$, or the criterion in  \cite[Theorem 2.4]{xie:zhan:20} if $\alpha \in (1,2)$. This can be summarized as follows: 
\begin{itemize}
\item[(i)] If $\alpha=2$, then strong well-posedness holds if the drift $\mathfrak b$ defined in \eqref{DEF_NL_DR} lies in $L^{\sc}(L^\ell)$ with $\sc,\ell \in (2,\infty]$ such that
$$\frac 2{\sc} + \frac d\ell < 1.$$
\item[(ii)] If $\textcolor{black}{\alpha\in (1,2)}$, then strong well-posedness holds if the drift $\mathfrak b$ defined in \eqref{DEF_NL_DR} satisfies that $(I-\Delta)^{\gamma / 2}\mathfrak b\in L^{\sc}(L^{\ell})$, \textcolor{black}{or \textcolor{black}{expressed equivalently in terms of Bessel potential space}, $\mathfrak b \in L^{\sc} \textcolor{black}{(}H^{\gamma,\ell}\textcolor{black}{)} $}, where $\gamma,\ell$ and $\sc$ are such that
\begin{equation}\label{COND_II}
\gamma \in \Big(1-\frac \alpha 2,1\Big),\quad \ell \in \Big(\frac{2d}\alpha\vee 2,\infty\Big),\quad \sc \in \Big(\textcolor{black}{\frac{\alpha}{\alpha-1}},\infty\Big),\quad \frac{\alpha}{\sc}+\frac{d}{\ell}<\textcolor{black}{\alpha-1}.
\end{equation}
\end{itemize}
{\color{black}
Note that the condition on $\mathfrak b$ in $(ii)$ \textcolor{black}{will actually follow if we manage to prove} that $  \mathfrak b\in L^{\sc}(B_{\ell,\textcolor{black}{1}}^{ \gamma })$.
Indeed, $B_{\ell,\textcolor{black}{2}}^{\gamma} \hookrightarrow H^{\gamma,\ell}$ when $\ell \ge 2 $ \textcolor{black}{(see e.g. \cite[Th. 2.5.6 p.88]{Triebel-83a} and \cite[Rk .2 p.96]{Triebel-83b})} and we know from 
\eqref{BesovEmbedding} that for all $\gamma>0$,   
$B_{\ell,1}^{ \gamma } \hookrightarrow B_{\ell,\textcolor{black}{2}}^{ \gamma }$. 
\textcolor{black}{Observing as well that, in the case $\alpha=2$, since from \eqref{EMBEDDING},\eqref{BesovEmbedding}, $B_{\ell,1}^{ \gamma } \hookrightarrow B_{\ell,1}^{ 0 } \hookrightarrow L^\ell$, it suffices to prove that $  \mathfrak b\in L^{\sc}(B_{\ell,\textcolor{black}{1}}^{ \gamma })$ for $\alpha\in (1,2]$ to ensure that both conditions  $(i)$ and $(ii)$ hold}.

Write now  from \eqref{YOUNG} and for $\ell$ meant to be large (but finite\footnote{we restrict here to a finite $\ell $ to consider the cases $\alpha=2 $ and $\alpha\in (1,2) $ in a similar way. The case $\alpha=2 $ could also be handled more directly considering $ \ell=+\infty$.}):
\begin{align*}
|\mathfrak b(s,\cdot)|_{B_{\ell,1}^\gamma}\le |b(s,\cdot)|_{B_{p,q}^{\beta}}|\brho_{t,\mu}(s,\cdot)|_{B_{\textcolor{black}{\ell_2},q'}^{\gamma-\beta}},\text{ with } \textcolor{black}{\ell_2=p'\frac{1}{\frac{p'}{\ell}+1}}. 
\end{align*}
\textcolor{black}{Note that one can choose any $\ell_2<p'$ \textit{close} to $p'$, since again $ \ell$ is arbitrarily  large but finite}. 
Write from the H\"older inequality,
\begin{align}\label{CTR_NL_DRIFT_STABLE}
\int_t ^T{\color{black} ds \ }|\mathfrak b(s,\cdot)|_{B_{\ell,1}^\gamma}^{ \sc} 
\le \Big(\int_t^T {\color{black} ds \ }|b(s,\cdot)|_{B_{p,q}^{\beta}}^{ \sc a}
\Big)^{\frac 1a}\Big(\int_t^T{\color{black}ds \ }|\brho_{t,\mu}(s,\cdot)|_{B_{\textcolor{black}{\ell_2},q'}^{\gamma-\beta}}^{ \sc a'} 
\Big)^{\frac 1{a'}}, \ a^{-1}+(a')^{-1}=1.
\end{align}
We now claim that for all $\overline{\sc}$ lying in the interval in \eqref{cond_time_gencase},
\begin{eqnarray*}
 \forall \overline{\ell}<p',\, \forall \overline{\vartheta} \in [0,1),\, \exists \theta>0 \text{ s.t. } \int_t^T ds |\brho_{t,\mu}(s,\cdot)|_{B^{-\beta+ \textcolor{black}{\overline{\vartheta}} \Gamma}_{\overline{\ell},q'}}^{\overline{\sc}} \le C(T-t)^{\theta}. \qquad (\star)
\end{eqnarray*}
As $b\in L^{r}((t,T],B_{p,q}^\beta) $, choosing now $a$ such that $a \sc=r $ \textcolor{black}{which implies} $1/{a'}=1-1/{a}=1-{ \sc}/{r} \iff a'= r/({r- \sc})$, $\gamma=  \overline{\vartheta} \Gamma$ with $\overline{\vartheta} \in (0,1)$ and $\ell$ large enough, the claim \textcolor{black}{follows} provided the set of parameters $\vartheta,\, \ell$ and $\sc = \overline{\sc}/a'$ satisfying \textcolor{black}{\eqref{COND_II}} and the above condition is non-empty.

Firstly, the condition $\overline{\sc}=\sc a'$ \textcolor{black}{belongs} to the interval in \eqref{cond_time_gencase} writes
\begin{eqnarray*}
 \textcolor{black}{\sc a'=}\sc\frac{r}{r- \sc}<\left(\frac 1\alpha\Big[-\beta+\overline{\vartheta} \Gamma+\frac d{p}\Big]
\right)^{-1} &\iff &\left(\frac 1\alpha\Big[-\beta+\overline{\vartheta}\Gamma+\frac d{p}\Big]
\right)r<\frac{r}{ \sc}-1\\
&\iff& \sc<r\bigg\{1+ \left(\frac 1\alpha\Big[-\beta+\overline{\vartheta}\Gamma+\frac d{p}\Big]
\right)r\bigg\}^{-1}.
\end{eqnarray*}
Thus, if
\begin{equation}\label{C_STRONG_GEN_S}\tag{$S_I$}
\textcolor{black}{\frac{\alpha}{\alpha-1}}<r\bigg\{1+ \left(\frac 1\alpha\Big[-\beta+\overline{\vartheta}\Gamma+\frac d{p}\Big]r
\right)\bigg\}^{-1},
\end{equation}
there exists $\sc$ satisfying the two last condition in $(ii)$ and such that $\overline{\sc}=\sc a'$  lies in the interval in \eqref{cond_time_gencase}. Indeed, since $\ell$ can be taken arbitrarily large (but finite), the above condition  allows to obtain
$$\frac{\alpha}{\sc}+\frac{d}{\ell}<\textcolor{black}{\alpha-1}.$$
As \eqref{C_STRONG_GEN_S} equivalently rewrites:
$$r\left(1-\frac 1\alpha-\frac 1r\right)>\left[-\frac \beta \alpha+\frac{\overline{\vartheta}\Gamma}{\alpha}+\frac d{p\alpha}\right]r\iff \textcolor{black}{\beta >  1 -\alpha+ \frac dp +\frac \alpha r+\overline{\vartheta} \Gamma}
$$
and, from the condition \eqref{cond_gencase} and the definition of \eqref{GAP}, for all $\vartheta\in [0,1) $ it holds
$$\beta >  1 -\alpha+ \frac dp +\frac \alpha r+\vartheta \Gamma,$$
condition \eqref{C_STRONG_GEN_S} is always satisfied under our current assumptions.

Secondly, we need to check that
$$
\exists \overline{\vartheta} \in (0,1) \text{ such that } \overline{\vartheta}\Gamma = \gamma \in \left(1-\frac \alpha 2,1\right). 
$$
Under \eqref{COS} - which matches \eqref{cond_gencase} in the case $\alpha =2$ - we have
$$ \Gamma=\beta-\left(1-\alpha+\frac dp+\frac \alpha r\right)>1-\frac \alpha2\iff \beta>2-\frac 32 \alpha+\frac dp+\frac \alpha r,$$
so that $\overline{\vartheta} $ can be taken sufficiently close to $1$ in order to have $\gamma=\overline{\vartheta} \Gamma>1- \alpha/ 2 $. 

To conclude, it only remains to prove that $(\star)$ holds true. Observe that for $\alpha=2 $ and taking $\ell=+\infty,\ \ell_2=p' $ the control readily follows from the representation \eqref{main_DUHAMEL} in Proposition \ref{WP_FK}  and Lemmas \ref{lem_unifesti_gencase_2}-\ref{FIRST_STAB}. For $\alpha\in (1,2) $,
still starting from the representation \eqref{main_DUHAMEL} of $\brho_{t,\mu}$, one may thus restart the proof of Lemma \ref{lem_unifesti_gencase_2} up to \eqref{FOR_STRONG_UNIQUENESS} and choosing now therein $\ell_2 < p'$ for the free integrability parameter, $\gamma=-\beta+ \overline{\vartheta} \Gamma$. 
In such case, as $\ell_2 < p'$, we obtain a lower singularity (in time) for the estimate of the first term in the r.h.s. (see \eqref{SING_STABLE_HK}) in \eqref{FOR_STRONG_UNIQUENESS}. As the horizon time $T$ is assumed to be small ($T <1$) therein, the bounds remain valid and \textcolor{black}{the} estimate $(\star)$ follows.
}
\end{proof}

\appendix

\mysection{Proof of Proposition \ref{PROP_APPROX}}
We first focus on the spatial approximation. Consider a time homogeneous kernel $b\in B_{b,q}^{\beta}$, with $ \beta \in (-1,0]$, and define for $\varepsilon>0$, $b^\varepsilon(x)=\tilde P_\varepsilon^\alpha b(x)= {\color{black}\tilde p^\alpha(\varepsilon,\cdot)\star b(x)}$, \textcolor{black}{$\alpha\in (1,2] $ and consequently choose $\tilde \alpha=\alpha $ in \eqref{HEAT_CAR}}. We are going to prove that for all $ \tilde \beta<\beta$ then
\begin{equation}
\label{SPATIAL_BESOV}
\|b^\varepsilon-b\|_{B_{p,q}^{\tilde \beta}}\underset{\varepsilon\longrightarrow 0}{\longrightarrow} 0.
\end{equation}
To this aim, we first consider the thermic part $\mathcal T_{p,q}^{\tilde \beta}(b^\varepsilon-b)$ of $\|b^\varepsilon-b\|_{B_{p,q}^{\tilde \beta}}$. \textcolor{black}{From the embedding \eqref{BesovEmbedding} we can assume w.l.o.g. that $\tilde \beta-\beta+\alpha>0 $}
For the proof we temporarily \textcolor{black}{specify} the notation of the operator $\mathcal T_{p,q}^{\tilde \beta}$ into $\mathcal T_{p,q,n}^{\tilde \beta}$ making the index $n$ in \eqref{HEAT_CAR} now explicit.  
As $\tilde \beta< 0$, $n$ can be chosen freely over the set of non-negative integers, meanwhile, as $b$ is in $B_{b,q}^{\beta}$, $\mathcal T_{p,q,{m}}^{\beta}(b)<\infty$ for any integer ${m}>\beta$. In particular if $\beta<0$ then $m\ge 0$ and we can choose $m=n\ge 0$, while if $\beta=0$, it 
{\color{black}would be} enough to consider the case $n=0$, $m=1$. \textcolor{black}{But to keep the same parameters for all cases we first choose $n=m=1$. We first assume $q<\infty $. Then,}
\begin{align}
\Big(\mathcal T_{p,q,{1}}^{\tilde \beta}(b^\varepsilon-b)\Big)^q:=&\int_{0}^1 \frac{dv}{v}v^{({1}-\frac{\tilde \beta}{\alpha})q}\|\partial_v 
\tilde p^\alpha(v,\cdot)\star (b_\varepsilon-b)\|_{L^p}^q\notag\\
\le& 2^{q-1}\int_0^{2\varepsilon} \frac{dv}{v}v^{({1}-\frac{\tilde \beta}{\alpha})q} \Big(\|\partial_v
\tilde p^\alpha(v,\cdot)\star b_\varepsilon\|_{L^p}^q+\|\partial_v 
\tilde p^\alpha(v,\cdot)\star b\|_{L^p}^q \Big)\notag\\
&+\int_{2\varepsilon}^1 \frac{dv}{v}v^{({1}-\frac{\tilde \beta}{\alpha})q} \| (\partial_v 
\tilde p^\alpha(v+\varepsilon,\cdot)-\partial_v 
\tilde p^\alpha(v,\cdot))\star b\|_{L^p}^q=:\mathscr T_1+\mathscr T_2.\label{DECOUP_SEUIL_THERMIQUE_PART}
\end{align}
Since
$$ \|\partial_v
\tilde p^\alpha(v,\cdot)\star b_\varepsilon\|_{L^p}^q=\|\tilde p^\alpha(\varepsilon,\cdot)\star\partial_v 
\tilde p^\alpha(v,\cdot)\star b\|_{L^p}^q\le \|\partial_v
\tilde p^\alpha(v,\cdot)\star b\|_{L^p}^q,$$
 - such inequality following from the $L^1-L^p$ convolution inequality and recalling that $ \|\tilde p^\alpha(\varepsilon,\cdot)\|_{\color{black}L^1}=1$ - we derive {\color{black} for $\beta\in (-1,0]
 $}:
\begin{align}
\mathscr T_1\le& 2^{q}\int_0^{2\varepsilon} \frac{dv}{v}v^{({1}-\frac{\tilde \beta}{\alpha})q}\|\partial_v
\tilde p^\alpha(v,\cdot)\star b\|_{L^p}^q \le 2^{q} (2\varepsilon)^{\frac{-\tilde \beta+\beta}{\alpha}q}\int_0^{2\varepsilon} \frac{dv}{v}v^{({1}-\frac{ \beta}{\alpha})q}\|\partial_v
\tilde p^\alpha(v,\cdot)\star b\|_{L^p}^q\notag\\
\le & 2^{q} (2\varepsilon)^{\frac{-\tilde \beta+\beta}{\alpha}q} \big(\mathcal T_{p,q,{1}}^\beta(b)\big)^q\underset{\varepsilon \rightarrow 0}{\longrightarrow}0,\label{CTR_T1_APPR_THERM}
\end{align}
since $\tilde \beta<\beta $. 

Let us now turn to $\mathscr T_2 $ for which we get:
\begin{align}
\mathscr T_2\le& \int_{2\varepsilon}^1\frac{dv}{v}v^{({1}-\frac{\tilde \beta}{\alpha})q} \varepsilon^q\bigg|\int_0^1 d\lambda \partial_v^{{2}} \tilde p^\alpha(v+\varepsilon\lambda, \cdot)\star b\,\bigg|_{L^p}^q\notag\\
\le& \varepsilon^{\frac{-\tilde \beta+\beta}\alpha q}\int_{2\varepsilon}^1\frac{dv}{v}v^{({1}-\frac{\tilde \beta}{\alpha})q} \varepsilon^{q+\frac{\tilde \beta-\beta}\alpha q}\Big( \int_0^1 d\lambda \| \partial_v^2 \tilde p^\alpha(v+\varepsilon\lambda, \cdot)\star b\|_{L^p}\Big)^q\notag\\
\le & \varepsilon^{\frac{-\tilde \beta+\beta}\alpha q}\int_{2\varepsilon}^1\frac{dv}{v}v^{({2}-\frac{\beta}{\alpha})q}  \int_0^1 d\lambda \| \partial_v^{2} \tilde p^\alpha(v+\varepsilon\lambda, \cdot)\star b\|_{L^p}^q\notag\\
\le & \varepsilon^{\frac{-\tilde \beta+\beta}\alpha q}\int_0^1 d\lambda \int_{2\varepsilon}^1\frac{dv}{v}v^{({2}-\frac{\beta}{\alpha})q}  \| \partial_v^{{2}} \tilde p^\alpha(v+\varepsilon\lambda, \cdot)\star b\|_{L^p}^q\notag\\
\le & \varepsilon^{\frac{-\tilde \beta+\beta}\alpha q}\int_0^1 d\lambda \int_{\varepsilon(2+\lambda)}^{1+\lambda\varepsilon}\frac{d\tilde v}{(\tilde v-\lambda \varepsilon)}(\tilde v-\lambda \varepsilon)^{({2}-\frac{\beta}{\alpha})q}  \| \partial_{\tilde v}^{2} \tilde p^\alpha(\tilde v, \cdot)\star b\|_{L^p}^q\notag,
\end{align}
the last inequality following from the change of variables $\tilde v=v+\epsilon\lambda$.
Now, since $\tilde v\ge 2\varepsilon $, we have that for all $\lambda \in [0,1] $, $\tilde v-\lambda\varepsilon\ge \tilde v/2 $ so that:
\begin{align}
\mathscr T_2\le & 2\varepsilon^{\frac{-\tilde \beta+\beta}\alpha q}\int_0^1 d\lambda \int_{\varepsilon(2+\lambda)}^{1+\lambda\varepsilon}\frac{d\tilde v}{\tilde v}\tilde v^{({2}-\frac{\beta}{\alpha})q}  \| \partial_{\tilde v}^2 \tilde p(\tilde v, \cdot)\star b\|_{L^p}^q\notag\\
\le & 2\varepsilon^{\frac{-\tilde \beta+\beta}\alpha q} \Big[\big(\mathcal T_{p,q,2}^{\beta}(b)\big)^q+\int_0^1 d\lambda \int_{1}^{1+\lambda\varepsilon}\frac{d\tilde v}{\tilde v}\tilde v^{({2}-\frac{\beta}{\alpha})q}  \| \partial_{\tilde v}^{2} \tilde p(\tilde v, \cdot)\star b\|_{L^p}^q\Big]. \label{PROV_T2_THERM}
\end{align}
Recall now from \eqref{EMBEDDING} and \eqref{YOUNG} that:
\begin{align*}
\| \partial_{\tilde v}^{2} \tilde p(\tilde v, \cdot)\star b\|_{L^p}\le \| \partial_{\tilde v}^{2} \tilde p(\tilde v, \cdot)\star b\|_{B_{p,1}^0}\le {\color{black}\cv}\|b\|_{B_{p,q}^\beta}\|\partial_{\tilde v}^{2} \tilde p(\tilde v, \cdot)\|_{B_{1,1}^{-\beta}}\le C\|b\|_{B_{p,q}^\beta}\tilde v^{-{2}-\frac{\beta}{\alpha}},
\end{align*}
which plugged into \eqref{PROV_T2_THERM} yields:
\begin{align*}
\mathscr T_2\le & 2\varepsilon^{\frac{-\tilde \beta+\beta}\alpha q} \Big[\big(\mathcal T_{p,q,{2}}^{\beta}(b)\big)^q+ C \varepsilon \|b\|_{B_{p,q}^\beta}^q\Big].
\end{align*}
This,  together with \eqref{CTR_T1_APPR_THERM}, since $\mathcal T_{p,q,{2}}^{\beta}(b)<+\infty $, finally yields:
\begin{equation}\label{FINAL_THERMIC_PART}
 \Big(\mathcal T_{p,q,{1}}^{\tilde \beta}(b^\varepsilon-b)\Big) \underset{\varepsilon \rightarrow 0 }{\longrightarrow} 0.
 \end{equation}
It is easily seen that the previous arguments can be adapted to the case $q=+\infty $ , by modifying the decomposition \eqref{DECOUP_SEUIL_THERMIQUE_PART} accordingly to the splitting 
\begin{align*}
\mathcal T_{p,\infty,1}^{\tilde \beta}(b_\varepsilon-b)&:=\sup_{v \in (0,1]}v^{({1}-\frac{\tilde \beta}{\alpha})}\|\partial_v  
\tilde p^\alpha(v,\cdot)\star (b_\varepsilon-b)\|_{L^p}\\
&\le\sup_{v \in (0,2\epsilon]} v^{({1}-\frac{\tilde \beta}{\alpha})}\|\partial_v  
\tilde p^\alpha(v,\cdot)\star (b_\varepsilon-b)\|_{L^p}+\sup_{v \in (2\epsilon,1]}v^{({1}-\frac{\tilde \beta}{\alpha})}\|\partial_v  
\tilde p^\alpha(v,\cdot)\star (b_\varepsilon-b)\|_{L^p}.
\end{align*}

Also, the last argument can be adapted to control the non-thermic part of the Besov norm. Namely,
\begin{align} \label{NON_THERMIC_PART}
\|\hat \phi \star (b^\varepsilon-b)\|_{L^p}=\| (\tilde p^\alpha(\varepsilon,\cdot) \star\hat \phi-\hat \phi)\star b\|_{L^p}\le C\|b\|_{B_{p,q}^\beta}\|\tilde p^\alpha(\varepsilon,\cdot) \star\hat \phi-\hat \phi\|_{B_{1,1}^{-\beta}}\underset{\varepsilon \rightarrow 0 }{\longrightarrow} 0
\end{align}
from the smoothness of $\hat \phi $ and the continuity of the shift operator in $L^p$. The statement \eqref{SPATIAL_BESOV} eventually follows from \eqref{NON_THERMIC_PART} and \eqref{FINAL_THERMIC_PART}. {\color{black}Moreover, as the by-product of the above arguments, 
$$|b^{\textcolor{black}{\varepsilon}}|_{B^{\beta}_{p,q}}\textcolor{black}{\underset{\eqref{YOUNG}}{\le} \cv|\tilde p^\alpha(\varepsilon,\cdot)|_{B_{1,\infty}^0}|b|_{B^{\beta}_{p,q}}\underset{\eqref{EMBEDDING}}{\le}} \cc|\tilde p^\alpha(\varepsilon,\cdot)|_{L^1}|b|_{B^{\beta}_{p,q}}=\cc|b|_{B^{\beta}_{p,q}},$$ 
for all $\varepsilon>0$.}

Consider now a time dependent drift i.e. $b\in L^r((t,T], B_{p,q}^\beta) $. Then for almost any $s\in (t,T]$, $b(s,\cdot)\in  B_{p,q}^\beta$. Thus, setting $b_{{\rm sp}}^{\varepsilon}(s,\cdot)=\tilde P_\varepsilon^\alpha b(s,\cdot)  $, we 
{\color{black}have}
$$\|b_{{\rm sp}}^\varepsilon(s,\cdot)-b(s,\cdot)\|_{B_{p,q}^{\tilde \beta}}\underset{\varepsilon \rightarrow 0}{\longrightarrow}0 .$$
If $r=+\infty$, {\color{black}the uniform control of $b_{{\rm sp}}^{\varepsilon}(s,\cdot)_{B^{\tilde \beta}_{p,q}}$} readily yield{\color{black}s}
$$\|b_{{\rm sp}}^\varepsilon-b\|_{L^\infty ((t,T],B_{p,q}^{\tilde \beta})}\underset{\varepsilon \rightarrow 0}{\longrightarrow}0  .$$

If now $r<+\infty$, since from the previous arguments {\color{black} and from \eqref{YOUNG} and \eqref{BesovEmbedding},}
$$\|b_{{\rm sp}}^\varepsilon(s,\cdot)-b(s,\cdot)\|_{B_{p,q}^{\tilde \beta}}\|\le \|b_{{\rm sp}}^\varepsilon(s,\cdot)\|_{B_{p,q}^{\tilde \beta}}+\|b(s,\cdot)\|_{B_{p,q}^{\tilde \beta}} \le \textcolor{black}{(1+\cv)}\|b(s,\cdot)\|_{B_{p,q}^{\tilde \beta}}\le {\color{black}(1+\cv)}\|b(s,\cdot)\|_{B_{p,q}^{ \beta}}\in L^r((t,T]),$$
we get from the bounded convergence theorem that 
$$\int_{t}^T{\color{black}ds \ }\|b_{{\rm sp}}^\varepsilon(s,\cdot)-b(s,\cdot)\|_{B_{p,q}^{\tilde \beta}}^r 
\underset{\varepsilon \rightarrow 0}{\longrightarrow}0.$$

Introduce $ b^\varepsilon(s,\cdot)=(b_{{\rm sp}}^\varepsilon(\cdot,\cdot)* \eta_\varepsilon)(s)$ for a mollifier $\eta_\varepsilon(\cdot)=\varepsilon^{-{\color{black}1}}\eta(\cdot/\varepsilon) $ for some smooth compactly supported $\eta$, where $*$ stands for the convolution in the time argument. Considering this additional convolution in time then yields the statement. {\color{black}The uniform control of $|b^\varepsilon|_{L^{\bar r}(B^{\beta}_{p,q})}$ \textcolor{black}{is then clear}.}

\mysection{Proof of Lemma \ref{LEM_DUHAMEL_MOLL}}
\label{PROOF_LEM_DUHAMEL_MOLL}
{\color{black}
 Given $s$ in $(t,T]$ and $f$ a bounded and smooth function on $\mathbb R^d$, with bounded derivatives, the function $\phi(v,x):=p^\alpha_{s-v}\star f(x)$ is itself a smooth solution of the backward equation:
\begin{equation}\label{CORE_BACKWARD}
\left\{
\begin{aligned}
&\partial_v\phi(v,x)+L^\alpha(\phi)(v,x)=0\,\text{in}\,[t,s)\times\mathbb R^d,\\
&\phi(s,\cdot)=f(\cdot).
\end{aligned}
\right.	
\end{equation} 
It\^o's formula applied to $\phi(v,X^{\varepsilon,t,\mu}_v)$ over the interval $[t,s]$ yields, \textcolor{black}{in view of \eqref{CORE_BACKWARD}},
\begin{eqnarray*}
f(X^{\varepsilon,t,\mu}_s) =\phi(s,X^{\varepsilon,t,\mu}_s)
= \phi(t,\xi)+\int_t^s\mathcal B_{\brho^{\varepsilon}_{t,\mu}}^{\varepsilon}(v,X^{\varepsilon,t,\mu}_v)\cdot\nabla \phi(v,X^{\varepsilon,t,\mu}_v)\,dv+\textcolor{black}{M_s^{\varepsilon,t,\mu}},
\end{eqnarray*}
where
\begin{equation*}
M_s^{\varepsilon,t,\mu}=
\left\{
\begin{aligned}
&\int_t^s \nabla\phi(v,X_v^{\varepsilon,t,\mu}) \cdot dW_v,\ \text{if} \ \alpha=2,\\ 
&\int_t^s \int_{\R^d\backslash\{0\}} [\phi(v,X_{v^-}^{\varepsilon,t,\mu} +x) -\phi(v,X_{v^-}^{\varepsilon,t,\mu})] \tilde N(dv,dx), \ \text{if} \ \alpha \in (1,2),
\end{aligned}
\right.
\end{equation*}
\noindent
{\color{black}denoting by $W$, the standard $d$-dimensional Brownian motion and by $\tilde N $, the compensated Poisson measure}.

Averaging the above expression further yields  
\begin{equation}\label{PDE_EPS_VAR_BIS}
\int dx \ f(x)\brho^{\epsilon}_{t,\mu}(s,x)
=\int p^\alpha_{s-t}\star f(y)\,\mu(dy)+\int_t^s dv \int_{\R^d} dy\,  \{ \mathcal B_{\brho^{\varepsilon}_{t,\mu}}^{\varepsilon}(v,y) \brho^{\varepsilon}_{t,\mu}(v,y)\}\cdot \nabla p^\alpha_{s-v}\star f(y).
\end{equation}
By Fubini's theorem, the convolution with  $p^\alpha_{s-t}$ can be shifted to $\mu$ and $\int p^\alpha_{s-t}\star f(y)\,\mu(dy)$ rewrites as 
$\int  dx\ f(x) \,p^\alpha_{s-t}\star\,\mu(x)$. In the same way, observe that, since  $|\brho^{\varepsilon}_{t,\mu}(v,\cdot))|_{L^1}=1$ for all $v>t$, 
\begin{eqnarray*}
&&\int_t^s\,dv\int \,dy\int\,d\textcolor{black}{x} \ |f(\textcolor{black}{x})||\nabla p^\alpha_{s-v}(y-\textcolor{black}{x})||\mathcal B_{\brho^{\varepsilon}_{t,\mu}}^{\varepsilon}(v,y)| |\brho^{\varepsilon}_{t,\mu}(v,y)|\\
&\le & |f|_{L^\infty}\int_t^s\,dv \ |\nabla p^\alpha_{s-v}|_{L^1}|\mathcal B_{\brho^{\varepsilon}_{t,\mu}}^{\varepsilon}(v,\cdot)|_{L^\infty}| \brho^{\varepsilon}_{t,\mu}(v,\cdot)|_{L^1}=|f|_{L^\infty}\int_t^s\,dv|\nabla p^\alpha_{s-v}|_{L^1}|\mathcal B_{\brho^{\varepsilon}_{t,\mu}}^{\varepsilon}(v,\cdot)|_{L^\infty}. 
\end{eqnarray*}
Using \eqref{YOUNG} and again that $|\brho^{\varepsilon}_{t,\mu}(v,\cdot))|_{L^1}=1$,  it follows that for all $v$, 
\begin{eqnarray*}
|\mathcal B_{\brho^{\varepsilon}_{t,\mu}}^{\varepsilon}(v,\cdot)|_{L^\infty}=|(b^\varepsilon(v,\cdot)\star \brho^{\varepsilon}_{t,\mu}(v,\cdot)|_{L^\infty}
\le &  |b^\varepsilon|_{L^\infty(L^\infty)}.
\end{eqnarray*}
Therefore, applying the embedding $|\nabla p^\alpha_{s-v}|_{L^1}\le C|\nabla p^\alpha_{s-v}|_{B^0_{1,1}}$ (following \eqref{BesovEmbedding}) and next \eqref{SING_STABLE_HK} for $\gamma=0$, $\ell=1=m$,  
\begin{eqnarray*}
&&\int_t^s\,dv\int \,dy\int\,d\textcolor{black}{x} |f(\textcolor{black}{x})||\nabla p^\alpha_{s-v}(y-\textcolor{black}{x})||\mathcal B_{\brho^{\varepsilon}_{t,\mu}}^{\varepsilon}(v,y)|| \brho^{\varepsilon}_{t,\mu}(v,y)|\\
&\le & C  |f|_{L^\infty}|b^\varepsilon|_{L^\infty(L^\infty)}\int_t^s\,\frac{dv}{(s-v)^{\frac{1}{\alpha}}}<\infty,
\end{eqnarray*}
the finiteness being granted by the fact that $\alpha>1$. Consequently,  \eqref{PDE_EPS_VAR_BIS} rewrites as
\[
\int dx \ f(x)\brho^{\varepsilon}_{t,\mu}(s,x) 
=\int dx\, f(x)\Bigg[ p^\alpha_{s-t}\star \,\mu(x)-\int_t^s dv\, \Big(\nabla p^\alpha_{s-v}\star \{ \mathcal B_{\brho^{\varepsilon}_{t,\mu}}^{\varepsilon}(v,x) \brho^{\varepsilon}_{t,\mu}(v,x)\}\Big)\Bigg].
\]
By a density argument, we can extend \textcolor{black}{the previous computations to} the class of functions $f$ \textcolor{black}{that are} only bounded measurable functions and we eventually deduce that, for all $t\le s\le T$ and for a.e. $x$ in $\mathbb R^d$, 
\[
\brho^{\textcolor{black}{\varepsilon}}_{t,\mu}(s,x)= p^\alpha_{s-t}\star \,\mu(x)-\int_t^s dv\, \Big(\nabla p^\alpha_{s-v}\star \{ \mathcal B_{\brho^{\varepsilon}_{t,\mu}}^{\varepsilon}(v,x) \brho^{\varepsilon}_{t,\mu}(v,x)\}\Big).
\]
By continuity of $y\mapsto \brho^{\varepsilon}_{t,\mu}(s,y)$, the equality extends to all point $y$ of $\mathbb R^d.$
}

\label{ADDENDUM_TO_LEMME_STAB}
\textbf{Acknowledgements.} For the first and third authors \textcolor{black}{this work has partially been supported by  the Russian Science Foundation project (project ${\rm N}^\circ$20-11-20119)}. For the second author, the article was prepared within the framework of the Basic Research Program at HSE University. {\color{black}This joint work and its companion \cite{chau:jabi:meno:22-2} were both initiated in May 2021}. The first author also thanks the Centre Henri Lebesgue ANR-11-LABX-0020-01 for creating an attractive mathematical environment.

\end{document}